\documentclass[10pt,letterpaper,twoside]{article}
\usepackage[margin=1in, marginparwidth=17mm]{geometry}
\usepackage{amsmath,amsthm,amssymb}
\usepackage{mathtools}
\usepackage[utf8]{inputenc}
\usepackage[T1]{fontenc}
\usepackage{amsfonts}
\usepackage{graphicx}
\usepackage{lmodern}
\usepackage{enumerate}
\usepackage{hyperref}
\usepackage{booktabs}
\usepackage{float}
\usepackage{fancyhdr}
\usepackage{bold-extra}
\usepackage{tabularx}
\usepackage{mathrsfs}
\usepackage{bm} 
\usepackage[normalem]{ulem} 
\usepackage[dvipsnames]{xcolor}
\usepackage{extarrows} 
\usepackage{empheq}
\usepackage{calligra}
\usepackage{caption}

\def\abs#1{\left| #1 \right|}
\newcommand{\norm}[1]{\left\lVert#1\right\rVert}

\newcommand{\inparen}[1]{\left(#1\right)}             
\newcommand{\inbraces}[1]{\left\{#1\right\}}           
\newcommand{\insquare}[1]{\left[#1\right]}             
\newcommand{\inangle}[1]{\left\langle#1\right\rangle} 

\let\tilde\widetilde
\let\bar\overline

\newcommand{\parag}[1]{ \paragraph{\bf #1}}

\newcommand{\ud}{\mathrm{d}}
\DeclareRobustCommand{\rchi}{{\mathpalette\irchi\relax}} 
\newcommand{\irchi}[2]{\raisebox{\depth}{$#1\chi$}} 

\DeclareMathOperator{\Var}{\rm Var}

\newcommand{\Tr}{\mathop{\mathrm{Tr}}}

\def\vec#1{{\boldsymbol{#1}}}
\def\mat#1{{\boldsymbol{#1}}}

\DeclarePairedDelimiterX{\expectarg}[1]{[}{]}{%
  \ifnum\currentgrouptype=16 \else\begingroup\fi
  \activatebar#1
  \ifnum\currentgrouptype=16 \else\endgroup\fi
}

\newcommand{\innermid}{\nonscript\;\delimsize\vert\nonscript\;}
\newcommand{\activatebar}{%
  \begingroup\lccode`\~=`\|
  \lowercase{\endgroup\let~}\innermid 
  \mathcode`|=\string"8000
}

\newcommand{\bx}{\bm{x}}
\newcommand{\by}{\bm{y}}

\newcommand{\Eb}{\mathbf{E}}

\newcommand{\Ub}{\mathbf{U}}

\newcommand{\cE}{\mathcal{E}}
\newcommand{\cF}{\mathcal{F}}

\newcommand{\cL}{\mathcal{L}}

\newcommand{\cN}{\mathcal{N}}

\newcommand{\cP}{\mathcal{P}}

\newcommand{\cR}{\mathcal{R}}

\newcommand{\cX}{\mathcal{X}}

\newcommand{\bxi}{\bm{\xi}}

\newcommand{\GG}{\mathbb{G}}
\newcommand{\HH}{\mathbb{H}}

\newcommand{\NN}{\mathbb{N}}

\newcommand{\RR}{\mathbb{R}}
\usepackage{algorithm}
\usepackage{algpseudocode}
\usepackage{algorithmicx}

\algnewcommand\algorithmicphase{\textbf{Phase}}
\algnewcommand\PHASE{\item[\algorithmicphase]}


\DeclarePairedDelimiterX{\infdivx}[2]{(}{)}{%
  #1\;\delimsize\|\;#2%
}

\newtheorem{theorem}{Theorem}[section]
\newtheorem{corollary}[theorem]{Corollary}
\newtheorem{lemma}[theorem]{Lemma}

\newtheorem{proposition}[theorem]{Proposition}

\theoremstyle{definition}
\newtheorem{definition}{Definition}[section]

\theoremstyle{remark}
\newtheorem*{remark}{Remark}

\newenvironment{proofof}[1]{\begin{trivlist} \item {\textbf{{Proof of
#1}.~~}}}
  {\qed\end{trivlist}}

\newcommand\nnfootnote[1]{%
  \begin{NoHyper}
  \renewcommand\thefootnote{}\footnote{#1}%
  \addtocounter{footnote}{-1}%
  \end{NoHyper}
}

\usepackage{pgf,tikz,pgfplots}
\pgfplotsset{compat=1.15}
\usepackage{mathrsfs}
\usetikzlibrary{arrows, decorations.pathreplacing, calligraphy}
\definecolor{ffffff}{rgb}{1,1,1}

\newcommand{\footremember}[2]{%
   \footnote{#2}
    \newcounter{#1}
    \setcounter{#1}{\value{footnote}}%
}
\newcommand{\footrecall}[1]{%
    \footnotemark[\value{#1}]%
}

\usepackage{tikz-cd}

\begin{document}

\title{Local independence in mean-field spin glasses}
\author{Timothy L.H. Wee\footremember{fn:author}{Department of Statistics and Data Science, Yale University. Email: timothy.wee@yale.edu, sekhar.tatikonda@yale.edu } \and Sekhar Tatikonda\footrecall{fn:author} }
\date{\today}

\maketitle

\begin{abstract}
    We present a new approach to local independence in spin glasses, i.e.~the phenomenon that any fixed subset of coordinates is asymptotically independent in the thermodynamic limit. The approach generalizes the rigorous cavity method from Talagrand by considering multiple cavity sites. Under replica-symmetric conditions of thin-shell and overlap concentration, the cavity fields are revealed to be asymptotically independent, conditionally on the disorder, which in turn leads to local independence. Conversely, it is shown that local independence implies those replica-symmetric properties. The framework is general enough to encompass the classical and soft spin ($[-1,1]$) Sherrington-Kirkpatrick models, as well as the Gardner spin glasses. \nnfootnote{\emph{{MSC2020 subject classifications.} 60K35, 60F05, 82B44}} \nnfootnote{\emph{Key words and phrases.} Mean-field spin glass, asymptotic independence, overlap concentration, random projections, Sherrington-Kirkpatrick, Perceptron, Shcherbina-Tirozzi, Gardner.}
\end{abstract}

\thispagestyle{empty}
\tableofcontents

\section{Introduction}

A probability measure is locally independent if a subset of its coordinates is asymptotically independent as the model dimensions grow to infinity. This phenomenon has been observed in various distributions with non-trivial dependencies in the form of Gibbs measures in statistical physics, and as high-dimensional posterior distributions in statistics, computer science, and engineering.

In mean-field spin glasses, local independence is associated with the \emph{replica-symmetric} region, which is where the free energy admits a single-letter characterization \cite{toninelli2002almeida}, and where it is generally believed that there is overlap concentration (e.g.~Almeida-Thouless conjecture (\cite{talagrand2010mean} Section 1.8), and validity of corrected mean-field variational principles or iterative procedures for computation of Gibbs observables (e.g.~TAP equations (\cite{talagrand2010mean}, Section 1.7) and convergence of AMP algorithms in Sherrington-Kirkpatrick model (\cite{bolthausen2014iterative}, Theorem 2.1)). Similarly, in related topics such as compressed sensing, random linear estimation, and matrix recovery that arise in information theory and statistical inference and estimation, local independence is associated with \emph{decoupling principles} and the resulting single-letter characterizations of mutual information and posterior means, which are the equivalent counterparts to free energy and magnetization (see e.g.~\cite{miolane2019fundamental}, \cite{reeves2016replica}, \cite{guo2005randomly}, \cite{korada2010tight}).

The study of local independence is therefore tantalizing as there remain important open questions about the replica-symmetric region, including its characterizing properties and boundaries. An improved understanding of this region is likely to reveal important structural information about Gibbs measures, with applications to understanding phase transitions in statistical and information-theoretic problems such as all-or-nothing threshold phenomena, and computational-statistical gaps. 


This paper presents a general framework for attaining local independence in mean-field spin glasses. The approach involves a generalization of Talagrand's cavity method (\cite{talagrand2010mean}, Section 1.6). We consider the joint law of the cavity fields around the spin sites that we wish to assert are asymptotically independent. Under conditions typically associated with the replica-symmetric phase, we show that joint law is in fact approximately independent Gaussian, which eventually leads to local independence. The elucidation of the joint cavities law is a consequence of a general random projections result that leverages the fact that most random low-dimensional projections of high-dimensional distributions with well-behaved geometry are Gaussian \cite{wee2023random}.

Furthermore, a converse is provided, showing that local independence yields thin-shell and overlap concentration---properties typically associated with replica-symmetry. These results reinforce and add to a general characterization of the replica-symmetric phase (see Figure \ref{fig:RS_equivalences} for a high-level depiction).

\subsection{Main results}

Consider a general setting where $-H_{N} : \RR^N \rightarrow \RR$ is a Hamiltonian associated to a Gibbs measure $\inangle{\cdot}$, i.e., for any integrable $f : \RR^N \rightarrow \RR$,
\begin{align}
    \inangle{f} = \frac{1}{Z_N} \int_{\RR^N} f(\vec{x}) \exp\inparen{-H_{N}(\vec{x})} \mu^{\otimes N}(\ud \vec{x}),
    \label{eq:intro_GibbsMeasure_definition}
\end{align}
where $\mu^{\otimes N}$ is some (i.i.d.) product probability measure on $\RR^N$, and where $Z_N := \int_{\RR^N} \exp\inparen{-H_{N}(\vec{x})} \mu^{\otimes N}(\ud \vec{x})$ is the partition function. We are concerned with disordered systems, that is, when $-H_N$ is a random variable, so that $\inangle{\cdot}$ is a random measure, with randomness referred to as the disorder. Write $\Eb$ for an average over the disorder. Throughout this paper, we always assume \textbf{exchangeability among sites} holds under $\Eb$, i.e.~that $\Eb \inangle{\cdot}$ is an exchangeable probability measure (see \eqref{eq:exchangeability_among_sites}).

For integer $k \leq N$, let $G_N^{(k)}$ denote the Gibbs marginal on the first $k$-coordinates of $\inangle{\cdot}$---that is $G_N^{(k)}\!\insquare{B} = \inangle{\inbraces{ \vec{x} \in \RR^N : (x_1,\dots,x_k) \in B  }}$. We say that the Gibbs measure $\inangle{\cdot}$ is \emph{locally independent} if there exists probability measures $\inparen{ \GG_j }_{j \leq k}$ on $\RR$ so that for integer $p \geq 1$,
\begin{align}
    \Eb\insquare{ \sup_{B \subseteq \RR^k} \inparen{ G_N^{(k)}\!\insquare{B} - \inparen{ \bigotimes_{j \leq k} \GG_j }\insquare{B}  } ^{2p} } = o(1) \quad \textnormal{as } N \rightarrow \infty,
    \label{eq:intro_LI}
\end{align}
where the supremum ranges over measurable sets.
Definition \eqref{eq:intro_LI} is a \emph{quenched} statement---it says that the law of the first $k$-marginals of $\inangle{\cdot}$ is close to a product measure, for typical realizations of the disorder. Note that by exchangeability among sites under $\Eb$, the statement holds equivalently for any $k$-marginal of $\inangle{\cdot}$.



We summarize the main contributions as follows.
\begin{itemize}
    \item Theorem \ref{thm:LI_abstract_theorem} provides general conditions on the Gibbs measure \eqref{eq:intro_GibbsMeasure_definition} for local independence \eqref{eq:intro_LI} to hold. We restrict ourselves to the Hamiltonians whose `cavity fields' around each site can be isolated (Definition \ref{definition:LI_abstract_deecomposableHam}). If the corresponding Gibbs measure satisfies thin-shell ($N^{-1}\norm{\vec{x}}^2$ concentrates) and overlap concentration ($N^{-1}\vec{x}^1 \boldsymbol{\cdot} \vec{x}^2$ concentrates), then \eqref{eq:intro_LI} holds. 
    
    In Theorem \ref{thm:LI=>TSOC_abstract} it is shown that for such classes of Hamiltonians, these conditions are in a sense necessary and sufficient (see points below).

    The result is abstract enough to capture the classical and $\insquare{-1,1}$-Sherrington-Kirkpatrick models, as well as the Gardner spin glasses---e.g.~the Ising Perceptron and Shcherbina-Tirozzi models.
    \item The approach in Theorem \ref{thm:LI_abstract_theorem} is novel and differs from existing approaches to local independence. The proof generalizes the rigorous cavity method in \cite{talagrand2010mean} Section 1.6 by considering $k$ cavity sites. The joint distribution of the $k$ cavity fields is characterized by the random projections result Theorem \ref{thm:ProjResult_DisorderedCase_supOverBL_LM_2p} from \cite{wee2023random}. In particular, conditioned on the disorder, the cavity fields are asymptotically independent, and this leads to the asymptotic independence among the $k$ coordinates.

    
    \item Theorem \ref{thm:LI_abstract_theorem} gives explicit representations of the product measure that approximates the $k$-marginals, and also non-asymptotic convergence rates. In particular, the approximate marginals depend on the distribution of the cavity fields which involve projections of certain Gibbs expectations in random directions. These random projections in turn have limiting distributions. Consequently, two forms of the approximate marginals are given: \eqref{eq:LI_abstract_GGj_statement} and \eqref{eq:LI_abstract_HHj_statement}. The latter reveals useful information: that the approximate marginals are both conditionally and \emph{unconditionally} independent wrt.~disorder.
    

    \item Theorem \ref{thm:LI=>TSOC_abstract} provides a converse to the Theorem \ref{thm:LI_abstract_theorem} in that if \eqref{eq:intro_LI} holds for $\bigotimes_j \GG_j$ that is both conditionally and unconditionally independent wrt.~disorder (as in \eqref{eq:LI_abstract_HHj_statement}), then it holds that the Gibbs measure has norm concentration (thin-shell) and overlap concentration. 

    \item Combined with the random projections results in \cite{wee2023random} (one direction is copied in Theorem \ref{thm:ProjResult_DisorderedCase_supOverBL_LM_2p}), the results of this paper suggest a characterization Figure \ref{fig:RS_equivalences} of the replica-symmetric phase; or equivalently a `pure state'. 
    
    A bi-directional relation between overlap concentration and local independence is already understood in the classical SK setting (\cite{talagrand2010mean} Theorem 1.4.15 and Theorem 1.7.1 and Exercise 1.7.2), with one-direction using SK specific arguments. Figure \ref{fig:RS_equivalences} illustrates that this phenomenon holds much more generally, extends beyond the $\inbraces{\pm 1}$ setting (where thin-shell is trivial), and adds another layer to the replica-symmetric characterization in terms of random projections of Gibbs measures. 
\end{itemize}

\begin{figure}
\centering
\begin{tikzpicture}[ampersand replacement=\&]
  \matrix (m) [matrix of math nodes,row sep=4em,column sep=9em,minimum width=2em]
  {
    \begin{array}{@{}c@{}}
        \textnormal{Thin-shell} \, +  \\
        \textnormal{Overlap concentration}  
    \end{array}
      \&     \begin{array}{@{}c@{}}
        \textnormal{Random, cond.~independent}   \\
        \textnormal{Gaussian projections}  
    \end{array}
     \\
     \& \begin{array}{@{}c@{}}
        \textnormal{Local independence for}  \\
        \textnormal{MF spin glasses}  
    \end{array}
     \\ };
  \path[-stealth]
    (m-1-1) edge [implies-implies,double equal sign distance] node [above] {\cite{wee2023random}}  (m-1-2)
    (m-1-2) edge [-implies,double equal sign distance] node [right] {Theorem \ref{thm:LI_abstract_theorem}} (m-2-2)
    (m-2-2) edge [-implies,double equal sign distance] node [left,yshift=-0.5em] {Theorem \ref{thm:LI=>TSOC_abstract}} (m-1-1);
\end{tikzpicture} 
\caption{}
\label{fig:RS_equivalences}
\end{figure}

\subsection{Outline of abstract local independence result: Theorem \ref{thm:LI_abstract_theorem}.}
\label{sec:intro_outline_LI}

Our approach is most easily illustrated with the classical SK model when $k = 2$. The classical SK model has Hamiltonian
\begin{align}
    -H_{N}(\vec{x}) &= \frac{\beta}{\sqrt{N}}\sum_{i < j \leq N} g_{ij} x_i x_j + h\sum_{i \leq N} x_i,
    \label{eq:intro_SK_Hamiltonian}
\end{align}
with $\beta > 0$ being the inverse temperature, and $(g_{ij})_{i < j \leq N}$ independent standard Gaussian r.v.'s, and $h \in \RR$ is an external field. The reference measure $\mu$ is taken to be the uniform probability on $\inbraces{\pm 1}$.

We may write 
\begin{align}
    -H_N(\vec{x}) &= -H_{N-2}^-(\vec{y}) + x_1 \inparen{ \beta\vec{A}_1^{\top} \vec{y} + h  } + x_2 \inparen{ \beta\vec{A}_2^{\top} \vec{y} + h  } + x_1 x_2 \frac{\beta}{\sqrt{N}} g_{12},
    \label{eq:intro_SK_Ham_rewrite}
\end{align}
where we have introduced the shorthands $\vec{A}_j := \inparen{ \frac{1}{\sqrt{N}}g_{j,i} }_{3 \leq i \leq N}$, and $\vec{y} := (x_3,\dots,x_N)$, and where
\begin{align}
    -H_{N-2}^-(\vec{y}) := \frac{\beta}{\sqrt{N}} \sum_{3 \leq i < j \leq N} g_{ij} x_i x_j + h \sum_{3 \leq i \leq N} x_i,
    \label{eq:intro_SK_Ham_truncated}
\end{align}
is a Hamiltonian on a $N-2$ system with Gibbs expectations denoted by $\inangle{\cdot}^-$. 

We expect the last term in \eqref{eq:intro_SK_Ham_rewrite} to be small, and in what follows we will drop it, believing that the Gibbs measure should not change much (this is made precise in Theorem \ref{thm:SK_classical_decomposition}).

Fix a configuration $\sigma_1,\sigma_2 \in \inbraces{\pm 1}$. The Gibbs probability mass function $G_N^{(2)} : \inbraces{\pm 1}^2 \rightarrow \insquare{0,1}$ on $x_1$ and $x_2$ is then approximately
\begin{align}
    G_{N}^{(2)}\inbraces{\sigma_1, \sigma_2} &\propto \sum_{\vec{y}} \exp\inparen{ \sigma_1 \inparen{ \beta\vec{A}_1^{\top} \vec{y} + h  } + \sigma_2 \inparen{ \beta\vec{A}_2^{\top} \vec{y} + h  }  } \exp \inparen{ -H_{N-2}^-(\vec{y})  }   \nonumber\\
    &\propto  \inangle{ \exp\inparen{ \sigma_1 \inparen{ \beta\vec{A}_1^{\top} \vec{y} + h  } + \sigma_2 \inparen{ \beta\vec{A}_2^{\top} \vec{y} + h  }  } }^-,
    \label{eq:intro_SK_2dimMarginal_firstExpansion}
\end{align}
where we got the second line by dividing top and bottom by $Z_{N-2}^-$: the partition function for $-H_{N-2}^-$. At this point, we are still confronted with a high-dimensional, intractable integral over $\vec{y}$. However, a random projections result Theorem \ref{thm:ProjResult_DisorderedCase_supOverBL_LM_2p} is now critical in allowing us to replace the expectation over $\vec{y}$ under $\inangle{\cdot}^-$ by a much simpler integral. 

First observe that the disorder in $\inangle{\cdot}^-$, i.e.~the disorder r.v.'s $(g_{ij})_{3 \leq i < j \leq N}$, is independent of $\vec{A}_1$ and $\vec{A}_2$, so that under $\inangle{\cdot}^-$, $\vec{y}$ is independent of $\vec{A}_j$'s. Next, note that under $\inangle{\cdot}^-$, $(N-2)^{-1}\norm{\vec{y}}^2 \equiv 1$ (thin-shell; because $y_i^2 \equiv 1$), and for high enough temperature, $(N-2)^{-1}\vec{y}^1 \boldsymbol{\cdot} \vec{y}^2 \simeq q$ (overlap concentration), for some number $0 \leq q < 1$. The projection result Theorem \ref{thm:ProjResult_DisorderedCase_supOverBL_LM_2p} then says that for large $N$:
\begin{align}
    \cL\inparen{ \begin{bmatrix} 
    \vec{A}_1^{\top} \vec{y} \\
    \vec{A}_2^{\top} \vec{y}
    \end{bmatrix} \,\bigg\rvert \textnormal{ disorder} } \simeq \cL\inparen{ \begin{bmatrix} 
    \vec{A}_1^{\top} \inangle{\vec{y}}^- + \sqrt{1 - q}\xi_1 \\
    \vec{A}_2^{\top} \inangle{\vec{y}}^- + \sqrt{1 - q} \xi_2
    \end{bmatrix} \,\bigg\rvert  \textnormal{ disorder} },
    \label{eq:intro_SK_projResult}
\end{align}
where $\xi_1, \xi_2$ are independent standard Gaussians, independent of all other sources of randomness. In other words, expectation over the $\xi_j$'s is an avatar for the expectation over $\vec{y}$ under $\inangle{\cdot}^-$. Crucially, the $\xi_j$'s are independent, and we get from \eqref{eq:intro_SK_2dimMarginal_firstExpansion} that, with $\Eb_\xi$ an expectation over all r.v.'s $\xi_j$ only,
\begin{align}
    G_{N}^{(2)}\inbraces{\sigma_1, \sigma_2} &\propto \Eb_\xi \insquare{  \exp \inparen{ \sigma_1 \inparen{ \beta \vec{A}_1^{\top}\inangle{\vec{y}}^- + \beta\sqrt{1 - q} \xi_1  + h}  }  \exp \inparen{ \sigma_2\inparen{ \beta \vec{A}_2^{\top}\inangle{\vec{y}}^- + \beta\sqrt{1 - q} \xi_2  + h}  }   }  \nonumber \\
    &\propto \inparen{ \frac{ \exp \inparen{ \sigma_1 \inparen{ \beta \vec{A}_1^{\top}\inangle{\vec{y}}^-  + h}  } }{ 2\cosh\inparen{  \beta \vec{A}_1^{\top}\inangle{\vec{y}}^-  + h }  }  } \inparen{ \frac{ \exp \inparen{ \sigma_2 \inparen{ \beta \vec{A}_2^{\top}\inangle{\vec{y}}^-  + h}  } }{ 2\cosh\inparen{  \beta \vec{A}_2^{\top}\inangle{\vec{y}}^-  + h }  } },
    \label{eq:intro_SK_2dimMarginal_LI}
\end{align} 
where we evaluated $\Eb \exp\lambda \xi = \exp\inparen{\lambda^2/2}$ for $\xi \sim \cN(0,1)$. Noting that $\vec{A}_j^{\top} \inangle{\vec{y}}^-$ is fixed when the disorder is conditioned upon, we see from \eqref{eq:intro_SK_2dimMarginal_LI} that the 2-dimensional SK marginal is indeed close to a product probability on $\inbraces{\pm 1}^2$. The result for the SK model is given formally in Theorem \ref{thm:SK_classical_localIndependence}.

Towards a generalization---one observes that the above argument works as long as the Hamiltonian admits a similar decomposition as in \eqref{eq:intro_SK_Ham_rewrite}. That is, when we can isolate the cavity fields experienced by the spins $x_1$, $x_2$ that we are considering for local independence. This motivates the definition of \emph{decomposable} Hamiltonians (Definitions \ref{definition:LI_abstract_deecomposableHam} and \ref{definition:LI_abstract_R0-deecomposableHam}). For instance, in the Ising Perceptron \eqref{eq:perceptron_Hamiltonian}, we will see that
\begin{align*}
    &-H_{N,M}(\vec{x}) = \sum_{m \leq M} u\inparen{  \frac{1}{\sqrt{N}} \sum_{i \leq N} g_{i,m} x_i } \\
    &\quad\quad  \simeq -H_{N-2,M}^-(\vec{y}) + x_1 \sum_{m \leq M} \frac{g_{1,m}}{\sqrt{N}} u'\inparen{ \frac{1}{\sqrt{N}} \sum_{3 \leq i \leq N} g_{i,m} x_i   } + x_2 \sum_{m \leq M} \frac{g_{2,m}}{\sqrt{N}} u'\inparen{ \frac{1}{\sqrt{N}} \sum_{3 \leq i \leq N} g_{i,m} x_i   },
\end{align*}
under certain conditions on the function $u : \RR \rightarrow \RR$, and for certain regimes of the ratio $M/N$. The situation appears different from the SK model, since we now consider random projections of the vector $\inparen{ u'(\cdots) }_{m \leq M}$ (see Section \ref{sec:perceptron} for definitions), which we refer to as the auxiliary spin system for the Gardner models (also called conjugate spin system \cite{barbier2021performance}). It is an advantage of our framework---Definitions \ref{definition:LI_abstract_deecomposableHam} and \ref{definition:LI_abstract_R0-deecomposableHam} and Theorems \ref{thm:ProjResult_DisorderedCase_supOverBL_LM_2p} and \ref{thm:LI_abstract_theorem}, that it is able to treat the situation for the SK, Gardner, and other mean-field spin glass models simultaneously.


\subsection{Related work}
\label{sec:intro_further_related_work}

\parag{Related local independence results.} We discuss a well-known approach to local independence in spin glasses (e.g.~Talagrand \cite{talagrand2010mean} Section 1.4), as well as the additional insights afforded by the results in this work. For Gibbs measures \eqref{eq:intro_GibbsMeasure_definition} with $\mu$ supported on $\inbraces{\pm 1}$, it is known that under symmetry among sites under the disorder $\Eb$, the quantity $\Eb \insquare{ \inparen{ \inangle{x_1\dots x_k} - \inangle{x_1}\dots\inangle{x_k}  }^{2}}$ is small whenever $N^{-1} \vec{x}^1 \boldsymbol{\cdot} \vec{x}^2$ concentrates under $\Eb \inangle{\cdot}$ (\cite{talagrand2010mean} Proposition 1.4.14, see also \cite{frohlich1987some}).


In the case of $\pm 1$ models, which can be parametrized by their mean vectors $\inangle{\vec{x}}$, this decorrelation can be extended to independence. Consequently, in the classical SK model \eqref{eq:intro_SK_Hamiltonian}, it is shown in \cite{talagrand2010mean} Theorem 1.4.15 that for say $\beta < 1/2$, where overlap concentration is known to hold,
\begin{align*}
    \Eb \insquare{ \textnormal{TV}\inparen{  G_{N}^{(k)} , \, P_k   }^2 } \leq \frac{K(k)}{N},
\end{align*}
where $G_N^{(k)}$ is as in \eqref{eq:intro_LI}, and where $P_k$ is the probability on $\inbraces{\pm 1}^k$ with $P_k\insquare{ \inbraces{x_j = \sigma_j : j \leq k}  } = 2^{-k} \prod_{j \leq k} \inparen{ 1 + \sigma_i \inangle{x_i}  }$.


On the other hand, if we consider a slight generalization of the SK model where the spins are allowed to take values in $\insquare{-1,1}$ (this is the SK model with $d$-component spins, with $d=1$, and $\mu$ the uniform probability over $\insquare{-1,1}$, see Section \ref{section:SK_[-1,1]_decompose}), then additional work or hypotheses may be needed to extend the decorrelations to full independence. 

For instance, under a perturbation of the Hamiltonian, the \emph{multi-overlaps} \cite{sanctis2009self} of the system can be shown to concentrate, and this yields asymptotic local independence results for general (not necessarily Ising spin) models that hold on average over the perturbations \cite{barbier2020strong}, \cite{barbier2022strong}. We mention also the work \cite{montanari2008estimating} which shows local independence in the setting of diluted models through a similar pertubative approach. These are related to the `pinning lemmas' to be discussed shortly after.

We remark that the results in this paper lead as in \eqref{eq:intro_SK_2dimMarginal_LI} to a concrete form of the approximate marginal densities, and also provide convergence rate guarantees. Furthermore, the results in this paper expose the relationship between random projections of the Gibbs measure and local independence, which adds an additional layer to the characterization of the RS phase as in Figure \ref{fig:RS_equivalences}.

\parag{Replica-symmetric phase and boundaries.} The local independence phenomenon is typically associated with the \emph{replica-symmetric} (RS) region, which is the region where the free energy $N^{-1}\Eb \log Z_N$ admits a RS formula (see \cite{talagrand2011mean} Definition 13.3.2). The RS phase is also associated with overlap concentration of the order parameters---that is, for the SK model, when $N^{-1}\vec{x}^1 \boldsymbol{\cdot} \vec{x}^2$ concentrates on some number whp.~$\Eb\inangle{\cdot}$. Despite considerable progress towards a rigorous understanding of the Sherrington-Kirkpatrick (SK) model, including proof of the celebrated Parisi formula \cite{talagrand2006parisi}, \cite{panchenko2013parisi}, there remain unanswered questions about the exact boundary where the RS formula fails, and the supremum of temperatures for which overlap concentration fails, this is referred to as the Almeida-Thouless conjecture (\cite{talagrand2010mean} Section 1.8, see \cite{toninelli2002almeida} and \cite{chen2021almeida} for SK context with rigorous results).

In some non-rigorous contexts, it is claimed that the RS phase is in fact synonymous with local independence, where the RS ansatz in the replica method from physics is believed to be equivalent to the assumption of asymptotic independence of $O(1)$ subset of coordinates under Gibbs measure used in the non-rigorous cavity method from physics (see \cite{mezard1987spin} Chapter V or \cite{mezard2009information} Chapter 14, see remarks in \cite{bolthausen2022gardner} Section 1.2.1).

Such RS phase transition boundaries are also important in approximate message passing (AMP) \cite{bayati2011dynamics}, \cite{feng2022unifying}. For instance in \cite{bolthausen2014iterative} it is shown that outside the Almeida-Thouless line, the AMP algorithm for the SK model does not converge. Even outside the high-temperature phase, the study of RS properties remains important as ultrametricity allows decomposition of the configuration space into clusters and the behavior of the Gibbs measure restricted to the pure states resembles that of a RS Gibbs measure; e.g.~conditional local magnetizations satisfying TAP equations \cite{auffinger2019thouless}, \cite{chen2022tap}. 

Many problems that occur in statistics, computer science, and engineering related to high-dimensional inference, estimation, and identifying algorithmic/computational phase transitions can be mapped to spin glass models (for overviews and information on \emph{planted} models see \cite{barbier2017phase}, \cite{zdeborova2016statistical}). These comments reflect a need for a greater understanding of the RS phase, including properties such as local independence that potentially characterize it, which is one of the motivations for this work.

\parag{Pinning lemmas, decoupling principles, mixture of products representations.} In statistical physics, local independence is related to a class of results known as `pinning lemmas' \cite{montanari2008estimating}. Such results state that under a small perturbation of the Hamiltonian, or conditioned on additional side information (e.g.~Gaussian noise-corrupted observations of $\vec{x}$), the spins under $\inangle{\cdot}$ decorrelate, and the overlap is forced to concentrate (see \cite{miolane2019fundamental} Chapter 2 and \cite{el2022information}). Pinning lemmas also hold when the Gibbs measure (posterior distribution) of $\vec{x}$ has sparse dependencies \cite{montanari2008estimating}, \cite{bapst2016harnessing}, \cite{mezard2009information} Parts IV and V. The local independence (often called decoupling principle) that results is related to the single-letterization/scalarization that occurs in RS free energy formulas and in TAP/AMP state evolutions---see e.g.~ \cite{korada2010tight}, \cite{guo2005randomly} for CDMA setting; \cite{bayati2011dynamics} Corollary 1 for dense factor graphs setting; see \cite{reeves2016replica} for the use of `weak decoupling' in a proof of the RS formula for random linear estimation, using the projection result for the zero-overlap setting \cite{reeves2017conditional}.

The pinning lemmas can also be used to construct \emph{mixture of products} approximations to Gibbs measures. By conditioning on the side-information, e.g.~noisy Gaussian observations of the posterior \cite{el2022information}, or revealing subsets of coordinates \cite{montanari2008estimating}, \cite{jain2019mean}, \cite{manurangsi2016birthday}, \cite{raghavendra2012approximating}, \cite{Austin2020} Lemma 9.3, or conditioned on partitions of the configuration space where the Hamiltonian has `low-complexity gradients' \cite{austin2019structure}, \cite{eldan2018gaussian}, \cite{eldan2018decomposition}. We briefly mention that there exists other routes to a mixture of products representations, for instance when the Hamiltonian is (or can be made into) a positive definite quadratic form (see \cite{bauerschmidt2019very} proof of Theorem 1).

\subsection{Organization and notation}

\parag{Organization}
In Section \ref{sec:LI_abstract} the abstract local independence result Theorem \ref{thm:LI_abstract_theorem} is proved; in Section \ref{sec:LI_implies_TSOC} the converse result is proved; the rest of the paper demonstrates the efficacy of these two general results on various spin glass models: in Section \ref{sec:SK_classical_[-1,1]_variant} we consider the classical and $[-1,1]$-SK models; in Section \ref{sec:perceptron} the Ising Perceptron model; in Section \ref{sec:Shcherbina-Tirozzi} the Shcherbina-Tirozzi model.   
\parag{Notation} For a vector $\vec{x}$ on $\RR^N$ we write $\norm{\vec{x}}$ for the Euclidean norm. For a random vector $\vec{x}$, the law of $\vec{x}$ is written $\cL(\vec{x})$. The conditional distribution of $\vec{x}$ given $\vec{y}$ is written $\cL(\vec{x} \, | \, \vec{y})$. Indicator functions of a set $A$ are denoted by $\rchi_A(\cdot)$. For integer $k$ and $B \subseteq \RR^k$, we write $\rchi_B$ to mean the indicator of the cylinder set $B \times \RR^\infty$. The symbol $\ud \vec{x}$ in an integral refers to Lebesgue measure on $\RR^N$. We use $K$ to refer to a constant that may depend on various model-specific parameters, but will never depend on the ambient dimension $N$. The value $K$ may change from line to line. 

We often write `using replicas' to mean the following: $\inparen{ \Eb\insquare{X} }^2 = \Eb\insquare{X^1 \cdot X^2} $, where $X^1, X^2$ are independent copies of $X$, i.e.~drawn iid from $\cL(X)$. 
The overlap $R_{1,2}$ of $\cL(\vec{x})$, where $\vec{x}$ is a random vector in $\RR^N$, is the normalized inner product between two iid vectors $\vec{x}^1$ and $\vec{x}^2$, distributed according to $\cL(\vec{x})$, i.e.~$R_{1,2} := N^{-1}\sum_{i \leq N} x_i^{1} x_i^{2}$.

In the spin glass setting, we use $\inangle{\cdot}$ to mean a Gibbs expectation, where $\inangle{\cdot}$ may be a random measure wrt.~the disorder. Expectation over the disorder is denoted $\Eb_{\ud}$. Replicas refer to i.i.d.~draws from the (quenched) Gibbs measure, and the replica index is indicated in the superscript---for instance $\inangle{\vec{x}} \boldsymbol{\cdot} \inangle{\vec{x}} = \inangle{ \vec{x}^1 \boldsymbol{\cdot} \vec{x}^2 }^{\otimes 2}$. We typically suppress the superscript $\otimes 2$ when it is clear from the context. The notation `$\Eb_{\xi}$' appears often, and it means an expectation over all random variables labeled $\xi$, including all replicas $\xi^\ell$'s, with all other r.v.'s given. The discrete hypercube $\inbraces{\pm 1}^N$ is abbreviated to $\Sigma_N$.


On the space of real square matrices we consider the Hilbert-Schmidt or Frobenius inner product $\inangle{A, B}_{\textnormal{HS}} = \Tr (AB^{\top})$ which induces the norm $\norm{A}_{\textnormal{HS}} = \sqrt{\Tr(AA^{\top})}$.



For a function $g : \RR^N \rightarrow \RR$, its Lipschitz norm is $\norm{g}_{\textnormal{Lip}} := \sup_{\bx \neq \by}  \abs{g(\bx) - g(\by)} / \abs{\bx - \by}$.


Gaussian integration by parts formula will be used many times and is abbreviated to GIPF; see e.g.~\cite{talagrand2010mean} Appendix A.4.

\section{Local independence---general theorem}
\label{sec:LI_abstract}

For every $N \in \NN$, let $\mat{A} = \inparen{A_s}_{s \leq S}$, $S = S(N)$ be a collection of real-valued r.v.'s. For any such $\mat{A}$, we define a Hamiltonian to be a function $-H_N(\cdot; \mat{A}) : \RR^N \rightarrow \RR$. The Hamiltonian is regarded as a random function, defined conditionally on the r.v.'s $\mat{A}$, which are called the \emph{disorder}. When there is no ambiguity, we typically suppress the notational dependence on the disorder and write $-H_N(\vec{x}) = -H_N(\vec{x}; \mat{A})$.

The associated Gibbs measure is denoted by $\inangle{\cdot}$---that is, for integrable $f : \RR^N \rightarrow \RR$,
\begin{align}
    \inangle{f} &= \frac{1}{Z_{N}} \int_{\RR^N} f(\vec{x}) \exp\inparen{-H_{N}(\vec{x})} \, \mu^{\otimes N}\inparen{\ud \vec{x}},
\end{align}
where $\mu$ is a reference measure on $\RR$, and $Z_N$ is the partition function. In other words we restrict to the case when the reference measure for $\inangle{\cdot}$ is an iid product measure. Notice that the Gibbs measure is a random probability measure, with randomness wrt.~the disorder. An expectation over the disorder is denoted by $\Eb_{\ud}$ or simply $\Eb$.

For this entire work we think of the number of coordinates to consider for local independence $2 \leq k < N$ as fixed (not growing with $N$). We let $\vec{x} = (x_1,\dots,x_N)$ denote a generic point with $N$ coordinates, and let $\vec{y}$ be the shorthand for $(x_{k+1},\dots,x_N)$.

Furthermore, for this entire work, we only consider spin glass models with \textbf{exchangeability among sites}, that is, for any integrable $f : \RR^k \rightarrow \RR$, 
\begin{align}
    \Eb \inangle{f(x_1,\dots,x_k)} = \Eb \inangle{ f(x_{i_1},\dots,x_{i_k})  },
    \label{eq:exchangeability_among_sites}
\end{align}
for any subset of indices $\inbraces{i_1,\dots,i_k}$ of $\inbraces{1,\dots,N}$. For convenience we therefore only consider decoupling the first $k$-spins, with the understanding that all results extend to any subset of $k$-spins.

\begin{definition}
A Hamiltonian $-H_{N}$ on $\RR^N$ is \emph{decomposable} if for every integer $1 \leq k < N$, it satisfies an identity
\begin{align}
    -H_{N}(\vec{x}; \mat{A}) &= -H^-_{N-k}(\vec{y}; \mat{A}^-) + \sum_{j \leq k} x_j\, \varrho\,  \vec{A}_j  \boldsymbol{\cdot} \vec{w} + \sum_{j \leq k} f_j(x_j; \, \vec{h}),
    \label{eq:meta_decomposableHamiltonian}
\end{align}
where the disorder can be written as $\mat{A} = \inparen{\mat{A}^-, \vec{A}_1,\dots,\vec{A}_k, \vec{h}}$, where $\mat{A}^-, \vec{A}_1,\dots,\vec{A}_k, \vec{h}$ are disjoint collections of the r.v.'s in $\mat{A}$, and for some integer $M = M(N) > 0$, we additionally have
\begin{itemize}
    \item $-H^-_{N-k}(\cdot; \mat{A}^-)$ is a Hamiltonian on $\RR^{N-k}$ with disorder $\mat{A}^-$;
    \item $\varrho > 0$ is a number, that may depend on $N$;
    \item $f_j(\cdot; \, \vec{h}) : \RR \rightarrow \RR$ are functions that do not depend on $N$; and $\vec{h}$ is independent of all other disorder;
    \item $\vec{w} = \vec{w}(\vec{y}; \mat{A}^-) \in \RR^M$, where $\vec{w}(\cdot; \mat{A}^-) : \RR^{N-k} \rightarrow \RR^M$ is a function;
    \item for all $j \leq k$, $\vec{A}_j \overset{\textnormal{iid}}{\sim} \cN(0, M^{-1} I_M)$; and $\inparen{\vec{A}_1,\dots,\vec{A}_k}$ are independent of $\mat{A}^-$.
\end{itemize}
\label{definition:LI_abstract_deecomposableHam}
\end{definition}
 In what follows, we usually suppress the dependence on disorder in the above definition, and say that a Hamiltonian is decomposable if we can write
\begin{align}
    -H_{N}(\vec{x}) &= -H^-_{N-k}(\vec{y}) + \sum_{j \leq k} x_j \, \varrho\, \vec{A}_j  \boldsymbol{\cdot} \vec{w}  + \sum_{j \leq k} f_j(x_j),
    \label{eq:meta_decomposableHamiltonian_informal}
\end{align}
where `the randomness in $\mat{A}_j$'s is independent of the disorder in $-H_{N-k}^-$.

It is helpful to relate this definition in the setting of the SK model described in \eqref{eq:intro_SK_Ham_rewrite}. There $k=2$, $-H_{N-k}$ is in \eqref{eq:intro_SK_Ham_truncated}, $M = N-k$, $\varrho = \beta \sqrt{N-k}/\sqrt{N}$, $\vec{w} = \vec{y}$\footnote{The extra level of generality in $\vec{w}$ is needed in the Gardner models (Section \ref{sec:perceptron}  and \ref{sec:Shcherbina-Tirozzi}).}, and $f_j(x_j) = hx_j$ (here $f_j$'s are non-random). We see however, that there is an additional term $x_1x_2 \frac{\beta}{\sqrt{N}} g_{12}$ not captured by the above definition. This term arises from the interactions among the $k$ spins that we are considering for independence. Typically this term is of smaller order, and we can drop it. This motivates the following definition.


\begin{definition}
For any integer $p \geq 1$, a Hamiltonian $-H_{N}$ on $\RR^N$ is $R_0(N,p)$\emph{-decomposable} if there exists a decomposable Hamiltonian $-H_{N,0}$ on $\RR^N$ such that for all integers $1 \leq k < N$, for any measurable subset $B$ of $\RR^k$, then
\begin{align}
    \Eb \insquare{\inparen{ G_N^{(k)}\!\insquare{B} - G_{N,0}^{(k)}\!\insquare{B}   }^{2p}} \leq R_0(N,p),
\end{align}
where $G_{N}^{(k)}$ and $G_{N,0}^{(k)}$ are respectively the marginals of $G_{N}$ and $G_{N,0}$ on any $k$ coordinates. In particular, a decomposable Hamiltonian is $0$-decomposable.
\label{definition:LI_abstract_R0-deecomposableHam}
\end{definition}

We can now state the main result of this paper; the proof is in Appendix \ref{sec:suppProofs_LI_abstract}.

\begin{theorem}
Let $-H_{N}(\vec{x})$ be a $R_{0}(N,p)$-decomposable Hamiltonian with an associated decomposable Hamiltonian $-H_{N,0}$ given by \eqref{eq:meta_decomposableHamiltonian}. Let $\inangle{\cdot}^-$ denote the Gibbs measure on $\RR^{N-k}$ associated to $-H^-_{N-k}$. Suppose further that
\begin{enumerate}
    \item thin-shell and overlap concentration for $\vec{w}$ holds under $\inangle{\cdot}^-$, i.e. there exists $\Xi > \Upsilon \geq 0$ such that
    \begin{align}
        \Eb \inangle{ \inparen{ \frac{1}{M} \sum_{m \leq M} w_m^{1,2} - \Xi }^2 }^- \leq \frac{K}{M}; \quad \Eb \inangle{ \inparen{ \frac{1}{M} \sum_{m \leq M} w_m^{1} w_m^{2} - \Upsilon }^2 }^- \leq \frac{K}{M},
        \label{eq:LI_abstract_TSOC_hypothesis}
    \end{align}
    \item for any $P \in \RR$, for all $j \leq k$,
    \begin{align}
        \int_{\RR} \exp\inparen{ \frac{x_j^2 \varrho^2}{2}\inparen{\Xi - \Upsilon} + f_j(x_j) + x_j P    } \mu\inparen{\ud x_j} \geq 1\;\; \textnormal{a.e.};
        \label{eq:LI_abstract_E>=1_hypothesis}
    \end{align}
    \item for all $m \leq M$, $\norm{w^2_m}_{\infty} \leq D^2$;
    \item there exists some constant $K_0 > 0$ such that
    \begin{align}
        \frac{1}{K_0} N \leq M \leq K_0 N 
        \label{eq:LI_abstract_M=N_hypothesis}
    \end{align}
\end{enumerate}
For every $0 < \epsilon < 1/2$, define $C > 0$ to be
    \begin{align*}
        C &:= \Eb \int_{\RR^k} \sum_{j\leq k} \abs{x_j} \exp\inparen{  4kp^2\varrho^2(D^2 + \Xi)(1+ 4\bar{\epsilon}^2) \inparen{ \sum_{j\leq k} \abs{x_j}}^2  + \sum_{j \leq k} f_j(x_j) } \mu^{\otimes k}\inparen{\ud x_1,\dots, \ud x_k},
    \end{align*}
where $\bar{\epsilon} := \inparen{\frac{1}{2\epsilon} - 1}$. Then for every integer $p \geq 1$, whenever $N \geq \exp\inparen{ \sqrt{16k}/\epsilon  }$, the following happens.
\begin{enumerate}
    \item Let $\GG_j$ denote the probability measure on $\RR$ given by, for every measurable $A \subseteq \RR$,
    \begin{align}
        \GG_j\insquare{A} &\propto \int_A \exp\inparen{ x_j \varrho \inparen{\vec{A}_j \boldsymbol{\cdot} \inangle{ \vec{w} }^- } + \frac{1}{2} x_j^2 \varrho^2\inparen{ \Xi - \Upsilon  }  + f_j(x_j)  } \mu\inparen{\ud x_j}.
    \end{align}
    Then we have
    \begin{align}
        \Eb_{\ud} \insquare{ \sup_{B \subseteq \RR^k} \inparen{ G_{N}^{(k)}\!\insquare{B}  -  \inparen{ \bigotimes_{j \leq k}  \GG_j} \insquare{B}   }^{2p}  } \leq K\inparen{ R_0(N,p) +  \frac{C}{N^{1/2 - \epsilon}\rchi_{\Upsilon > 0} + N^{1/4 - \epsilon/2} \rchi_{\Upsilon = 0}}  },
        \label{eq:LI_abstract_GGj_statement}
    \end{align} 
    \item Let $\HH_j$ denote the probability measure on $\RR$ given by, for every measurable $A \subseteq \RR$,
    \begin{align}
        \HH_j\insquare{A} &\propto \int_A \exp\inparen{ x_j \varrho \sqrt{\Upsilon}z_j + \frac{1}{2} x_j^2 \varrho^2\inparen{ \Xi - \Upsilon  }  + f_j(x_j)  } \mu\inparen{\ud x_j},
    \end{align}
    where $\inparen{z_j}_{j \leq k}$ are independent standard Gaussians, independent of all other sources of randomness. Then we have
    \begin{align}
        \Eb_{\ud}\Eb_{\vec{z}} \insquare{ \sup_{B \subseteq \RR^k} \inparen{ G_{N}^{(k)}\!\insquare{B}  -  \inparen{ \bigotimes_{j \leq k}  \HH_j} \insquare{B}   }^{2p}  } \leq K\inparen{ R_0(N,p) +  \frac{C}{N^{1/2 - \epsilon}} }.
        \label{eq:LI_abstract_HHj_statement}
    \end{align}
\end{enumerate}
\label{thm:LI_abstract_theorem}
\end{theorem}

\begin{remark}
Comments on the conditions of the theorem.
\begin{itemize}
    \item Hypothesis \ref{eq:LI_abstract_TSOC_hypothesis} uses concentration rates $O(1/M)$ for convenience; an analogous result with a different rate holds if the right-hand sides of \eqref{eq:LI_abstract_TSOC_hypothesis} is replaced by any $o(1)$ function.
    \item Hypothesis \eqref{eq:LI_abstract_E>=1_hypothesis} is generally satisfied for $\pm 1$ models, where $f_j(x_j) = hx_j$ plays the role of external field, and where $\mu$ is the uniform measure on $\inbraces{\pm 1}$. In that case the LHS of \eqref{eq:LI_abstract_E>=1_hypothesis} becomes a $\cosh(\cdots)$, which is bounded below by $1$. In the $\insquare{-1,1}$---SK model, that is when $\mu$ is the uniform probability over $[-1,1]$, and $f_j(x_j) = hx_j$, we can, by positivity of the integrand, lower bound the integral over appropriate intervals for which the integrand has positive exponent, from which \eqref{eq:LI_abstract_E>=1_hypothesis} holds. In the Shcherbina-Tirozzi model \eqref{eq:LI_abstract_E>=1_hypothesis} is satisfied if the $\kappa$ in the regularizing term $-\kappa \norm{\vec{x}}^2$ is  chosen large enough (see \eqref{eq:ST_Hamiltonian_original}).

    The lower bound of $1$ is also for convenience---any constant lower bound will work, because Lemma \ref{lemma:TalVol1_Lemma1.7.14} can be amended accordingly, and this only changes the constant in the final result (see Equation \eqref{eq:LI_abstract_D'-D,E'-E}).
    
    \item The boundedness condition on $\vec{w}$ is satisfied for the models under consideration. In the SK models $\vec{y}$ plays the role of $\vec{w}$. In the Gardner spin glass models, $\alpha u'(S_m)$ plays the role of $w_m$ (see \eqref{eq:perceptron_Hamiltonian} or \eqref{eq:ST_Hamiltonian_original} for definitions), and $\alpha$ is a constant and it is often assumed that $u'$ is bounded.
    \item Hypothesis \eqref{eq:LI_abstract_M=N_hypothesis} is typical of mean-field models. In the SK models $M$ will be $N-k$, and for $k$ fixed it is easy to see that \eqref{eq:LI_abstract_M=N_hypothesis} is satisfied. In the Gardner spin glass models, $M$ represents the number of intersecting hyperplanes under consideration, and the regime of interest is for $M$ growing proportionally with $N$, so that the ratio $M/N$ is constant (or converges to a constant).
\end{itemize}

\end{remark}

\section{Local independence implies thin-shell and overlap concentration}
\label{sec:LI_implies_TSOC}

In this section we give a type of converse to Theorem \ref{thm:LI_abstract_theorem}. Essentially, when an exchangeable random probability measure on $\RR^N$ has its marginals close in total variation to a product measure, i.e.~the measure is locally independent, then thin-shell and overlap concentration holds.

Typically we are interested in the disordered case where the probability measure $G_N$ is random wrt.~the disorder $\Eb$, and where exchangeability among sites holds. In this case we define the variance of the thin-shell and overlap respectively by
\begin{align*}
    \Var R_{11} = \Eb G_N\insquare{ \inparen{ \frac{\norm{\vec{x}}^2}{N} - \Eb G_N\insquare{\frac{\norm{\vec{x}}^2}{N}} }^2  } ; \quad \Var R_{12} = \Eb G_N\insquare{ \inparen{ \frac{ \vec{x}^1 \boldsymbol{\cdot} \vec{x}^2 }{N} - \Eb G_N\insquare{\frac{ \vec{x}^1 \boldsymbol{\cdot} \vec{x}^2 }{N}} }^2  }.
\end{align*}
In the disordered case, we require that $G_N$ is locally independent for typical realizations of the disorder; and that additionally, the random product measure that it is close to is also \emph{unconditionally} a product measure, in a manner made precise in \eqref{eq:LI=>TSOC_HHj_is_uncond_a_product_measure_hypothesis}.

\begin{definition}
    For two probability measures $\mu$ and $\nu$ on the same measurable space $(\cX, \cF)$, their \emph{total variation distance} is defined by
    \begin{align}
        \norm{\mu - \nu}_{\textnormal{TV}} = \sup_{B \subseteq \cF} \abs{\mu\insquare{B} - \nu\insquare{B}}.
        \label{eq:TV_definition}
    \end{align}
\end{definition}

\begin{theorem}
    Let $G_N$ be a random probability measure on $\RR^N$, with randomness denoted by $\Eb_{\ud}$. Suppose that exchangeability among sites holds as in \eqref{eq:exchangeability_among_sites}. Let $G_{N}^{(k)}$ be the marginal of $G_N$ on the first $k$-coordinates. Let $\HH_j$, $j \leq 2$ be random probability measures on $\RR$ with randomness $\Eb_{z}$, independent of the randomness in $\Eb_{\ud}$. Suppose that
    \begin{align}
        \Eb \norm{ G_N^{(2)} - \HH_1 \otimes \HH_2  }_{\textnormal{TV}} \leq \frac{K}{N^\eta}.
        \label{eq:LI=>TSOC_quenched_LIHypothesis}
    \end{align}
    Suppose further that the $\HH_j$'s independent and identically distributed under $\Eb_{z}$, i.e., there exists a probability measure $\HH$ on $\RR$ such that for all measurable functions $f_1, f_2 : \RR \rightarrow \RR$,
    \begin{align}
        \Eb_{z} \inparen{\HH_1 \otimes \HH_2} \insquare{ f_1 f_2 } = \HH f_1 \cdot \HH f_2,
        \label{eq:LI=>TSOC_HHj_is_uncond_a_product_measure_hypothesis}
    \end{align}
    so that in fact $\Eb_{z} \HH_1 = \HH$. Denote $\nu_N := \Eb_{\ud} G_N$, and suppose $\nu_N$ is an exchangeable measure. Let $\nu_N^{(k)} := \Eb_{\ud} G_{N}^{(k)}$. Then the following occurs.
    \begin{enumerate}
        \item Suppose that
        \begin{align}
            \int x^4 \nu_N^{(1)}(\ud x) \leq S_N; \quad \textnormal{and} \quad \int x^4 \, \HH(\ud x) \leq K,
            \label{eq:LI=>TSOC_FourthMomentControl}
        \end{align}
        for some sequence $(S_N)_{N \geq 1}$ in $\RR$, for which $\inf_N S_N \geq K_0 > 0$. Then for $\alpha  = \min(1, \eta/2)$, we have
        \begin{align}
            \Var R_{12} \leq \frac{K S_N}{N^{\alpha}},
            \label{eq:LI=>TSOC_VarR12_statement}
        \end{align}
        \item  Furthermore, if it additionally holds that 
        \begin{align}
            \int x^8 \nu_N^{(1)}(\ud x) \leq T_N; \quad \textnormal{and} \quad \int x^8 \, \HH(\ud x) \leq K,
            \label{eq:LI=>TSOC_EighthMomentControl}
        \end{align}
        for some sequence $(T_N)_{N \geq 1}$ in $\RR$ for which $\inf_N T_N \geq K_0 > 0$, then
        \begin{align}
            \Var R_{11} \leq \frac{K (\sqrt{T_N} + S_N)}{N^{\alpha}}.
            \label{eq:LI=>TSOC_VarR11_statement}
        \end{align}
    \end{enumerate}
    \label{thm:LI=>TSOC_abstract}
\end{theorem}

\begin{proof}
    From definition \eqref{eq:TV_definition} and hypotheses \eqref{eq:LI=>TSOC_quenched_LIHypothesis} and \eqref{eq:LI=>TSOC_HHj_is_uncond_a_product_measure_hypothesis}, we have by Jensen's inequality that
    \begin{align}
        \norm{ \nu_N^{(2)} - \HH^{\otimes 2} }_{\textnormal{TV}} \leq \Eb \norm{ G_N^{(2)} - \HH_1 \otimes \HH_2  }_{\textnormal{TV}} \leq  \frac{K}{N^\eta}.
        \label{eq:LI=>TSOC_annealed_LIHypothesis}
    \end{align}
    
    Let $\vec{x} = \inparen{x_i}_{i \leq N} \sim \nu_N$. Expanding, and using exchangeability, we obtain
    \begin{align}
        \Var R_{12} &= \frac{1}{N} \insquare{ \inparen{\nu_N\!\insquare{x_1^2}}^2 - \inparen{\nu_N\!\insquare{x_1}}^4  } \nonumber\\
        &\quad\quad + \frac{N(N-1)}{N^2} \inparen{ \nu_N\!\insquare{x_1x_2} + \nu_N\!\insquare{x_1} \nu_N\!\insquare{x_2} } \inparen{ \nu_N\!\insquare{x_1x_2} - \nu_N\!\insquare{x_1} \nu_N\!\insquare{x_2} }.
        \label{eq:LI=>TSOC_VarR12_Expansion}
    \end{align}
    It is straightforward to use \eqref{eq:LI=>TSOC_FourthMomentControl}, exchangeability, and H{\"o}lder's inequality to obtain the following bounds
    \begin{align}
        \inparen{ \nu_N\!\insquare{x_1^2} }^2 \leq S_N; \quad \nu_N\!\insquare{ x_1 x_2  } \leq \sqrt{S_N}; \quad \nu_N\!\insquare{x_1} \nu_N\!\insquare{x_2} \leq \sqrt{S_N}; \quad \nu_N\!\insquare{x_1} \leq S_N^{1/4}.
        \label{eq:LI=>TSOC_CoordinateWiseControl_I}
    \end{align}
    The objective is to show that the term $\nu_N\!\insquare{x_1x_2} - \nu_N\!\insquare{x_1} \nu_N\!\insquare{x_2}$ in \eqref{eq:LI=>TSOC_VarR12_Expansion} is small. For brevity write $(x,y)$ for any two-coordinates; so that in particular 
    \begin{align*}
        \nu_N\!\insquare{xy} = \int xy \; \nu_N^{(2)}(\ud x, \ud y), \quad \textnormal{and} \quad \HH^{\otimes 2}\insquare{xy} = \HH\!\insquare{x}\HH\!\insquare{y} = \inparen{\HH\!\insquare{x}}^2.
    \end{align*}
    Therefore we can write
    \begin{align}
        \nu_N\!\insquare{x y} - \nu_N\!\insquare{x} \nu_N\!\insquare{y} \leq \textnormal{I} + \textnormal{II},
    \end{align}
    where 
    \begin{align*}
        \textnormal{I} := \abs{ \nu_N\!\insquare{x y} - \HH^{\otimes 2} \insquare{xy}  }; \quad \textnormal{and} \quad \textnormal{II} := \abs{ \inparen{ \nu_N\!\insquare{x}}^2 - \inparen{  \HH\!\insquare{x}  }^2  }. 
    \end{align*}
    Let $\lambda_N$ be the signed measure given by
    \begin{align}
        \lambda_N := \nu_N^{(2)} - \HH^{\otimes 2},
    \end{align}
    and let $\lambda_N := \lambda_N^+ - \lambda_N^-$ be its Hahn-Jordan decomposition, and let $(\cP, \cN) \subseteq \RR^2 \times \RR^2$ be its Hahn decomposition. That is, $\cP$, $\cN$ are measurable sets such that $\lambda_N^+\insquare{B} = \lambda_N\insquare{B \cap \cP}$ and $\lambda^-_{N}\insquare{B} = -\lambda_N\insquare{B \cap \cN}$. We have
    \begin{align}
        \textnormal{I} = \abs{\lambda_N \insquare{xy}} \leq \abs{\lambda_N^+\insquare{xy}} + \abs{\lambda_N^- \insquare{xy}} &\leq \sqrt{\lambda_N^+ \insquare{x^2 y^2}} \sqrt{\lambda_N^+\insquare{\RR^2}} + \sqrt{\lambda_N^- \insquare{x^2 y^2}} \sqrt{\lambda_N^-\insquare{\RR^2}} \nonumber\\
        &\leq \sqrt{ \lambda_N \insquare{x^2 y^2} } \sqrt{\lambda_N \insquare{\cP}} + \sqrt{\lambda_N \insquare{x^2 y^2}} \sqrt{\lambda_N \insquare{\cN}} \nonumber\\
        &\leq \frac{K \sqrt{S_N}}{N^{\eta/2}},
        \label{eq:LI=>TSOC_boundforI}
    \end{align}
    where in the last inequality we have used that $\lambda_N\insquare{ x^2 y^2 } = \nu_N^{(2)}\!\insquare{x^2 y^2 } - \HH^{\otimes 2} \insquare{ x^2 y^2 } \leq \nu_N\!\insquare{x^4} + \HH\!\insquare{x^4} \leq S_N + K \leq K' \cdot S_N$ by \eqref{eq:LI=>TSOC_CoordinateWiseControl_I} and that $S_N \geq K_0 > 0$ for all $N$; and also that $\lambda_N\insquare{\cP}$ and  $\lambda_N\insquare{\cN}$ are both bounded above by $\norm{ \nu_N^{(2)} - \HH^{\otimes 2} }_{\textnormal{TV}}$ for which \eqref{eq:LI=>TSOC_annealed_LIHypothesis} applies.
    
    For any measurable $B_1 \subseteq \RR$, we can choose $B = B_1 \times \RR$ in definition \eqref{eq:TV_definition} so that \eqref{eq:LI=>TSOC_annealed_LIHypothesis} yields 
    \begin{align}
        \norm{ \nu_N^{(1)} - \HH }_{\textnormal{TV}} \leq \frac{K}{N^\eta}.
        \label{eq:LI=>TSOC_annealed_LIHypothesis_1dim}
    \end{align}
    We have
    \begin{align}
        \textnormal{II} &= \abs{ \nu_N^{(1)}\!\insquare{x} + \HH\!\insquare{x} } \abs{ \nu_N^{(1)}\!\insquare{x} - \HH\!\insquare{x} } \leq \frac{K \sqrt{S_N}}{N^{\eta/2}},
        \label{eq:LI=>TSOC_boundforII}
    \end{align}
    where the first factor is bounded by \eqref{eq:LI=>TSOC_CoordinateWiseControl_I} and the second factor bounded by the same mechanism as in \eqref{eq:LI=>TSOC_boundforI}, using \eqref{eq:LI=>TSOC_annealed_LIHypothesis_1dim}. Substituting \eqref{eq:LI=>TSOC_CoordinateWiseControl_I}, \eqref{eq:LI=>TSOC_boundforI}, and \eqref{eq:LI=>TSOC_boundforII} into \eqref{eq:LI=>TSOC_VarR12_Expansion} gives
    \begin{align*}
        \Var R_{12} \leq \frac{S_N}{N} + \frac{K S_N}{N^{\eta/2}}
    \end{align*}
    which finishes the proof of \eqref{eq:LI=>TSOC_VarR12_statement}.

    The proof of \eqref{eq:LI=>TSOC_VarR11_statement} is analogous. By expanding and using exchangeability we have the identity 
    \begin{align}
        \Var R_{11} &= \frac{1}{N} \insquare{ \nu_N\!\insquare{x_1^4} - \inparen{\nu_N\!\insquare{x_1^2}}^2  }  + \frac{N(N-1)}{N^2} \inparen{ \nu_N\!\insquare{x_1^2 x_2^2} + \nu_N\!\insquare{x_1^2} \nu_N\!\insquare{x_2^2} }.
        \label{eq:LI=>TSOC_VarR11_Expansion}
    \end{align}
    The first term in \eqref{eq:LI=>TSOC_VarR11_Expansion} is bounded by $S_N/N$. For the second term, write 
    \begin{align*}
        \inparen{ \nu_N\!\insquare{x^2 y^2} + \nu_N\!\insquare{x^2} \nu_N\!\insquare{y^2} } \leq \textnormal{I}' + \textnormal{II}',
    \end{align*}
    where $\textnormal{I}' := \abs{  \nu_N\!\insquare{x^2 y^2} - \HH^{\otimes 2} \insquare{x^2 y^2}  }$ and $\textnormal{II}' := \abs{  \inparen{ \nu_N\insquare{x^2} }^2 - \inparen{ \HH\!\insquare{x^2} }^{2}   }$. With analogous arguments as in \eqref{eq:LI=>TSOC_boundforI} and \eqref{eq:LI=>TSOC_boundforII}, we obtain respectively
    \begin{align*}
        \textnormal{I}' \leq \frac{K \sqrt{T_N}}{N^{\eta/2}}, \quad \textnormal{and} \quad \textnormal{II}' \leq \frac{K S_N}{N^{\eta/ 2}}.
    \end{align*}
    This finally yields $\Var R_{11} \leq S_N/N + K(\sqrt{T_N} + S_N)/N^{\eta/2}$ which finishes the proof. 
\end{proof}

\section{Sherrington-Kirkpatrick model---classical and \texorpdfstring{$[-1,1]$}{[-1,1]}-variant}
\label{sec:SK_classical_[-1,1]_variant}

\subsection{Decomposing the classical-SK Hamiltonian}

Recall the SK Hamiltonian in \eqref{eq:intro_SK_Hamiltonian}. We assume that in the external field $h \sum_{i \leq N} x_i$ $h \in \RR$ is fixed, but it is straightforward to extend the following results to the case when the external field is $\sum_{i \leq N} h_i x_i$ where $h_i$'s are iid with randomness independent of all other parts of the disorder.

In this section we apply the local independence framework Theorem \ref{thm:LI_abstract_theorem} to the classical SK model, thus formalizing the outline in Section \ref{sec:intro_outline_LI}.

The purpose of the following interpolation is to get rid of the interactions among the coordinates $x_1,x_2,\dots,x_k$ that we are considering for local independence.

Define the interpolating Hamiltonian, for $0 \leq t \leq 1$,
\begin{align}
    -H_{N,t}(\vec{x}) &= \frac{\beta}{\sqrt{N}}\sum_{k+1 \leq i,l \leq N} g_{il}x_i x_l + \sum_{j \leq k} x_j \sum_{k+1 \leq i \leq N} \frac{\beta g_{j,i}}{\sqrt{N}} x_i + \sqrt{t} \frac{\beta}{\sqrt{N}}\sum_{j < j' \leq k} g_{j,j'}x_j x_{j'} + \sum_{i \leq N} h x_i, 
    \label{eq:SK_H_N,t}
\end{align}
with associated Gibbs measures denoted by $\inangle{\cdot}_t$, and also $\nu_t\insquare{\cdot} = \Eb \inangle{\cdot}_t$. Note that $\nu_1 = \nu$.

\begin{proposition}
Let $f$ be a function on $\Sigma_N^n$. Suppose that $f$ satisfies $\abs{f} \leq 1$, and that $f$ is a possibly random function, but independent of the disorder r.v.'s $\inparen{ g_{j,j'} }_{j < j' \leq k}$. Then
\begin{align*}
    \abs{\nu \insquare{ f} - \nu_0 \insquare{ f }} \leq \frac{K(n,\beta) k^2}{N}.
\end{align*}
\label{proposition:SK_t-interpolation}
\end{proposition}

\begin{proof}
By symmetry among the pairs $(x_j, x_{j'})$, $j < j' \leq k$, we have
\begin{align*}
     \frac{\ud}{\ud t} \nu_t \insquare{ f }  &= \frac{\beta}{2\sqrt{Nt}} \binom{k}{2} \inparen{ \sum_{\ell \leq n} \nu_t\insquare{ g_{1,2} x_1^\ell x_2^\ell f  } - n \nu_t\insquare{ g_{1,2} x_1^{n+1} x_2^{n+1} f  }  }.
\end{align*}
We further apply GIPF to simplify each term. Write $\nu_t\insquare{   g_{1,2}x_1^\ell x_2^\ell f   } = \Eb \insquare{g_{1,2} \inangle{x_1^\ell x_2^\ell f}_t }$, and we regard $\inangle{x_1^\ell x_2^\ell f}_t$ as some function $F$ of $g_{1,2}$. Actually, $F$ is a function of all $g_{i,j}$'s, but due to zero-mean and independence, we simply have $\Eb \insquare{ g_{1,2} F(g_{1,2}) } = \Eb F'(g_{1,2}) $. For $\ell \leq n$, this yields
\begin{align*}
    \nu_t\insquare{   g_{1,2}x_1^\ell x_2^\ell f   } &= \frac{\beta\sqrt{t}}{\sqrt{N}} \inparen{  \sum_{\ell' \leq n} \nu_t\insquare{   x_1^\ell x_1^{\ell'} x_2^{\ell} x_2^{\ell'} f   } - n\nu_t \insquare{   x_1^\ell x_1^{n+1} x_2^{\ell} x_2^{n+1} f   }    },
\end{align*}
and for $\ell = n+1$,
\begin{align*}
    \nu_t\insquare{ g_{1,2} x_1^{n+1} x_2^{n+1} f  } &= \frac{\beta\sqrt{t}}{\sqrt{N}} \inparen{  \sum_{\ell' \leq n+1} \nu_t\insquare{   x_1^{n+1} x_1^{\ell'} x_2^{n+1} x_2^{\ell'} f   } - (n+1)\nu_t \insquare{   x_1^{n+1} x_1^{n+2} x_2^{n+1} x_2^{n+2} f   }    }.
\end{align*}
The result follows by bounding $\nu_t\insquare{   x_1^\ell x_1^{\ell'} x_2^{\ell} x_2^{\ell'} f   } \leq \nu_t \abs{f} \leq 1$, and substituting the above into the expression for $\ud \nu_t \insquare{f} /\ud t$.
\end{proof}

The following elementary identity is used repeatedly throughout this paper. It follows by rearranging he binomial expansion of $(1-1)^n$.
\begin{lemma}
    For $n$ even,
\begin{align}
    \binom{n}{0} = -\binom{n}{n} + \binom{n}{n-1} - \binom{n}{n-2} + \cdots - \binom{n}{2} + \binom{n}{1}.
    \label{eq:(1-1)^n_expand_identity}
\end{align}
\end{lemma}

\begin{theorem}
Let $f$ be a function on $\Sigma_{N}$ with $\abs{f}\leq 1$, then
\begin{align*}
    \Eb \insquare{ \abs{ \inangle{f} - \inangle{f}_0  }^{2p} } \leq \frac{K(p, \beta)k^2}{N}.
\end{align*}
\label{thm:SK_classical_decomposition}
\end{theorem}

\begin{proof}
Denote $f^\ell := f(\vec{x}^\ell)$. Using \eqref{eq:(1-1)^n_expand_identity} and replicas, we have
\begin{align}
    \Eb_{\textnormal{d}} \insquare{\abs{ \inangle{f} - \inangle{f}_{0}  }^{2p}  } &= \sum_{1 \leq r \leq 2p} \inparen{-1}^{2p - r} \binom{2p}{r} \Eb_{\textnormal{d}} \insquare{ \inangle{f^{1} \cdots f^{r} } \inangle{ f^{1}\dots f^{2p - r}  }_{0} - \inangle{ f^{1} \cdots f^{2p} }_{0}  }.
    \label{eq:SK_quenchedCloseness_betweenoriginal_and_decomposed_expansion}
\end{align}
For every $1 \leq r \leq 2p$, write
\begin{align*}
    &\abs{ \Eb_{\textnormal{d}} \insquare{ \inangle{f^{1} \cdots f^{r} } \inangle{ f^{1}\dots f^{2p - r}  }_{0} - \inangle{ f^{1} \cdots f^{2p} }_{0}  } }\\
    &\quad = \abs{ \Eb_{\textnormal{d}} \insquare{ \inangle{f^{1} \cdots f^{r} \inangle{ f^{1}\dots f^{2p - r}  }_{0} }  - \inangle{ f^{1} \cdots f^{r} \inangle{ f^{1}\dots f^{2p - r}  }_{0} }_{0}  } }\\
    &\quad \leq \frac{K(r,\beta) k^2}{N},
\end{align*}
where the inequality follows from Proposition \ref{proposition:SK_t-interpolation} applied with the function $F$ on $\Sigma_{N}^{r}$ defined by $F(\vec{x}^1,\dots,\vec{x}^r) := f^{1} \cdots f^{r} \inangle{ f^{1}\dots f^{2p - r}  }_{0}$. Note that $\abs{F} \leq 1$ and that $F$ is independent of the disorder r.v.'s $\inparen{ g_{j,j'} }_{j < j' \leq k}$. The result then follows from \eqref{eq:SK_quenchedCloseness_betweenoriginal_and_decomposed_expansion}.
\end{proof}

\subsection{Local independence for classical SK}

From \eqref{eq:SK_H_N,t}, we define the following Hamiltonian on $\Sigma_{N-k}$:
\begin{align}
    -H_{N-k}^-(x_1,\dots,x_k) &= \frac{\beta^-}{\sqrt{N-k}} \sum_{k+1 \leq i \leq N} g_{ij} x_i x_j + h\sum_{k+1 \leq i \leq N} x_i,
    \label{eq:SK_truncated_system}
\end{align}
whose Gibbs expectations are denoted $\inangle{\cdot}^-$, and where
\begin{align}
    \beta^- &:= \beta \sqrt{ \frac{N-k}{N} }.
    \label{eq:SK_beta_relation}
\end{align}
Note that $\beta^- < 1/2$ whenever $\beta < 1/2$. It follows immediately that we have overlap concentration for the truncated $(N-k)$-system.
\begin{theorem}[\cite{talagrand2010mean} Equation 1.89]
For every $\beta < 1/2$, it holds that
\begin{align}
    \Eb\insquare{ \inangle{ \inparen{ \frac{1}{N-k} \sum_{k+1 \leq i \leq N} x_i^1 x_i^2 - q^-}^2 }^-  }  \leq \frac{K}{N-k},
\end{align}
where $q^-$ satisfies $q^- = \Eb \tanh^2\inparen{ \beta^- \sqrt{q^-} z + h  }$.
\label{thm:SK_OC}
\end{theorem}

\begin{theorem}[Local independence, classical SK]
Let $\beta < 1/2$ and $h \neq 0$. Denote by $G_N^{(k)}$ the SK Gibbs marginal on the first $k$-coordinates. Then for every $p \geq 1$, for every $0 < \epsilon < 1/2$, whenever $N \geq \exp\inparen{ \sqrt{16k}/\epsilon  }$,
\begin{enumerate}
    \item we have
    \begin{align}
        \Eb_{\ud}\insquare{ \sup_{ (\sigma_1,\dots,\sigma_k) \in \Sigma_k } \inparen{ G_N^{(k)} \inparen{ \sigma_1,\dots,\sigma_k }   - \prod_{j \leq k} \frac{\exp \sigma_j \inparen{ \beta \vec{a}_j^{\top} \inangle{\vec{y}}^- + h  } }{ 2\cosh \inparen{\beta \vec{a}_j^{\top} \inangle{\vec{y}}^- + h}  } }^{2p} }  \leq \frac{K}{N^{\frac{1}{2} - \epsilon}  },
        \label{eq:SK_classical_LI_statement_partialLimiting}
    \end{align}
    where $\vec{a}_j := \inparen{\frac{1}{\sqrt{N}} g_{j,i}  }_{k+1 \leq i \leq N}$;
    \item furthermore, for independent standard Gaussians $(z_j)_{j \leq k}$, independent of everything else, 
    \begin{align}
        \Eb_{\ud} \Eb_{\vec{z}}\insquare{ \sup_{(\sigma_1,\dots,\sigma_k) \in \Sigma_k} \inparen{ G_N^{(k)} \inparen{ \sigma_1,\dots,\sigma_k } - \prod_{j \leq k} \frac{\exp \sigma_j \inparen{ \beta^- \sqrt{q^-}z_j + h  } }{ 2\cosh \inparen{\beta^- \sqrt{q^-}z_j + h}  } }^{2p}  } \leq \frac{K}{N^{\frac{1}{2} - \epsilon}  },
        \label{eq:SK_classical_LI_statement_Limiting}
    \end{align}
\end{enumerate}
where $q^-$ is from Theorem \ref{thm:SK_OC}, and for constants $K$ depending on $\beta, h, k, p$ and $\epsilon$.
\label{thm:SK_classical_localIndependence}
\end{theorem}
\begin{proof}
From Theorem \ref{thm:SK_classical_decomposition}, the SK model is $(K/N)$-decomposable with a truncated system given by \eqref{eq:SK_truncated_system}. The result follows by applying Theorem \ref{thm:LI_abstract_theorem} with $\varrho = \beta^-$, $\vec{A}_j = \inparen{  \frac{1}{\sqrt{N-k}} g_{j,i}  }_{k + 1 \leq i \leq N}$, $\vec{w} = (x_{k+1},\dots,x_N)$, $f_j(x_j) = h x_j$, $\Xi = 1$, $\Upsilon = q^-$, $M = N-k$, and $B = \inbraces{  \vec{x} \in \Sigma_N : x_j = \sigma_j, j \leq k   }$, $\mu$ is uniform measure on $\inbraces{\pm 1}$. Observe that all hypotheses are satisfied---hypothesis 1 is satisfied by \eqref{thm:SK_OC}, hypothesis 2 is satisfied since for any $P \in \RR$,
\begin{align*}
    \int_{\RR} \exp\inparen{ \frac{x_j^2 \varrho^2}{2}(\Xi - \Upsilon) + f_j(x_j) + x_j P  } \mu(\ud x_j) = \exp \inparen{  \frac{\beta^{-,2}}{2}(1-q^-)  } \cosh(P + h) \geq 1.
\end{align*}
\end{proof}

\begin{theorem}[Converse to local independence, classical SK]
Suppose that \eqref{eq:SK_classical_LI_statement_Limiting} holds for $k \leq 2$, $p = 1$. Let $\nu_N \insquare{\cdot} := \Eb_{\ud} \inangle{\cdot}$, and choose any $0 < \delta < 1/8$. Then
\begin{align}
    \nu_N\!\insquare{ \inparen{ R_{12} - \nu_N R_{12}  }^2   } \leq \frac{K}{N^{1/8 - \delta}},
    \label{eq:SK_classical_converseToLI}
\end{align}
for some constant $K$ that depends on $\beta, h$, and $\delta$.
\label{thm:SK_classical_conversetoLI}
\end{theorem}

\begin{proof}
    Apply Theorem \ref{thm:LI=>TSOC_abstract} with $\eta = 1/4 - 2\delta$, $\HH_j (\sigma_j) \propto \exp \sigma_j (\beta \sqrt{q^- } z_j + h)$ (defined conditionally on $z_j$). Note that $\nu_N^{(1)} x^4 \equiv 1$ and $\HH x^4 \equiv 1$.
\end{proof}

\begin{remark}
    Theorem \ref{thm:SK_classical_conversetoLI} should be compared to \cite{talagrand2010mean} Exercise 1.7.2---two hypotheses are needed for overlap concentration: that $\vec{x}$ is typically locally independent conditioned on the disorder, and also that $\inangle{\vec{x}}$ is also locally independent under the disorder. 
    
    In our case we also used that $\vec{x}$ is typically locally independent conditioned on disorder, but access to an explicit form of the marginal densities as in \eqref{eq:SK_classical_LI_statement_Limiting} (and \eqref{eq:LI_abstract_HHj_statement}) allows one to deduce that the mixture of products $\Eb_\vec{z} \HH_1 \otimes \HH_2$ is also a product measure so that \eqref{eq:LI=>TSOC_HHj_is_uncond_a_product_measure_hypothesis} obtains.

    Note that Talagrand is able to attain sharper rates and results---showing that $R_{12} \simeq q$, where $q = \Eb \tanh^2(\beta \sqrt{q} z + h)$, using SK specific arguments. The insight our results afford, in using the abstract results Theorem \ref{thm:LI_abstract_theorem} and Theorem \ref{thm:LI=>TSOC_abstract}, is to illustrate the apparent generality of the phenomenon that overlap concentration is in a sense equivalent to local independence. 
\end{remark}

\subsection{Decomposing the \texorpdfstring{$[-1,1]$}{[-1,1]}-SK Hamiltonian}
\label{section:SK_[-1,1]_decompose}

Consider now a variant of the classical SK model, where the $\inbraces{\pm 1}$ spins are replaced by spins living in $[-1,1]$. In fact, this model (and also the classical SK model) is a special case of the SK model with $d$-component spins (\cite{talagrand2010mean} Section 1.12), where $d = 1$.

In particular, let $\mu$ be the measure on $[-1,1]$ with density $\mu(\ud x) = \exp(xh) \tilde{\mu}(\ud x)$, where $\tilde{\mu}$ is the normalized uniform measure over $[-1,1]$, and where $h \in \RR$. The $[-1,1]$-SK model has Hamiltonian and Gibbs measure given by
\begin{align}
    -H_N(\vec{x}) &= \frac{\beta}{\sqrt{N}}\sum_{i < j} g_{ij} x_i x_j; \quad\quad \inangle{f} = \frac{1}{Z_N} \int f(\vec{x}) \exp(-H_N(\vec{x})) \mu^{\otimes N} (\ud \vec{x}),
\end{align}
for $f : [-1,1]^N \rightarrow \RR$. One sees that this is very close to the SK model, with the external field absorbed into $\mu$, and having an integral instead of a sum over the discrete hypercube. In fact, since $x_i \in [-1,1]$, we still have $\abs{x_i} \leq 1$, and arguments for Proposition \ref{proposition:SK_t-interpolation} and Theorem \ref{thm:SK_classical_decomposition} carry forward with almost no modification. We obtain the following.
\begin{theorem}
Let $f$ be a function on $\insquare{-1,1}^N$ with $\abs{f}\leq 1$, then
\begin{align*}
    \Eb \insquare{ \abs{ \inangle{f} - \inangle{f}_0  }^{2p} } \leq \frac{K(p, \beta)k^2}{N}.
\end{align*}
\label{thm:SK_[-1,1]_decomposition}
\end{theorem}

\subsection{Local independence for \texorpdfstring{$[-1,1]$}{[-1,1]}---SK model}

Using relation \eqref{eq:SK_beta_relation}, we have $\beta^- < \beta$ so that when thin-shell and overlap concentration hold for the original system when $L \beta \leq 1$, it holds also for the truncated system at temperature $\beta^-$.

\begin{theorem}[\cite{talagrand2010mean} Theorem 1.12.2]
Let $L \beta \leq 1$ for some universal constant $L$, then
\begin{align}
    \Eb\insquare{  \inparen{ \frac{1}{N-k} \sum_{k+1 \leq i \leq N} x_i^{1,2} - \rho^-  }^{2}  } &\leq \frac{K}{N-k},
\end{align}
and
\begin{align}
    \Eb\insquare{  \inparen{ \frac{1}{N-k} \sum_{k+1 \leq i \leq N} x_i^{1}x_i^2 - q^-  }^{2}  } &\leq \frac{K}{N-k},
\end{align}
where $q^-$ and $\rho^-$ satisfy, for $\cE(x) := \exp \inparen{ x\beta^-\sqrt{q^-} + \frac{\beta^{-,2}}{2}x^2 (\rho^- - q^-)  }$,
\begin{align}
    q^- &= \Eb \insquare{ \frac{1}{Z^2} \inparen{\int x \cE(x) \mu(\ud x) }^2  }; \quad\quad \rho^- = \Eb \insquare{ \frac{1}{Z} \int x^2 \cE(x) \mu(\ud x)   }.
    \label{eq:SK_[-1,1]_q^-,rho^-}
\end{align}
\end{theorem}

\begin{theorem}[Local independence $\insquare{-1,1}$---SK model]
Let $L \beta \leq 1$ for some universal constant $L$ and let $h \neq 0$. Denote by $G_{N}^{(k)}$ the Gibbs marginal on the first $k$-coordinates  Then for every $p \geq 1$, for every $0 < \epsilon < 1/2$, whenever $N \geq \exp\inparen{ \sqrt{16k}/\epsilon  }$,
\begin{enumerate}
    \item we have 
    \begin{align}
        \Eb_{\ud}\insquare{  \sup_{B \subseteq \RR^k} \inparen{ G_{N}^{(k)}\!\insquare{B} -  \frac{1}{Z_k} \int_B \prod_{j \leq k} \exp\inparen{ x_j \beta \vec{a}_j^{\top} \inangle{\vec{y}}^-  + \frac{x_j^2 \beta^{-,2}}{2}(\rho^- - q^-)  } \mu (\ud x_j)  }^{2p}  } \leq \frac{K}{N^{\frac{1}{2} - \epsilon}  },
        \label{eq:SK_[-1,1]_LI_statement_partialLimiting}
    \end{align}
    where $\vec{a}_j := \inparen{\frac{1}{\sqrt{N}} g_{j,i}  }_{k+1 \leq i \leq N}$;
    \item furthermore, for independent standard Gaussians $(z_j)_{j \leq k}$ independent of everything else,
    \begin{align}
        \Eb_{\ud} \Eb_{\vec{z}}\insquare{  \sup_{B \subseteq \RR^k} \inparen{ G_{N}^{(k)}\!\insquare{B} -  \frac{1}{Z_k} \int_B \prod_{j \leq k} \exp\inparen{ x_j \beta^- \sqrt{q^-} z_j  + \frac{x_j^2 \beta^{-,2}}{2}(\rho^- - q^-)  } \mu (\ud x_j)  }^{2p}  } \leq \frac{K}{N^{\frac{1}{2} - \epsilon}  },
        \label{eq:SK_[-1,1]_LI_statement_Limiting}
    \end{align}
\end{enumerate}
where $Z_k$ is a normalizing constant, and where $\rho^-$, $q^-$ satisfy \eqref{eq:SK_[-1,1]_q^-,rho^-}, and where $K$ are constants depending on $\beta, h, p, k$, and $\epsilon$.
\end{theorem}

\begin{proof}
The proof is similar to that of Theorem \ref{thm:SK_classical_localIndependence}. Theorem \ref{thm:LI_abstract_theorem} is applied with the same definitions, except that $\Xi = \rho^-$ (it is no longer true that $x_i^2 \equiv 1$), and with $f_j(x_j) = hx_j$, we set $\mu$ (in Theorem \ref{thm:LI_abstract_theorem}) to be $\tilde{\mu}$: the normalized uniform measure over $[-1,1]$. We may write, for any $P \in \RR$,
\begin{align*}
    &\int_{\RR} \exp\inparen{ \frac{x_j^2 \varrho^2}{2}(\Xi - \Upsilon) + f_j(x_j) + x_j P  } \tilde{\mu}(\ud x_j) \\
    &\quad= \int_{[-1,0]} \exp \inparen{ \frac{x_j^2 \beta^{-,2}}{2}(\rho^- - q^-) + x_j(P + h)  } \tilde{\mu}(\ud x_j) + \int_{[0,1]} \exp \inparen{ \frac{x_j^2 \beta^{-,2}}{2}(\rho^- - q^-) + x_j(P + h)  } \tilde{\mu}(\ud x_j),
\end{align*}
and observe that if $P + h \geq 0$, then the second integral on the RHS above is $\geq 1$, otherwise if $P + h < 0$, then the first integral is $\geq 1$. Since both integrals are nonnegative, hypothesis 2 follows.
\end{proof}

\begin{theorem}[Converse to local independence, $\insquare{-1,1}$---SK]
Suppose that \eqref{eq:SK_[-1,1]_LI_statement_Limiting} holds for $k \leq 2$, $p = 1$. Let $\nu_N \insquare{\cdot} := \Eb_{\ud} \inangle{\cdot}$, and choose any $0 < \delta < 1/8$. Then
\begin{align}
    \nu_N\!\insquare{ \inparen{ R_{12} - \nu_N R_{12}  }^2   } \leq \frac{K}{N^{1/8 - \delta}}, \quad \textnormal{and} \quad \nu_N\!\insquare{ \inparen{ R_{11} - \nu_N R_{11}  }^2   } \leq \frac{K}{N^{1/8 - \delta}}
    \label{eq:SK_[-1,1]_converseToLI}
\end{align}
for some constant $K$ that depends on $\beta, h$, and $\delta$.
\label{thm:SK_[-1,1]_conversetoLI}
\end{theorem}

\begin{proof}
    Apply Theorem \ref{thm:LI=>TSOC_abstract} with $\eta = 1/4 - 2\delta$, $\HH_j (\ud x_j) \propto \exp \inparen{ x_j \beta^- \sqrt{q^- } z_j + x_j^2 \beta^{-,2} (\rho^- - q^-)/2   } \mu(\ud x_j)$ (defined conditionally on $z_j$). Note that $\nu_N^{(1)} x^4, \nu_N^{(1)} x^8 \leq 1$ and $\HH x^4, \HH x^8 \leq 1$.
\end{proof}

\section{Ising Perceptron model}
\label{sec:perceptron}

The Perceptron model (\cite{talagrand2010mean} Chap.~2) originated in the context of neural networks, and was conceived to address the Gardner problem on the discrete hypercube. The Perceptron model is also connected with posterior distributions of generalized linear models, for certain nonlinear functions $u$ (definitions below).

The Hamiltonian associated with the $N$-system Ising Perceptron model is given by
\begin{align}
    -H_{N,M}(\vec{x}) &:= \sum_{m \leq M} u\inparen{ \frac{1}{\sqrt{N}} \sum_{i \leq N} g_{i,m} x_i  }, \quad \vec{x} \in \Sigma_{N}, \label{eq:perceptron_Hamiltonian}
\end{align}
where $M$ is an integer, $(g_{i,m})_{i \leq N, m \leq M}$ are standard Gaussian r.v.'s, and $u$ is a smooth function on $\RR$ satisfying
\begin{align}
    \abs{u^{(d)}} \leq D, \quad 0 \leq d \leq 3,
    \label{eq:perceptron_|u^(d)|<=D}
\end{align}
for some number $D > 0$. The regime of interest is when $M$ grows proportionally with $N$, with $\alpha := M/N$ denoting the ratio. These are the conditions used in \cite{talagrand2010mean} Chap.~2; remark that there exists several variants of this model that are well-studied, e.g.~the spherical 
Perceptron \cite{gardner1987maximum}, \cite{franz2016simplest}, \cite{el2022algorithmic}, \cite{bolthausen2022gardner} (this last giving an excellent history of the Perceptron).

Let us describe the general approach to showing local independence for the Perceptron Gibbs measure (with similar techniques used for the Shcherbina-Tirozzi model in Section \ref{sec:Shcherbina-Tirozzi}. As before it is easier to consider the $k=2$ case first. To compute the Gibbs marginal on the coordinates $(x_1,x_2)$, we first have to coerce the Hamiltonian into the decomposed form (Definition \ref{definition:LI_abstract_deecomposableHam}), which isolates cavity fields associated with $x_1$ and $x_2$. However, since the spins are surrounded by the function $u$, this decomposition is not immediate. Instead, we will use interpolation and Gaussian integration by parts to show that the Gibbs measure associated with the decomposed Hamiltonian is not too far from the original Gibbs measure.

The following notation will be adopted for the rest of the Perceptron and Shcherbina-Tirozzi models: for $\ell \geq 1$,
\begin{align}
    S_m^{\ell} &:= \frac{1}{\sqrt{N}} \sum_{i \leq N} g_{i,m} x_i^\ell; \quad\quad S_m^{0,\ell} := \frac{1}{\sqrt{N}} \sum_{k+1 \leq i \leq N} g_{i,m} x_i^{\ell},
    \label{eq:perceptron_definition_of_S_m's}
\end{align}
so that $-H_{N,M}(\vec{x}) = \sum_{m \leq M} u(S_m)$, and $S_m = S_m^0 + N^{-1/2} g_{1,m}x_1 + N^{-1/2} g_{2,m}x_2$ (for $k=2$). A remarkable phenomenon in the Gardner models is that while we are deliberately modeling a spin system in the variables $x_i$, there is an auxiliary spin system in the variables $S_m$. In high-temperature, we expect overlap concentration, and this manifests in terms of
\begin{align}
    \frac{1}{N} \sum_{i \leq N} x_i^1 x_i^2 &\simeq q; \quad\quad \frac{1}{N} \sum_{m \leq M} u'\inparen{S_m^1} u'(S_m^2) \simeq r, \label{eq:perceptron_OC_main_and_auxiliary_intuition}
\end{align}
for some numbers $q$ and $r$, to be understood as the limiting overlap values for the two spin systems. The equations for $q$ and $r$ are coupled, and are called the replica-symmetric equations \eqref{eq:perceptron_RS_equations}. The scaling $1/N$ in \eqref{eq:perceptron_OC_main_and_auxiliary_intuition} could be replaced by $1/M$, we would simply have to replace the RHS by $r/\alpha$.

To elucidate an appropriate decomposed form, a Taylor expansion provides some insight:
\begin{align}
    \sum_{m \leq M} u\inparen{ S_m  } &= \sum_{m \leq M} u\inparen{ S_m^0  } + \sum_{m \leq M} \inparen{ x_1\frac{g_{1,m}}{\sqrt{N}} + x_2\frac{g_{2,m}}{\sqrt{N}}  } u'\inparen{ S_m^0  } \nonumber\\
     &\quad\quad\quad\quad + \frac{1}{2} \sum_{m \leq M} \inparen{ x_1\frac{g_{1,m}}{\sqrt{N}} + x_2\frac{g_{2,m}}{\sqrt{N}}  }^{2} u''\inparen{ S_m^0  } + \cdots. 
    \label{eq:perceptron_taylorExpansion}
\end{align}
The third term in \ref{eq:perceptron_taylorExpansion} equates to
\begin{align*}
    \frac{1}{2N} \sum_{m \leq M} g_{1,m}^2 u''\inparen{ S_m^0  } + \frac{x_1 x_2}{N} \sum_{m \leq M} g_{1,m}g_{2,m} u''\inparen{ S_m^0  } + \frac{1}{2N} \sum_{m \leq M} g_{2,m}^2 u''\inparen{ S_m^0  }.
\end{align*}
If the intuition in \ref{eq:perceptron_OC_main_and_auxiliary_intuition} is correct, we expect a type of local independence among the $S_m$'s (and also $S_m^0$'s), so that a law of large numbers kicks in, and the above quantity is close to a constant for $N$ large, whence the third term in \ref{eq:perceptron_taylorExpansion} can be absorbed into the normalizing constant. The terms $\cdots$ are of smaller order and can be neglected. It follows that a reasonable decomposition of the Perceptron Hamiltonian \eqref{eq:perceptron_Hamiltonian} is
\begin{align*}
    -H_{N,M,0}(\vec{x}) = \sum_{m \leq M} u\inparen{ S_m^0  } + x_1 \sum_{m \leq M}  \frac{g_{1,m}}{\sqrt{N}} u'\inparen{ S_m^0 } + x_2 \sum_{m \leq M}  \frac{g_{2,m}}{\sqrt{N}} u'\inparen{ S_m^0 },
\end{align*}
where the $0$ subscript indicates that this is the desirable end of our interpolating Hamiltonian $-H_{N,M,t}$, which will be defined shortly. 

Denote $\vec{w} = \inparen{ u'(S_m^0) }_{m \leq M}$. Observe that under the cavity system, that is, the $N-2$ system induced by the Hamiltonian $\sum_{m \leq M} u(S_m^0)$, the vector $\vec{w}$ is independent of the randomness in $(g_{1,m})_{m \leq M}$ and $(g_{2,m})_{m \leq M}$. It follows from the general principle about Gaussian projections of random vectors with thin-shell and overlap concentration (Theorem \ref{thm:ProjResult_DisorderedCase_supOverBL_LM_2p}) that if we additionally have
\begin{align}
    \frac{1}{N} \sum_{m \leq M} u'^{2}\inparen{S_m} \simeq \tau,
\end{align}
then we expect, with $A_{i,m} \equiv N^{-1/2} g_{i,m}$ and $\vec{A}_i = (A_{i,m})_{m \leq M}$, that under the truncated system,
\begin{align*}
    \cL\inparen{ \vec{A}_1^{\top} \vec{w}, \vec{A}_2^{\top} \vec{w} \; | \; \vec{A}  } \simeq \cN\inparen{ \begin{bmatrix}
    \vec{A}_1^{\top} \inangle{\vec{w}} \\ \vec{A}_2^{\top} \inangle{\vec{w}}
    \end{bmatrix}, \, \begin{bmatrix}
    \tau - r & 0 \\
    0 & \tau - r
    \end{bmatrix}},
\end{align*}
from which the quenched local independence for the Gibbs marginal on $(x_1, x_2)$ will follow. As in the case of $\pm 1$ models such as the classical SK model, the $\tau - r$ that is attached to $\xi$ will cancel in the numerator and denominator. (This will not be the case in the Shcherbina-Tirozzi model)

\subsection{Decomposing the Ising Perceptron Hamiltonian}
\label{sec:perceptron_decomposeHamiltonian}

The interpolating Hamiltonian is defined as follows: for $0 \leq t \leq 1$,
\begin{align}
    -H_{N,M,t}(\vec{x}) &= \sum_{m \leq M} u\inparen{ \frac{1}{\sqrt{N}} \sum_{k+1 \leq i \leq N} g_{i,m}x_i + \sqrt{\frac{t}{N}} \sum_{j \leq k}  x_j g_{j,m}    } \nonumber \\
    &\quad\quad\quad\quad\quad\quad + \sqrt{\frac{1-t}{N}} \sum_{m \leq M} \inparen{ \sum_{j \leq k} x_j \tilde{g}_{j,m}    } u'\inparen{  \frac{1}{\sqrt{N}} \sum_{k+1 \leq i \leq N} g_{i,m}x_i   },
    \label{eq:perceptron_tInterpolatingHamiltonian}
\end{align}
where $(\tilde{g}_{j,m})_{j \leq k, m \leq M}$ are independent standard Gaussian r.v.'s, independent of all other sources of randomness. One notes that this additional source of disorder only really appears when $0 < t < 1$. At the endpoint $t = 0$, the disorder is in the r.v.'s $(g_{i,m})_{k \leq i \leq N, m \leq M}$ and $(\tilde{g}_{j,m})_{j \leq k, m \leq M}$, and it is only a matter of renaming the $\tilde{g}_{j,m}$ variables into $g_{j,m}$. The purpose of introducing the $\tilde{g}_{j,m}$'s is to avoid the appearance of difficult cross-terms in the interpolation.

Gibbs expectations wrt.~this Hamiltonian are denoted by $\inangle{\cdot}_t$. More precisely, for any function $f: \Sigma_N^n \rightarrow \RR$, we define
\begin{align*}
    \nu_t \insquare{f} = \Eb \inangle{f}_t = \Eb \frac{\sum_{\vec{x}^1, \dots, \vec{x}^n} f(\vec{x}^1,\dots, \vec{x}^n) \exp\inparen{ - \sum_{\ell \leq n} H_{N,M,t}^\ell }  }{ \sum_{\vec{x}^1, \dots, \vec{x}^n} \exp\inparen{ - \sum_{\ell \leq n} H_{N,M,t}^\ell }},
\end{align*}
where we denote $-H_{N,M,t}^\ell := -H_{N,M,t}(\vec{x}^{\ell})$. We also introduce the notation, for $\ell \geq 1$, $0 \leq t \leq 1$, $m \leq M$ (recalling \eqref{eq:perceptron_definition_of_S_m's}:
\begin{align*}
    S_{m,t}^{\ell} &:= S_{m}^{0,\ell} + \sqrt{\frac{t}{N}} \sum_{j \leq k} x_j g_{j,m}.
\end{align*}

The following $t$-interpolation is by a mechanical application of GIPF, similar to the proof of Proposition \ref{proposition:SK_t-interpolation}, the proof is given in Appendix \ref{sec:suppProofs_for_perceptron}.

\begin{proposition}
Let $f$ be a function on $\Sigma_N^n$. Then
\begin{align*}
    \frac{\ud}{\ud t} \nu_t \insquare{f} &= \textnormal{I} + \textnormal{II} + \textnormal{III},
\end{align*}
where
\begin{align*}
    \textnormal{I} &:= \frac{\alpha k}{2} \inparen{ \sum_{\ell \leq n} \nu_t \insquare{u''\inparen{ S_{M,t}^{\ell}} f } -n \nu_t\insquare{ u''\inparen{S_{M,t}^{n+1}} f   }  } \\
    \textnormal{II} &:= \alpha k \left( \sum_{\ell < \ell' \leq n } \inparen{  \nu_t\insquare{ x_1^\ell x_1^{\ell'} u'\inparen{ S_{M,t}^\ell  } u'\inparen{ S_{M,t}^{\ell'} } f  } - \nu_t\insquare{ x_1^\ell x_1^{\ell'} u'\inparen{ S_{M}^{0,\ell}  } u'\inparen{ S_{M}^{0,\ell'} } f  } }  \right.\\
    &\quad\quad\quad\quad -n\sum_{\ell \leq n} \inparen{  \nu_t\insquare{ x_1^\ell x_1^{n+1} u'\inparen{ S_{M,t}^\ell  } u'\inparen{ S_{M,t}^{n+1} } f  } - \nu_t\insquare{ x_1^\ell x_1^{n+1} u'\inparen{ S_{M}^{0,\ell}  } u'\inparen{ S_{M}^{0,n+1} } f  } } \\
    &\quad\quad\quad\quad \left. + \frac{n(n+1)}{2} \inparen{ \nu_t\insquare{ x_1^{n+1} x_1^{n+2} u'\inparen{ S_{M,t}^{n+1}  } u'\inparen{ S_{M,t}^{n+2} } f  } - \nu_t\insquare{ x_1^\ell x_1^{n+1} u'\inparen{ S_{M}^{0,n+1}  } u'\inparen{ S_{M}^{0,n+2} } f  }  } \!\phantom{\Bigg\rvert}\! \right) \\
    \textnormal{III} &:= \frac{\alpha k}{2} \inparen{  \sum_{\ell \leq n+1} \inparen{ \nu_t\insquare{ u'^{2}\inparen{  S_{M,t}^{\ell}  }  f }    -    \nu_t\insquare{ u'^{2}\inparen{  S_{M}^{0,\ell}  }  f }  } }.
\end{align*}
\label{proposition:perceptron_t_interpolation}
\end{proposition}

Let us explain the intuition for how Proposition \ref{proposition:perceptron_t_interpolation} may be used; that is, how the RHS might be small. Terms $\textnormal{II}$ and $\textnormal{III}$ follow the same mechanism as follows. If we can create a cavity-in-$M$ system (see \cite{talagrand2010mean} Section 2.3), that is, a system that is independent of the ``randomness in $M$'': $(g_{i,M})_{i \leq N}$, then the law of the map $\vec{x}^{\ell} \mapsto N^{-1/2} \sum_{k \leq i \leq N} g_{i,M} x_i^{\ell}$ under this cavity-in-$M$ system is approximately $\cL\inparen{ \sqrt{q}z + \sqrt{1 - q}\xi^{\ell} \, | \, z}$, where $z$ and $\xi^\ell$ are independent standard Gaussians. This again arises from the general principle about Gaussian projections of random vectors (in this case $(x_i)_{k \leq i \leq N}$) with thin-shell and overlap concentration.

The law of the map $\vec{x}^{\ell} \mapsto N^{-1/2} \sum_{k \leq i \leq N} g_{i,M} x_i^{\ell} + (tN)^{-1/2} \sum_{j \leq k} x_j^{\ell} g_{j,M}$ under the cavity-in-$M$ system will be asymptotically similar, since the additional term involving $t$ is negligible for large $N$. In other words
\begin{align*}
    \nu_t \insquare{ u'\inparen{ S_{M,t}^{\ell}  } u'\inparen{ S_{M,t}^{\ell'} } f  } \simeq \nu_t\insquare{  u'\inparen{ \sqrt{q} z + \sqrt{1 - q}\xi^\ell  } u'\inparen{ \sqrt{q} z + \sqrt{1 - q}\xi^{\ell'}  } f } \simeq \nu_t \insquare{ u'\inparen{ S_{M}^{0,\ell}  } u'\inparen{ S_{M}^{0,\ell'} } f  },
\end{align*}
for some function $f$. This implies that the $u'(S_{M,t}^{\ell})$ and the $u'(S_{M}^{0,\ell})$ cancel each other in terms $\textnormal{II}$ and $\textnormal{III}$ in Proposition \ref{proposition:perceptron_t_interpolation}. On the other hand, term $\textnormal{I}$ involves decoupling $u''(S_{M,t}^{\ell})$ from $f$, again using the large-$N$ law of $S_{M,t}^{\ell}$:
\begin{align*}
    \nu_t\insquare{u''\inparen{S_{M,t}^{\ell}}f(\vec{x}^1,\dots,\vec{x}^n)} &\simeq \nu_t\insquare{u''\inparen{\sqrt{q}z + \sqrt{1 - q}\xi^{\ell}}f(\vec{x}^1,\dots,\vec{x}^n)} \\
    &= \nu_t\insquare{u''\inparen{\sqrt{q}z + \sqrt{1 - q}\xi^{n+1}}f(\vec{x}^1,\dots,\vec{x}^n)},
\end{align*}
where equality follows from $z$ and $\xi$ being independent of all other sources of randomness. 

We are thus motivated to consider the Hamiltonian
\begin{align}
    -H_{N,M-1,t}(\vec{x}) &= \sum_{m \leq M-1} u\inparen{ \frac{1}{\sqrt{N}} \sum_{k+1 \leq i \leq N} g_{i,m}x_i + \sqrt{\frac{t}{N}} \sum_{j \leq k}  x_j g_{j,m}    } \nonumber \\
    &\quad\quad\quad\quad + \sqrt{\frac{1-t}{N}} \sum_{m \leq M-1} \inparen{ \sum_{j \leq k} x_j \tilde{g}_{j,m}    } u'\inparen{  \frac{1}{\sqrt{N}} \sum_{k+1 \leq i \leq N} g_{i,m}x_i   },
    \label{eq:perceptron_-H_N,M-1,t}
\end{align}
with the associated Gibbs expectations are denoted by $\inangle{\cdot}_{t,\sim}$. Define
\begin{align*}
    T_{M}^{\ell} = T_{M,k}^{\ell} := \sum_{j \leq k} x_j^{\ell} \tilde{g}_{j,M}.
\end{align*}
Let $f$ be a function on $\Sigma_{N}^{n}$.
From the relation $-H^\ell_{N,M,t} = -H^\ell_{N,M-1,t} + u\inparen{S^\ell_{M,t}} + \sqrt{(1-t)/N} T_M^{\ell} u'\inparen{ S_{M}^{0,\ell}  }$, we have the identity
\begin{align}
    \nu_t\insquare{f} &= \Eb \frac{ \inangle{ f \exp\inparen{ \sum_{\ell \leq n} u\inparen{S_{M,t}^{\ell}} + \sqrt{\frac{1-t}{N}} T_{M}^{\ell} u'\inparen{ S_{M}^{0,\ell}  } } }_{t,\sim}  }{ \inangle{ \exp\inparen{ u\inparen{S_{M,t}^{1}} + \sqrt{\frac{1-t}{N}} T_{M}^{1} u'\inparen{ S_{M}^{0,1}  } } }_{t,\sim}^{n}  }.
    \label{eq:perceptron_cavityInM_method}
\end{align}
Following the discussion above, within the cavity-in-$M$ system $\inangle{\cdot}_{t,\sim}$, we expect that the law of $S_{M,t}^{\ell}$ will be close to that of $\sqrt{q}z + \sqrt{1-q}\xi^\ell$. However, we first need to get rid of the nuisance term involving $T_M^\ell$ in \eqref{eq:perceptron_cavityInM_method}.

Define for $0 \leq \gamma \leq 1$,
\begin{align}
    \nu_{t,\gamma} \insquare{f} &:= \Eb \frac{ \inangle{ f \exp\inparen{ \sum_{\ell \leq n} u\inparen{S_{M,t}^{\ell}} + \sqrt{\gamma}\sqrt{\frac{1-t}{N}} T_{M}^{\ell} u'\inparen{ S_{M}^{0,\ell}  } } }_{t,\sim}  }{ \inangle{ \exp\inparen{  u\inparen{S_{M,t}^{1}} + \sqrt{\gamma}\sqrt{\frac{1-t}{N}} T_{M}^{1} u'\inparen{ S_{M}^{0,1}  } } }_{t,\sim}^{n}  }.
\end{align}

The proof of this next result is a straightforward application of GIPF, the $1/N$ factor in front of the term $T_M^\ell$ that we want to get rid off makes this easy; proof is given in Appendix \ref{sec:suppProofs_for_perceptron}.

\begin{lemma}
Let $f$ be a function on $\Sigma_{N}^{n}$, independent of the randomness in $\inparen{ \tilde{g}_{j,M} }_{j \leq M}$. Then for every $0 \leq t \leq 1$,
\begin{align*}
    \abs{\nu_t \insquare{f} - \nu_{t,0}\insquare{f}  } &\leq \frac{ K(n,D) k  \nu_{t,\gamma}\abs{f}}{N}.
\end{align*}
\label{lemma:perceptron_gammaInterpolation}
\end{lemma}

Lemma \ref{lemma:perceptron_gammaInterpolation} can be understood as: we can go from $\nu_t$ to $\nu_{t,0}$ by paying an additive factor of order $1/N$. This is a small price, so from now we will always assume we are working with $\nu_{t,0}$ rather than $\nu_{t}$.

Let $z$ and $\xi^{\ell}$ be independent standard Gaussians. Define $\theta^\ell := \sqrt{q}z + \sqrt{1 - q}\xi^\ell$, where
\begin{align}
    q = \Eb \tanh^2\insquare{ \sqrt{r}z  }; \quad\quad r =\alpha\Eb\insquare{ \inparen{ \frac{\Eb_\xi u'(\theta) \exp \inparen{u(\theta)}}{\Eb_\xi \exp\inparen{u(\theta)}} }^2 },
    \label{eq:perceptron_RS_equations}
\end{align}
which are the replica-symmetric equations for the Perceptron model. Define
\begin{align}
    S_{v}^{0,\ell} &:= \sqrt{v} S_{M}^{0,\ell} + \sqrt{1 - v} \theta^\ell \\
    S_{v}^{\ell} &:= \sqrt{v} S_{M,t}^{\ell} + \sqrt{1 - v} \theta^\ell.
    \label{eq:perceptron_S_v^ell_definition}
\end{align}
The above quantities depend on $t$ and $M$, but in the following ``$v$-interpolation'', we think of them as fixed. Let $f$ be a function on $\Sigma_{N}^{n}$; and let $B_v$ be a function that may depend on the randomness in $M$. Define
\begin{align}
    \nu_{t,0,v}\insquare{B_v f} &:= \Eb \frac{ \Eb_{\xi} \inangle{ B_v f \exp\inparen{ \sum_{\ell \leq n} u\inparen{S_{v}^{\ell}} } }_{t,\sim}  }{ \inparen{ \Eb_{\xi} \inangle{ \exp u\inparen{S_{v}^{1}}   }_{t,\sim}}^{n}  },
\end{align}
where a reminder that the notation $\Eb_{\xi}$ denotes an expectation over all r.v.'s $\xi^{\ell}$, conditioned upon everything else, and where $\Eb$ here is an expectation over all sources of disorder, including $(g_{i,m})$ and the new disorder $z$. 

The proof of the next result is long and tedious, even if the mechanism is straightforward; it is largely inspired by the $v$-interpolation considered by Talagrand \cite{talagrand2010mean} Lemma 2.3.2. This technique will be used many times in the sequel---when it appears subsequently, we will refer to this proof rather than spell out all the details.

\begin{lemma}
Let $f$ be a function on $\Sigma_N^{n}$, possibly random, but independent of the randomness in $(g_{i,M})_{i \leq N}$, $(\tilde{g}_{j,M})_{j \leq M}$, $z$, and $(\xi^{\ell})$.
Let $B_v$ be one of the following: $1$, $u'(S_v^1)u'(S_v^2)$, $u'(S_v^1)u'(S_v^{n+1})$, $u'(S_v^{n+1})u'(S_v^{n+2})$, $u'^{2}(S_v^1)$, $u'^{2}(S_v^{n+1})$, $u''(S_v^1)$, $u''(S_v^{n+1})$, $u'(S_v^{0,1})u'(S_v^{0,2})$, $u'(S_v^{0,1})u'(S_v^{0,n+1})$, $u'(S_v^{0,n+1})u'(S_v^{0,n+2})$, $u'^{2}(S_v^{0,1})$, $u'^{2}(S_v^{0,n+1})$. Then for every $1/\tau_1 + 1/\tau_2 = 1$,
\begin{align*}
    \abs{\frac{\partial}{\partial v} \nu_{t,0,v} \insquare{B_v f} } &\leq K(n,D)k  \insquare{ \inparen{ \nu_{t,0,v}\abs{f}^{\tau_1} }^{1/\tau_1} \inparen{ \nu_{t,0,v}\abs{R_{1,2} - q}^{\tau_2} }^{1/\tau_2}  + \frac{1}{N} \nu_{t,0,v}\abs{f}  }.
\end{align*}
\label{lemma:perceptron_vInterpolation}
\end{lemma}

In order to make the $t$, $\gamma$, and $v$ interpolations useful, we need to get rid of the dependence of the expectations on the RHS on $t$, $\gamma$, and $v$, so that we may use the concentration of $R_{1,2} \simeq q$ under $\nu$. The approach is identical to that used for instance in \cite{talagrand2010mean} Proposition 2.3.6.

\begin{lemma} 
Let $f$ be a function on $\Sigma_{N}^{n}$ that is independent of the randomness in $(g_{i,M})_{i \leq N}$, $z$, $(\xi^\ell)_{\ell \leq n}$, and satisfying $f \geq 0$. Then for every $0 \leq t,v,\gamma \leq 1$,
\begin{align}
    \nu_{t,0,v}\insquare{f} &\leq K(n,k,D) \nu_{t,0}\insquare{f} 
    \label{eq:perceptron_getRidOfv}
\end{align}
and
\begin{align}
    \nu_{t,\gamma}\insquare{f} &\leq K(n,k,D) \nu_{t}\insquare{f} 
    \label{eq:perceptron_getRidOfgamma}
\end{align}
\end{lemma}

\begin{proof}
Apply Lemma \ref{lemma:perceptron_vInterpolation} with $B_v \equiv 1$, and $\tau_1 = 1$, $\tau_2 = \infty$, noting that $\abs{R_{1,2} - q} \leq 2$, to obtain
\begin{align*}
    \abs{\frac{\partial }{\partial v} \nu_{t,0,v}\insquare{f} } &\leq K(n,k,D) \nu_{t,0,v}\insquare{f},
\end{align*}
and the result follows by integrating the above differential inequality with respect to $v$, noting that $\nu_{t,0,1} = \nu_{t,0}$. This proves \eqref{eq:perceptron_getRidOfv}.

For \eqref{eq:perceptron_getRidOfgamma}, it is the same idea: from \eqref{eq:perceptron_gammaInterpolation_finalEquality} we obtain 
\begin{align*}
    \abs{\frac{\partial }{\partial \gamma} \nu_{t,\gamma}\insquare{f} } &\leq K(n,k,D) \nu_{t,\gamma}\insquare{f},
\end{align*}
and the result follows from integrating in $\gamma$, and using that $\nu_{t,1} = \nu_{t}$.
\end{proof}

From Proposition \ref{proposition:perceptron_t_interpolation}, Lemma \ref{lemma:perceptron_gammaInterpolation}, Lemma \ref{lemma:perceptron_vInterpolation}, and \eqref{eq:perceptron_getRidOfv} and \eqref{eq:perceptron_getRidOfgamma}, it is a matter of triangle inequalities to derive the following statement; details are given in Appendix \ref{sec:suppProofs_for_perceptron}.

\begin{lemma}
Let $f$ be a function on $\Sigma_{N}^{n}$ that is independent of the randomness in $(g_{i,M})_{i \leq N}$, $z$, $(\xi^\ell)_{\ell \leq n}$. Then
\begin{align*}
    \abs{\frac{\ud}{\ud t} \nu_{t} \insquare{f} } &\leq \alpha K(n,k,D)  \insquare{ \inparen{ \nu_{t}\abs{f}^{\tau_1} }^{1/\tau_1} \inparen{ \nu_{t}\abs{R_{1,2} - q}^{\tau_2} }^{1/\tau_2}  + \frac{1}{N} \nu_{t}\abs{f}  }.
\end{align*}
\label{lemma:perceptron_ddt_nu_t_f_inTermsOf_nu_t_f}
\end{lemma}

\begin{lemma}
There is a number $K(n,k,D)$ with the property that whenever $\alpha\leq 1/K(k,D) $, and $f \geq 0$ is a function on $\Sigma_N^{n}$, and is independent of the randomness in $(g_{i,M})_{i \leq N}$, $(\tilde{g}_{j,M})_{j \leq k}$, $z$, and $(\xi^{\ell})_{\ell \leq n}$, then
\begin{align}
    \nu_t\insquare{f} \leq 2 \nu\insquare{f}.
    \label{eq:perceptron_getRidOft}
\end{align}
\end{lemma}
\begin{proof}
Apply Lemma \ref{lemma:perceptron_ddt_nu_t_f_inTermsOf_nu_t_f} with $\tau_1 = 1$, $\tau_2 = \infty$ to obtain
\begin{align*}
    \abs{\frac{\ud }{\ud t} \nu_t \insquare{f}} \leq \alpha \bar{K(n,k,D)}\nu_t\insquare{f},
\end{align*}
where $\bar{K(n,k,D)}$ does not depend on $n$. If $\alpha \bar{K(n,k,D)} \leq \log 2$. Integrating the above differential inequality yields the desired statement.
\end{proof}

We immediately have the following upon applying \eqref{eq:perceptron_getRidOft} in Lemma \ref{lemma:perceptron_ddt_nu_t_f_inTermsOf_nu_t_f}.

\begin{theorem}
Let $f$ be a function on $\Sigma_{N}^{n}$ which is independent of the randomness in $(g_{i,M})_{i \leq N}$, $\inparen{\tilde{g}_{j,M}}_{j \leq k}$, $z$, $(\xi^\ell)_{\ell \leq n}$. There exists a constant $K^*(n,k,D)$ such that whenever $\alpha \leq 1/K^*(n,k,D)$, then
\begin{align}
    \abs{ \nu\!\insquare{f} - \nu_{0}\!\insquare{f}  } \leq K(n,k,D)\insquare{  \inparen{ \nu\abs{f}^{\tau_1} }^{1/\tau_1} \inparen{ \nu\abs{R_{1,2} - q}^{\tau_2} }^{1/\tau_2}  + \frac{1}{N} \nu\abs{f}   },
\end{align}
where $q$ is given by the solution to \ref{eq:perceptron_RS_equations}.
\label{thm:perceptron_decomposeHamiltonian_annealed}
\end{theorem}

Theorem \ref{thm:perceptron_decomposeHamiltonian_annealed} says that the annealed expectations of the original Hamiltonian are close to that of the decomposed Hamiltonian. However, we need a quenched statement---with high probability disorder, $\inangle{f}$ is close to $\inangle{f}_{0}$. This will be a consequence of the `replica trick', which is related to the `Hoeffding principle' in the proof of Theorem \ref{thm:ProjResult_DisorderedCase_supOverBL_LM_2p}.

Before that, we need to address a limitation of Theorem \ref{thm:perceptron_decomposeHamiltonian_annealed}. We will need closeness of annealed expectations for functions such as $f\inangle{g}_{0}$, for some deterministic functions $f$ and $g$. The Gibbs expectation $\inangle{\cdot}_0$ depends on $g_{i,M}$'s, so that Theorem \ref{thm:perceptron_decomposeHamiltonian_annealed} does not apply. Instead, we will create a $(N,M+1)$-system $\inangle{\cdot}^+$, whose annealed expectations we denote by $\nu^+ = \Eb\inangle{\cdot}^+$. (Note that quantities related to the $(N,M+1)$ system have a $+$ superscript appended). We then have $\nu^+\insquare{f\inangle{g}_{0}} \simeq \nu_0^+\insquare{f\inangle{g}_{0}}$, since now $f\inangle{g}_0$ is independent of the randomness in $g_{i,M+1}$'s. We also need to show $\nu^+ \simeq \nu$ and $\nu^+_0 \simeq \nu_0$.

Onto the actual work of making the above precise: define the $(N,M+1)$ analog of \eqref{eq:perceptron_tInterpolatingHamiltonian}, for $0 \leq t \leq 1$,
\begin{align}
    -H_{N,M+1,t} &= \sum_{m \leq M+1} u\inparen{ \frac{1}{\sqrt{N}} \sum_{k+1 \leq i \leq N} g_{i,m}x_i + \sqrt{\frac{t}{N}} \sum_{j \leq k}  x_j g_{j,m}    } \nonumber \\
    &\quad\quad\quad\quad + \sqrt{\frac{1-t}{N}} \sum_{m \leq M+1} \inparen{ \sum_{j \leq k} x_j \tilde{g}_{j,m}    } u'\inparen{  \frac{1}{\sqrt{N}} \sum_{k+1 \leq i \leq N} g_{i,m}x_i   },
    \label{eq:perceptron_H_N,M+1,t}
\end{align}
whose associated Gibbs expectations we denote by $\inangle{\cdot}^+_{t}$. Denote $\alpha^+ = (M+1)/N$. Let $z^+$, $\inparen{ \xi^{+,\ell} }_{\ell}$ be independent standard Gaussians and set $\theta^{+,\ell} := \sqrt{q^+}z^+ + \sqrt{1-q^+}\xi^{+,\ell}$. Define the replica-symmetric equations for the $\nu^+$ system as follows:
\begin{align}
    q^+ = \Eb \tanh^2\insquare{ \sqrt{r^+}z^+  }; \quad\quad r^+ =\alpha^+\Eb\insquare{ \inparen{ \frac{\Eb_\xi u'(\theta^+) \exp \inparen{u(\theta^+)}}{\Eb_\xi \exp\inparen{u(\theta^+)}} }^2 }
    \label{eq:perceptron_RS_equations_+}
\end{align}

\begin{corollary}
Let $f$ be a function on $\Sigma_{N}^{n}$ which is independent of the randomness in $(g_{i,M+1})_{i \leq N}$, $\inparen{ \tilde{g}_{j,M+1} }_{j \leq k}$, $z^+$, $(\xi^{+,\ell})_{\ell \leq n}$. Then there exists a constant $K^+(n,k,D)$ such that whenever $\alpha^+ \leq 1/K^+(n,k,D)$, then
\begin{align}
    \abs{ \nu^+\!\insquare{f} - \nu^{+}_{0}\!\insquare{f}  } \leq K(n,k,D)\insquare{  \inparen{ \nu^+\abs{f}^{\tau_1} }^{1/\tau_1} \inparen{ \nu^+\abs{R_{1,2} - q^+}^{\tau_2} }^{1/\tau_2}  + \frac{1}{N} \nu^+\abs{f}   }
\end{align}
\label{thm:perceptron_decomposeHamiltonian_annealed_M+1}
\end{corollary}

\begin{proof}
The result follows by repeating the proof of Theorem \ref{thm:perceptron_decomposeHamiltonian_annealed}, this time creating a ``cavity-in-($M+1$)''. We simply have to replace all occurrences of $S_{M}^{\ell}$, $S_{M,t}^{\ell}$, $S_{M}^{0,\ell}$ by $S_{M+1}^{\ell}$, $S_{M+1,t}^{\ell}$, $S_{M+1}^{0,\ell}$, and use $\alpha^+$, $q^+$, $r^+$ instead of $\alpha$, $q$, $r$.
\end{proof}

\begin{theorem}
Let $f$ be a function on $\Sigma_{N}^{n}$ which is independent of the randomness in $(g_{i,M+1})_{i \leq N}$, $\inparen{ \tilde{g}_{j,M+1} }_{j \leq k}$, $z^+$, $(\xi^{+,\ell})_{\ell \leq n}$. Then for every $1/\tau_1 + 1/\tau_2 = 1$,
\begin{align}
    \abs{\nu^+\insquare{f} - \nu\insquare{f}} \leq K(n,D)\insquare{ \inparen{ \nu^+\abs{f}^{\tau_1} }^{1/\tau_1} \inparen{ \nu^+\abs{R_{1,2} - q^+}^{\tau_2} }^{1/\tau_2}   },
    \label{eq:perceptron_nu^+_to_nu}
\end{align}
and there exists a constant $K^+(n,k,D)$ such that whenever $\alpha^+ \leq 1/K^+(n,k,D)$,
\begin{align}
    \abs{\nu^+_0\insquare{f} - \nu_0\insquare{f}} \leq K(n,k,D)\insquare{ \inparen{ \nu^+\abs{f}^{\tau_1} }^{1/\tau_1} \inparen{ \nu^+\abs{R_{1,2} - q^+}^{\tau_2} }^{1/\tau_2}  + \frac{1}{N}\nu^+\abs{f} }.
    \label{eq:perceptron_nu^+_0_to_nu_0}
\end{align}
\label{thm:perceptron_+_to_originalSystem}
\end{theorem}

The proof is in fact a $v$-interpolation for fixed $t$ and is given in Appendix \ref{sec:suppProofs_for_perceptron}. 


\begin{theorem}
For every $k \leq N$, fix $(\sigma_1,\dots,\sigma_k) \in \Sigma_{k}$. Define
\begin{align*}
    C = C_{\sigma^1,\dots,\sigma^k} = \inbraces{\vec{x} \in \Sigma_N : (x_1,\dots,x_k) = (\sigma_1,\dots,\sigma_k)}.
\end{align*}
For every integer $p \geq 1$, there exists a constant $K^*(p,k,D)$ such that whenever $\alpha K^*(p,k,D) \leq 1$, then for any $s \in [1,\infty)$,
\begin{align}
    \Eb_{\textnormal{d}} \insquare{\inparen{ \inangle{\rchi_{C}} - \inangle{\rchi_C}_{0}  }^{2p}  } \leq K(p,k,D) \insquare{  \inparen{ \nu^+ \abs{R_{1,2} - q^+}^{s}  }^{1/s} + \frac{1}{N} }.
\end{align}
\label{thm:perceptron_Hamiltonian_decomposition}
\end{theorem}

\begin{proof}
Denote $\rchi_C^{\ell} := \rchi_C(\vec{x}^{\ell})$ (this will not be confused with powers of the indicator function). Using \eqref{eq:(1-1)^n_expand_identity} and replicas, we have
\begin{align}
    \Eb_{\textnormal{d}} \insquare{\inparen{ \inangle{\rchi_{C}} - \inangle{\rchi_C}_{0}  }^{2p}  } &= \sum_{1 \leq r \leq 2p} \inparen{-1}^{2p - r} \binom{2p}{r} \Eb_{\textnormal{d}} \insquare{ \inangle{\rchi_{C}^{1} \cdots \rchi_{C}^{r} } \inangle{ \rchi_C^{1}\dots \rchi_{C}^{2p - r}  }_{0} - \inangle{ \rchi_{C}^{1} \cdots \rchi_C^{2p} }_{0}  }.
    \label{eq:perceptron_quenchedCloseness_betweenoriginal_and_decomposed_expansion}
\end{align}
For every $1 \leq r \leq 2p$, write
\begin{align*}
    &\abs{ \Eb_{\textnormal{d}} \insquare{ \inangle{\rchi_{C}^{1} \cdots \rchi_{C}^{r} } \inangle{ \rchi_C^{1}\dots \rchi_{C}^{2p - r}  }_{0} - \inangle{ \rchi_{C}^{1} \cdots \rchi_C^{2p} }_{0}  } }\\
    &\quad = \abs{ \Eb_{\textnormal{d}} \insquare{ \inangle{\rchi_{C}^{1} \cdots \rchi_{C}^{r} \inangle{ \rchi_C^{1}\dots \rchi_{C}^{2p - r}  }_{0} }  - \inangle{ \rchi_{C}^{1} \cdots \rchi_C^{r} \inangle{ \rchi_C^{1}\dots \rchi_{C}^{2p - r}  }_{0} }_{0}  } }\\
    &\quad = \abs{\nu\insquare{f} - \nu_0\insquare{f}},
\end{align*}
where $f$ is the function on $\Sigma_{N}^{r}$ defined by $f := \rchi_{C}^{1} \cdots \rchi_{C}^{r} \inangle{ \rchi_C^{1}\dots \rchi_{C}^{2p - r}  }_{0}$. The function $f$ depends on the randomness in $g_{i,M}$'s---we will need to move the expectations to the $\nu^+$ system, and then we can exploit the fact that $f$ is independent of the randomness in $g_{i,M+1}$'s. We have that there exists a constant $K^+(r,k,D)$ such that whenever $\alpha^+ K^+(r,k,D) \leq 1$,
\begin{align*}
    \abs{\nu\insquare{f} - \nu_0\insquare{f}} &\leq \abs{\nu\insquare{f} - \nu^+\insquare{f}} + \abs{\nu^+\insquare{f} - \nu^{+}_0\insquare{f}} + \abs{\nu^{+}_{0}\insquare{f} - \nu_0\insquare{f}} \\
    &\leq K(r,k,D) \insquare{ \inparen{\nu^+\abs{R_{1,2} - q}}^{1/s} + \frac{1}{N} },
\end{align*}
which follows from \eqref{eq:perceptron_nu^+_to_nu} and \eqref{eq:perceptron_nu^+_0_to_nu_0} applied to the first and third terms respectively, and Theorem \ref{thm:perceptron_decomposeHamiltonian_annealed_M+1} to the second term, noting that $\abs{f} \leq 1$. 

Substituting the above into \eqref{eq:perceptron_quenchedCloseness_betweenoriginal_and_decomposed_expansion}, we obtain that for $K^+(p,k,D) := \max_{1 \leq r \leq 2p} K^+(r,k,D)$, whenever $\alpha^+ K^+(p,k,D) \leq 1$, then
\begin{align*}
    \Eb_{\textnormal{d}} \insquare{\inparen{ \inangle{\rchi_{C}} - \inangle{\rchi_C}_{0}  }^{2p}  } &\leq K(p,k,D) \insquare{ \inparen{\nu^+\abs{R_{1,2} - q}}^{1/s} + \frac{1}{N} }.
\end{align*}
It remains to assert the existence of a constant $K^*$ such that whenever $\alpha K^* \leq 1$, the above holds. Define $K^*(p,k,D) = 2 K^+(p,k,D)$. Then whenever $\alpha K^*(p,k,D) \leq 1$, we have 
\begin{align*}
    \alpha^+ K^+ = \alpha \frac{M+1}{M} K^+ \leq \alpha 2 K^+ = \alpha K^* \leq 1. 
\end{align*}
\end{proof}

\subsection{Thin-shell and overlap concentration for main and auxiliary spin systems in Perceptron}

Based on the $-H_{N,M,0}$ Hamiltonian defined in \eqref{eq:perceptron_-H_N,M-1,t}, we will be interested in $(N-k)$-Gibbs systems whose Hamiltonian is a function on $\Sigma_{N-k}$ that we will express as
\begin{align}
    -H_{N-k,M}^{-}(\vec{x}) &:= \sum_{m \leq M} u\inparen{ \frac{1}{\sqrt{N}} \sum_{k+1 \leq i \leq N} g_{i,m}x_i  } = \sum_{m \leq M} \tilde{u} \inparen{ \frac{1}{\sqrt{N-k}} \sum_{k+1 \leq i \leq N} g_{i,m}x_i  },
    \label{eq:perceptron_H^-_N-k,M}
\end{align}
where $\tilde{u}$ is the smooth function defined by
\begin{align}
    \tilde{u}(y) &:= u\inparen{ \sqrt{\frac{N-k}{N}} y }.
    \label{eq:perceptron_tilde_u_definition}
\end{align}
We denote expectations with respect to the Hamiltonian in \eqref{eq:perceptron_H^-_N-k,M} by $\nu^- = \Eb \inangle{\cdot}^{-}$.
It is straightforward to check that $\tilde{u}$ satisfies the same bounded conditions as $u$ in \eqref{eq:perceptron_|u^(d)|<=D}, with the same $D$:
\begin{align}
    \abs{\tilde{u}^{(d)}} \leq D, \quad 0 \leq d \leq 3.
    \label{eq:perceptron_|u(tilde)^(d)|<=D}
\end{align}
We next define the parameters associated to this truncated $(N-k)$-system. Let $z$, $\xi$ be independent standard Gaussians, and let $\theta^- = \sqrt{q^{-}}z + \sqrt{1 - q^{-}}\xi$, where
\begin{align}
    q^- = \Eb \tanh^2\insquare{ \sqrt{r^-}z  }; \quad\quad r^- =\alpha^-\Eb\insquare{ \inparen{ \frac{\Eb_\xi \tilde{u}'(\theta^-) \exp \inparen{
    \tilde{u}(\theta^-)}}{\Eb_\xi \exp\inparen{\tilde{u}(\theta^-)}} }^2 }
    \label{eq:perceptron_RS_equations_-},
\end{align}
where $\alpha^- := M/(N-k)$. We also define the ``limiting thin-shell'' value for the $\tilde{u}'$ spins
\begin{align}
    \tau^- &:= \alpha^-\Eb\insquare{  \frac{\Eb_\xi \tilde{u}^{\,\prime 2}(\theta^-) \exp \inparen{
    \tilde{u}(\theta^-)}}{\Eb_\xi \exp\inparen{\tilde{u}(\theta^-)}}  }
    \label{eq:perceptron_tau^-}
\end{align}

Let us recall that $\nu^+$ is associated with the $(N,M+1)$-system, with the same $u$ as the original system; while $\nu^-$ is associated with the ``truncated'' $(N-k,M)$-system, where $\tilde{u}$ replaces $u$. In both these cases, since $u$ and $\tilde{u}$ both satisfy the same conditions \eqref{eq:perceptron_|u^(d)|<=D} and  \eqref{eq:perceptron_|u(tilde)^(d)|<=D}, it is clear that if we have thin-shell and overlap concentration for $(x_i)$ (and indeed $\inparen{u'(S_m)}$ as will be seen later) under the original $\nu$ system, that we also have them under $\nu^-$ and $\nu^+$, under a different set of parameters: respectively $(\alpha^-, q^-, r^-, \tilde{u})$ and $(\alpha^+, q^+, r^+, u)$. The following is therefore a consequence of \cite{talagrand2010mean} Theorem 2.4.1.  

\begin{theorem}
\begin{enumerate}[(a)]
    \item There exists a constant $K^- = K^-(k,D)$ such that whenever $\alpha^- K^- \leq 1$, then
\begin{align}
    \nu^-\insquare{ \inparen{ R^{-}_{1,2} - q^- }^{2}  } \leq \frac{L}{N-k},
    \label{eq:perceptron_R12_OC_-}
\end{align}
where $R^{-}_{1,2} = (N-k)^{-1/2} \sum_{k+1 \leq i \leq N} x_i^1 x_i^2$, and $q^-$ is given by \eqref{eq:perceptron_RS_equations_-}. 
    \item There exists a constant $K^+ = K^+(k,D)$ such that whenever $\alpha^+ K^+ \leq 1$, then
\begin{align}
    \nu^+\insquare{ \inparen{ R_{1,2} - q^+ }^{2}  } \leq \frac{L}{N},
    \label{eq:perceptron_R12_OC_+}
\end{align}
where $q^+$ is given by \eqref{eq:perceptron_RS_equations_+}.
\end{enumerate}
\label{thm:perceptron_R12_OC}
\end{theorem}

Define
\begin{align}
    S_{m}^{-, \ell} := \frac{1}{\sqrt{N-k}} \sum_{k+1 \leq i \leq N} g_{i,m} x_i^{\ell},
    \label{eq:perceptron_Sm^-_definition}
\end{align}
so that in fact 
\begin{align}
    \tilde{u}^{\,\prime}(S_m^{-}) = \sqrt{ \frac{N-k}{N} } u'(S_m^0)
    \label{eq:perceptron_tildeu'_u'_relation}
\end{align}

\begin{theorem}
Let $r^-$ and $\tau^-$ be given as in \eqref{eq:perceptron_RS_equations_-} and \eqref{eq:perceptron_tau^-} respectively, then there exists $K^*(k,D)$ such that whenever $\alpha^- K^* \leq 1$, we have
\begin{align}
    \nu^-\insquare{ \inparen{ \frac{1}{N-k} \sum_{m \leq M} \tilde{u}^{\,\prime 2} \inparen{ S_m^{-}  } - \tau^-  }^2 } \leq \frac{K(k,D)}{N-k}
    \label{eq:perceptron_u(tilde)'(S_m^-)_TS}
\end{align}
and 
\begin{align}
    \nu^-\insquare{ \inparen{ \frac{1}{N-k} \sum_{m \leq M} \tilde{u}^{\,\prime} \inparen{ S_m^{-,1}  } \tilde{u}^{\,\prime} \inparen{ S_m^{-,2}  } - r^-  }^2 } \leq \frac{K(k,D)}{N-k}
    \label{eq:perceptron_u(tilde)'(S_m^-)_OC}
\end{align}
\label{thm:perceptron_TSOC_auxiliarySystem}
\end{theorem}

The overlap concentration for $\inparen{u'(S_m)}$ was proven for the original system $\nu$ in \cite{talagrand2010mean} Lemma 2.4.3. It follows that \eqref{eq:perceptron_u(tilde)'(S_m^-)_OC} is simply an extension to the $\nu^-$ setting. It remains only to prove \eqref{eq:perceptron_u(tilde)'(S_m^-)_TS}; the proofs are given in Appendix \ref{sec:suppProofs_for_perceptron}.

\subsection{Local independence for Ising Perceptron}

By Theorem \ref{thm:perceptron_Hamiltonian_decomposition}, the Ising Perceptron Hamiltonian \eqref{eq:perceptron_Hamiltonian} is $(K/N)$-decomposable with 
\begin{align}
    -H_{N,M,0}(\vec{x}) &= -H_{N-k,M}^-(\vec{y}) + \sum_{j \leq k} x_j \sum_{m \leq M} \sqrt{\alpha^-} \frac{g_{j,m}}{\sqrt{M}} \tilde{u}^{\, \prime } \inparen{ S_m^-  },
\end{align}
where recall that $\alpha^- = M/(N-k)$, $\tilde{u}$ is defined in \eqref{eq:perceptron_tilde_u_definition}, $S_m^-$ is defined in \eqref{eq:perceptron_Sm^-_definition}, and we have used relation \eqref{eq:perceptron_tildeu'_u'_relation}, and where the truncated system Hamiltonian $-H_{N-k,M}^-$ is given in \eqref{eq:perceptron_H^-_N-k,M}, and has associated Gibbs expectations denoted by $\inangle{\cdot}^-$. 

\begin{theorem}[Local independence, Ising Perceptron]
Denote by $G_N^{(k)}$ the Gibbs marginal on the first $k$-coordinates. There exists a constant $K^* = K^*(k,D)$ so that whenever $\alpha K^* \leq 1$, for every $p \geq 1$, for every $0 < \epsilon < 1/2$, whenever $N \geq \exp\inparen{\sqrt{16k}/\epsilon}$, 
\begin{enumerate}
    \item we have 
    \begin{align}
        \Eb_{\ud}\insquare{ \sup_{(\sigma_1,\dots,\sigma_k) \in \Sigma_k} \inparen{   G_N^{(k)} (\sigma_1,\dots,\sigma_k)   - \prod_{j \leq k} \frac{\exp \sigma_j \vec{a}_j^{\top} \inangle{ \Ub }^-    }{ 2\cosh \inparen{ \vec{a}_j^{\top} \inangle{\Ub}^- }  } }^{2p}  } \leq \frac{K}{N^{\frac{1}{2} - \epsilon}  },
        \label{eq:perceptron_LI_statement_partialLimiting}
    \end{align}
    where $\vec{a}_j := \inparen{\frac{1}{\sqrt{N}} g_{j,m}  }_{m \leq M}$, and $\Ub := \inparen{ u'(S_m^0)  }_{m \leq M}$;
    \item furthermore, for independent standard Gaussians $(z_j)_{j \leq k}$ independent of everything else,
    \begin{align}
        \Eb_{\ud}\Eb_{\vec{z}}\insquare{ \sup_{(\sigma_1,\dots,\sigma_k) \in \Sigma_k} \inparen{   G_N^{(k)} (\sigma_1,\dots,\sigma_k)   - \prod_{j \leq k} \frac{\exp \sigma_j \sqrt{r^-} z_j     }{ 2\cosh \inparen{ \sqrt{r^-} z_j }  } }^{2p}  } \leq \frac{K}{N^{\frac{1}{2} - \epsilon}  },
        \label{eq:perceptron_LI_statement_Limiting}
    \end{align}
\end{enumerate}
where $r^-$ is from \eqref{eq:perceptron_RS_equations_-}, and for constants $K$ depending on $K^*, k, p$, and $\epsilon$.
\label{thm:perceptron_LI}
\end{theorem}

\begin{proof}
Apply Theorem \ref{thm:LI_abstract_theorem} with $\varrho \equiv 1$, $w_m = \sqrt{\alpha^-} \tilde{u}^{\, \prime} (S_m^-)$, $\Xi = \tau^-$, $\Upsilon = r^-$, $\mu$ is the uniform probability on $\inbraces{\pm 1 }$, $f_j \equiv 0$, and with $\vec{A}_j = \inparen{ \frac{1}{\sqrt{M}} g_{j,m} }_{m \leq M}$. It is more convenient to apply Theorem \ref{thm:LI_abstract_theorem} with $\vec{A}_j$ and $\vec{w}$ so defined, then later we use the relation
\begin{align*}
    \vec{A}_j^{\top} \inangle{\vec{w}}^-  = \sum_{m \leq M} \frac{g_{j,m}}{\sqrt{M}} \sqrt{\alpha^-} \tilde{u}^{\, \prime }(S_m^-) = \sum_{m \leq M} \frac{g_{j,m}}{\sqrt{N}} u^{\, \prime }(S_m^0)  =\vec{a}_j^{\top} \inangle{\Ub}^-
\end{align*}
to convert interchangeably. Note that for any $P \in \RR$,
\begin{align*}
    \int_{\RR} \exp \inparen{ \frac{x_j^2 \varrho^2}{2}(\Xi - \Upsilon) + x_jP  } \mu(\ud x_j) = \exp\inparen{ \frac{1}{2}(\tau^- - r^-)  }2\cosh(P) \geq 1,
\end{align*}
so that hypothesis 2 in Theorem \ref{thm:LI_abstract_theorem} is satisfied. Hypothesis 1 in Theorem \ref{thm:LI_abstract_theorem} is satisfied by Theorem \ref{thm:perceptron_TSOC_auxiliarySystem}; note that the accuracy in \eqref{eq:perceptron_u(tilde)'(S_m^-)_TS}, \eqref{eq:perceptron_u(tilde)'(S_m^-)_OC} is $K/(N-k)$, but since the ratio $M/(N-k) = \alpha^-$ is of constant order, this is really $K/M$.  
\end{proof}

\begin{theorem}[Converse to local independence, Ising Perceptron]
Suppose that \eqref{eq:perceptron_LI_statement_Limiting} holds for $k \leq 2$, $p = 1$. Let $\nu_N \insquare{\cdot} := \Eb_{\ud} \inangle{\cdot}$, and choose any $0 < \delta < 1/8$. Then
\begin{align}
    \nu_N\!\insquare{ \inparen{ R_{12} - \nu_N R_{12}  }^2   } \leq \frac{K}{N^{1/8 - \delta}},
    \label{eq:perceptron_converseToLI}
\end{align}
for some constant $K$ that depends on $K^*$ and $\delta$.
\label{thm:perceptron_conversetoLI}
\end{theorem}

\begin{proof}
    Apply Theorem \ref{thm:LI=>TSOC_abstract} with $\eta = 1/4 - 2\delta$, $\HH_j (\sigma_j) \propto \exp \sigma_j (\sqrt{r^-} z_j)$ (defined conditionally on $z_j$). Note that $\nu_N^{(1)} x^4 \equiv 1$ and $\HH x^4 \equiv 1$.
\end{proof}

\section{Shcherbina-Tirozzi model}
\label{sec:Shcherbina-Tirozzi}


The Shcherbina-Tirozzi (ST) model comes from the Gardner problem on the sphere (\cite{shcherbina2003rigorous}; \cite{talagrand2010mean} Chapter 3; \cite{talagrand2011mean} Chapter 8). The model has connections with several statistical models of interest including LASSO estimation \cite{donoho2016high}, and mismatched Bayesian linear regression \cite{barbier2021performance}. The Hamiltonian $-H_{N,M}$ is a function on $\RR^N$ given by
\begin{align}
    -H_{N,M}(\vec{x}) &:= \sum_{m \leq M} u\inparen{ \frac{1}{\sqrt{N}}\sum_{i \leq N} g_{i,m}x_i   } - \kappa \norm{\vec{x}}^2 + h\sum_{i \leq N} g_i x_i,
    \label{eq:ST_Hamiltonian_original}
\end{align}
where $\kappa > 0$, $h \geq 0$, $(g_{i,m})_{i \leq N, m \leq M}$, $(g_i)_{i \leq N}$ are independent standard Gaussians, and where $u \leq 0$ is a concave function that satisfies
\begin{align}
    \abs{u^{(d)}} \leq D,\; \textnormal{ for } 1 \leq d \leq 4; \quad \textnormal{ and } \quad 
    u(x) \geq -D (1 + \abs{x}),\; \textnormal{for all } x \in \RR.
    \label{eq:ST_u_asmpts}
\end{align}
Let us make further assumptions that there exists constants $\kappa_0$ and $h_0$ such that
\begin{align}
    \kappa \geq \kappa_0\; \quad 0 \leq h \leq h_0\; \quad M \leq 10N,
    \label{eq:ST_other_asmpts}
\end{align}
which are the same conditions as in \cite{talagrand2010mean} Equation (3.25). However, we will have to assume something concrete about $\kappa_0$: that 
\begin{align}
    \kappa_0 > \max\inparen{ 48 k^2 p^3 (1+ 8 \bar{\epsilon}^2)(\alpha D^2 + h^2), \, \sqrt{150(1+h^2)} \alpha D^4  },
    \label{eq:ST_kappa0_assumption}
\end{align}
where $k, p \geq 1$ are integers, and $\bar{\epsilon} := \inparen{\frac{1}{2\epsilon} - 1}$ for some $0 < \epsilon < 1/2$.

It is readily seen that $-H_{N,M}$ is ($2\kappa$)-strongly-concave:
\begin{align}
    H_{N,M}\inparen{\frac{\vec{x}+\vec{y}}{2}} &\leq \frac{1}{2} \inparen{H_{N,M}(\vec{x}) + H_{N,M}(\vec{y})} - \kappa\norm{\frac{\vec{x} - \vec{y}}{2}}^2.
    \label{eq:-H_N,M_strongConcavity}
\end{align}
The Gibbs measure $\inangle{\cdot}$ is given by
\begin{align}
    \inangle{f} &= \frac{1}{Z_{N,M}} \int_{\RR^N} f(\vec{x}) \exp\inparen{-H_{N,M}(\vec{x})} \, \ud \vec{x},
\end{align}
where $\ud\vec{x}$ is Lebesgue measure on $\RR^N$ and $Z_{N,M} = \int \exp\inparen{-H_{N,M}(\vec{x})} \, \ud \vec{x}$. In the ST model it is no longer true that $x_i^2 \equiv 1$. To ease notation, we write $x_i^{\ell, 2}$ for $\inparen{x_i^\ell}^2$ (square of the $\ell$-th replica).

One notes the similarity between the ST Hamiltonian \eqref{eq:ST_Hamiltonian_original} and that of the Perceptron \eqref{eq:perceptron_Hamiltonian}. The main technical difficulty is that the Gibbs measure now lives on $\RR^N$, so that many quantities such as $x_i^2$, and $R_{1,2} - q$ are not immediately bounded. The magnetic field term involving $h$ will not require any new concepts to handle. On the other hand, the term $-\kappa \norm{\vec{x}}^2$ is crucial, and has the effect of encouraging $\inangle{\cdot}$ to live on a sphere, so that the $x_i$'s are essentially bounded. This, along with concentration of log-strongly-concave measures for Lipschitz functions \cite{maurey1991some}, will be the main tools in overcoming the aforementioned technical difficulty.

As usual we want to decompose the Hamiltonian to isolate appropriate cavity fields around those spins which we are considering for asymptotic independence. The desirable Hamiltonian, associated with $\inangle{\cdot}_0$, is
\begin{align}
    -H_{N,M,0}(\vec{x}) &:= \sum_{m \leq M} u\inparen{ \frac{1}{\sqrt{N}}\sum_{k+1 \leq i \leq N} g_{i,m} x_i   } - \kappa\norm{\vec{x}}^2 + h\sum_{i \leq N} g_i x_i \nonumber\\
    &\quad\quad\quad\quad\quad\quad + \sum_{j \leq k} x_j \sum_{m \leq M} \frac{g_{j,m}}{\sqrt{N}} u'\inparen{ \frac{1}{\sqrt{N}}\sum_{k+1 \leq i \leq N} g_{i,m} x_i}  +  \frac{\sigma^-}{2} \sum_{j \leq k} x_j^2,
    \label{eq:ST_H_N,M,0}
\end{align}
where $\sigma^-$ satisfies the replica-symmetric equations for the \emph{truncated} system \eqref{eq:ST_RS_r_tau_sigma}, \eqref{eq:ST_RS_rho_q}. For reference, we state the RS equations for the original system \eqref{eq:ST_Hamiltonian_original}: let $z$ and $\xi$ be standard Gaussian r.v.'s, and $\theta := \sqrt{q}z + \sqrt{\rho - q} \xi$, where
\begin{align}
    r = \alpha \Eb \insquare{ \inparen{ \frac{\Eb_\xi \insquare{ u'(\theta) \exp u(\theta)   }}{\Eb_\xi \exp u(\theta)} }^{2} }; \quad \tau = \alpha \Eb \frac{\Eb_\xi \insquare{ u^{\prime 2}(\theta) \exp u(\theta)   }}{\Eb_\xi \exp u(\theta)}; \quad 
    \sigma &= \alpha \Eb \frac{\Eb_\xi \insquare{ u''(\theta) \exp u(\theta)   }}{\Eb_\xi \exp u(\theta)},
    \label{eq:ST_RS_r_tau_sigma}
\end{align}
(in \cite{talagrand2010mean} Chapter 3 it is written $\bar{r} := \sigma + \tau$); and
\begin{align}
    \rho = \frac{1}{2\kappa + r - \sigma - \tau} + \frac{r + h^2}{(2\kappa + r - \sigma - \tau)^2}; \quad\quad q = \frac{r + h^2}{(2\kappa + r - \sigma - \tau)^2}.
    \label{eq:ST_RS_rho_q}
\end{align}

The following is a very important technical fact.

\begin{lemma}[\cite{talagrand2010mean} Lemma 3.2.3] 
It holds that $r \geq \sigma + \tau$.
\label{lemma:ST_r_>=_rbar}
\end{lemma}

In order to show that $\nu f \simeq \nu_0 f$, we will employ a technique that goes through the `limiting end' of both the original Hamiltonian and $-H_{N,M,0}$. Define
\begin{align}
    -H_{N,M,0\textnormal{-lim}}(\vec{x}) &:= \sum_{m \leq M} u\inparen{ \frac{1}{\sqrt{N}}\sum_{k+1 \leq i \leq N} g_{i,m} x_i   } - \kappa\norm{\vec{x}}^2 + h\sum_{i \leq N} g_i x_i \nonumber\\
    &\quad\quad\quad\quad\quad + \sqrt{r}\sum_{j \leq k} x_j z_j  + \sum_{j \leq k} \log \Eb_\xi \insquare{ \exp x_j \sqrt{\tau - r}\xi_j  } +  \frac{\sigma}{2} \sum_{j \leq k} x_j^2,
    \label{eq:ST_H_N,M,0lim}
\end{align}
where $(z_j)_{j \leq k}$, $(\xi_j)_{j \leq k}$ are independent standard Gaussians. Associated Gibbs expectations are denoted by $\inangle{\cdot}_{\textnormal{0-lim}}$. In \eqref{eq:ST_H_N,M,0lim}, we have introduced new disorder $z_j$'s---this is what is meant by limiting end, because we know that $\sqrt{r}z_j$ arises from the limiting distribution of $N^{-1/2} \sum_{m \leq M} g_{j,m} \inangle{u'(S_m^0)}^{-}$, where $\inangle{\cdot}^-$ refers to the truncated $(N-k)$-system implied by \eqref{eq:ST_H_N,M,0}, given by \eqref{eq:ST_Hamiltonian_truncated_N-k} below (note that $N$, $M$ are still finite quantities). In \eqref{eq:ST_H_N,M,0lim}, it is straightforward to evaluate the expectation in $\xi_j$'s, which produces a term $\sum_{j \leq k} x_j^2 (\tau-r) /2$, which could be combined with the last term. However, we prefer this definition as it more clearly reveals the link between $-H_{N,M,0}$ and $-H_{N,M,0\textnormal{-lim}}$. 

The main advantage of $-H_{N,M,\textnormal{0-lim}}$ is that the strong-concavity property in \eqref{eq:-H_N,M_strongConcavity} is preserved. This is a consequence of showing that $\tau + \sigma - r \leq 0$ as in Lemma \ref{lemma:ST_r_>=_rbar}. This allows many properties of $-H_{N,M}$: thin-shell, overlap concentration, etc., to be recovered in the interpolating system. 


\subsection{Truncated Shcherbina-Tirozzi model}

Let $-H^{-}_{N-k,M,0\textnormal{-lim}}$ be the Hamiltonian of the truncated $N-k$ system implied by \eqref{eq:ST_H_N,M,0lim} (this is the same as that implied by \eqref{eq:ST_H_N,M,0}),
\begin{align}
    -H^{-}_{N-k,M}(x_{k+1},\dots,x_N) &= \sum_{m \leq M} u\inparen{S_{m}^{0}} - \kappa \sum_{k+1 \leq i \leq N} x_i^2 + h\sum_{k+1 \leq i \leq N} g_i x_i,
    \nonumber \\
    &= \sum_{m \leq M} \tilde{u}\inparen{S_{m}^{-}} - \kappa \sum_{k+1 \leq i \leq N} x_i^2 + h\sum_{k+1 \leq i \leq N} g_i x_i,
    \label{eq:ST_Hamiltonian_truncated_N-k}
\end{align}
with expectations denoted by $\nu^- \insquare{f} = \Eb \inangle{f}^-$ and partition function $Z_{N-k,M}^{-}$, and where, as in the truncated system for Perceptron \eqref{eq:perceptron_H^-_N-k,M}, $\tilde{u}$ is the smooth function defined by $\tilde{u}(y) := u\inparen{ \sqrt{\frac{N-k}{N}} y }$, so that $S_m^{-} = (N-k)^{-1/2} \sum_{k+1 \leq i \leq N} g_{i,m} x_i$. It is clear that the assumptions \eqref{eq:ST_u_asmpts} and \eqref{eq:ST_other_asmpts} are satisfied with $\tilde{u}$ replacing $u$, and with $\alpha^- := M/(N-k) \leq K$. The purpose of these remarks is to assert that the nice properties of the ST model (e.g. thin-shell and overlap concentration of main and auxiliary systems) in the original $(N,M)$-system will carry forward to this truncated $(N-k,M)$-system.

We note the useful relations
\begin{align}
    \tilde{u}^{\,\prime}(S_m^{-}) = \sqrt{ \frac{N-k}{N} } u'(S_m^0); \quad\quad \tilde{u}^{\,\prime\prime}(S_m^{-}) =  \frac{N-k}{N}  u''(S_m^0)
    \label{eq:ST_u_utilde_relation}
\end{align}

Importantly, the RS parameters associated to this truncated $(N-k)$-system are as follows. Let $z$, $\xi$ be independent standard Gaussians, and let $\theta^{-} = \sqrt{q^{-}}z + \sqrt{\rho^- - q^{-}}\xi$, and define.
\begin{align}
    &r^- = \alpha^- \Eb \insquare{ \inparen{ \frac{\Eb_\xi \insquare{ \tilde{u}^{\,\prime}(\theta^-) \exp \tilde{u}(\theta^-)   }}{\Eb_\xi \exp \tilde{u}(\theta^-)} }^{2} }; \quad\quad \tau^- = \alpha^- \Eb \frac{\Eb_\xi \insquare{ \tilde{u}^{\,\prime 2}(\theta^-) \exp \tilde{u}(\theta^-)   }}{\Eb_\xi \exp \tilde{u}(\theta^-)}; \nonumber\\ 
    &\quad\quad\quad\quad\quad\quad\quad\quad\quad \sigma^- = \alpha^- \Eb \frac{\Eb_\xi \insquare{ \tilde{u}^{\prime\prime}(\theta^-) \exp \tilde{u}(\theta^-)   }}{\Eb_\xi \exp \tilde{u}(\theta^-)},
    \label{eq:ST_truncated_RS_r_tau_sigma}
\end{align}
and
\begin{align}
    \rho^- = \frac{1}{2\kappa + r^- - \sigma^- - \tau^-} + \frac{r^- + h^2}{(2\kappa + r^- - \sigma^- - \tau^-)^2}; \quad\quad q^- = \frac{r^- + h^2}{(2\kappa + r^- - \sigma^- - \tau^-)^2}.
    \label{eq:ST_truncated_RS_rho_q}
\end{align}
We find it convenient to use \eqref{eq:ST_u_utilde_relation} to rewrite \eqref{eq:ST_truncated_RS_r_tau_sigma} as
\begin{align}
    &r^- = \alpha \Eb \insquare{ \inparen{ \frac{\Eb_\xi \insquare{ u^{\,\prime}\inparen{\sqrt{\frac{N-k}{N}} \theta^- } \exp u\inparen{\sqrt{\frac{N-k}{N}} \theta^- }   }}{\Eb_\xi \exp u\inparen{\sqrt{\frac{N-k}{N}} \theta^- } } }^{2} }; \quad \tau^- = \alpha \Eb \frac{\Eb_\xi \insquare{ u^{\,\prime 2}\inparen{\sqrt{\frac{N-k}{N}} \theta^- } \exp u\inparen{\sqrt{\frac{N-k}{N}} \theta^- }   }}{\Eb_\xi \exp u\inparen{\sqrt{\frac{N-k}{N}} \theta^- }}; \nonumber\\ 
    &\quad\quad\quad\quad\quad\quad\quad\quad\quad\quad\quad\quad\quad \sigma^- = \alpha \Eb \frac{\Eb_\xi \insquare{ u^{\prime\prime}(\inparen{\sqrt{\frac{N-k}{N}} \theta^- } \exp u\inparen{\sqrt{\frac{N-k}{N}} \theta^- }   }}{\Eb_\xi \exp u\inparen{\sqrt{\frac{N-k}{N}} \theta^- }}.
    \label{eq:ST_truncated_RS_r_tau_sigma_convenientForm}
\end{align}

\subsection{Thin-shell and overlap concentration for main and auxiliary systems in Shcherbina-Tirozzi}
\label{sec:TSOC_for_mainAuxiliary_ST}

Thin-shell and overlap concentration for the main ST system was proven by Talagrand \cite{talagrand2010mean}, and stated as follows. The analogous result for the truncated system \eqref{eq:ST_Hamiltonian_truncated_N-k} is a straightforward adjustment for $\alpha^-$ and $\tilde{u}$.

\begin{theorem}[Essentially in \cite{talagrand2010mean} Theorems 3.1.18 and 3.2.14]
For $p \leq N/4$,
\begin{align}
    \nu\insquare{  \inparen{ R_{1,1} - \rho }^{2p}  } \leq \inparen{ \frac{Kp}{N}  }^{k}; \quad\quad \nu\insquare{  \inparen{ R_{1,2} - q }^{2p}  } \leq \inparen{ \frac{Kp}{N}  }^{k},
\end{align}
where $\rho$ and $q$ satisfy \eqref{eq:ST_RS_rho_q}.
\label{thm:ST_TSOC_mainSystem}
\end{theorem}

\begin{proof}
Write $\nu\insquare{ \inparen{ R_{1,1} - q }^{2p} } \leq 2^{2p}\inparen{ \nu\insquare{ \inparen{ R_{1,1} - \nu\insquare{R_{1,1}} }^{2p} } + \abs{  \nu\insquare{R_{1,1}} - q  }^{2p} }$ and apply \cite{talagrand2010mean} Theorems 3.1.18 and 3.2.14 on the first and second terms respectively, using the symmetry among sites to get that $\nu\insquare{R_{1,1}} = \nu\insquare{x_1^{2}}$ and $\rho = \nu_{0\textnormal{-lim}}\insquare{x_1^{2}}$. The case for $R_{1,2}$ is similar.
\end{proof}

For thin-shell and overlap concentration in the auxiliary system $\inparen{ u'(S_m)  }_{m \leq M}$, note that the argument used for the Ising Perceptron (Theorem \ref{thm:perceptron_TSOC_auxiliarySystem}) will not be satisfactory here---there we had used a condition such as ``$\alpha K \leq 1$'' for a type of contraction argument as in \eqref{perceptron_TSOC_auxiliary_contractionPart}, to show that adding a layer in $M$ improves the concentration of $(\alpha u^{\prime 2}(S_M) - \tau)^2$. Since the ST Gibbs measure is log-concave for every realization of disorder, we always have overlap concentration of the main system, and there is a notion that the system is always `in high-temperature', so we will not use such a condition.

Theorem \ref{thm:ST_TSOC_auxiliary} is stated for the $(N,M)$-system, which is of independent interest. A similar result is obtained in \cite{barbier2021performance} Theorem 3.3, in a slightly different setting, and with a proof that seems different from that of Theorem \ref{thm:ST_TSOC_auxiliary}. 


\begin{theorem}
Let $\tau$ and $r$ be given as in \eqref{eq:ST_RS_r_tau_sigma}. Then 
\begin{align}
    \nu \insquare{ \inparen{  \frac{1}{N} \sum_{m \leq M} u^{\prime 2} (S_m) - \tau }^{2}  } \leq \frac{K(D,\kappa_0,h_0)}{N}
    \label{eq:ST_TS_auxiliarySystem}
\end{align}
and 
\begin{align}
    \nu \insquare{ \inparen{  \frac{1}{N} \sum_{m \leq M} u^{\prime} (S^{1}_m) u^{\prime} (S^{2}_m) - r }^{2}  } \leq \frac{K(D,\kappa_0,h_0)}{N}
    \label{eq:ST_OC_auxiliarySystem}
\end{align}
\label{thm:ST_TSOC_auxiliary}
\end{theorem}

The results extend to the truncated Shcherbina-Tirozzi system, and we immediately get the following.

\begin{corollary}
Let $\tau^-$ and $r^-$ be given as in \eqref{eq:ST_truncated_RS_r_tau_sigma}, then 
\begin{align}
    \nu^- \insquare{ \inparen{  \frac{1}{N-k} \sum_{m \leq M} \tilde{u}^{\prime 2} (S^-_m) - \tau^- }^{2}  } \leq \frac{K(D,\kappa_0,h_0)}{N-k},
    \label{eq:ST_TS_truncated_auxiliarySystem}
\end{align}
and 
\begin{align}
    \nu^- \insquare{ \inparen{  \frac{1}{N-k} \sum_{m \leq M} \tilde{u}^{\prime} (S^{-,1}_m) \tilde{u}^{\prime} (S^{-,2}_m) - r }^{2}  } \leq \frac{K(D,\kappa_0,h_0)}{N-k}.
    \label{eq:ST_OC_truncated_auxiliarySystem}
\end{align}
\label{corollary:ST_TSOC_auxiliary_truncated}
\end{corollary}

\subsection{Decomposing the Shcherbina-Tirozzi Hamiltonian}
\label{sec:ST_decomposeHamiltonian}

The purpose of this section is to establish the following decomposition result for the ST Hamiltonian.

\begin{theorem}
Let $B$ be a measurable subset of $\RR^k$. For every $0 < \epsilon < 1/4$, for every $p \geq 1$, whenever $\kappa$ satisfies \eqref{eq:ST_kappa0_assumption}, and whenever $N \geq k + \exp \inparen{ \sqrt{8kp}/\epsilon }$, we have
\begin{align}
    \Eb_{\ud} \insquare{ \inparen{ G_{N}^{(k)}\insquare{B} - G_{N,0}^{(k)}\insquare{B} }^{2p}  } \leq \frac{K}{(N-k)^{1/4 - \epsilon}},
\end{align}
where $K = K(\epsilon, p, k, \kappa_0, h_0, D) > 0$ is some constant.
\label{thm:ST_Hamiltonian_decomposition}
\end{theorem}

The proof of Theorem \ref{thm:ST_Hamiltonian_decomposition} goes through Propositions \ref{proposition:ST_nuf_to_nu0limf}, \ref{proposition:ST_nu0_to_tildenu0lim}, and \ref{proposition:ST_nu0lim_to_tildenu0lim} stated below. We first copy the Hamiltonians that are involved here for clarity.
The original ST Hamiltonian \eqref{eq:ST_Hamiltonian_original}:
\begin{align*}
    -H_{N,M}(\vec{x}) &:= \sum_{m \leq M} u\inparen{ \frac{1}{\sqrt{N}}\sum_{i \leq N} g_{i,m}x_i   } - \kappa \norm{\vec{x}}^2 + h\sum_{i \leq N} g_i x_i,
\end{align*}
is associated to $\nu f$; the desirable decomposable Hamiltonian \eqref{eq:ST_H_N,M,0} is
\begin{align*}
    -H_{N,M,0}(\vec{x}) &:= \sum_{m \leq M} u\inparen{ \frac{1}{\sqrt{N}}\sum_{k+1 \leq i \leq N} g_{i,m} x_i   } - \kappa\norm{\vec{x}}^2 + h\sum_{i \leq N} g_i x_i \nonumber\\
    &\quad\quad\quad\quad\quad\quad + \sum_{j \leq k} x_j \sum_{m \leq M} \frac{g_{j,m}}{\sqrt{N}} u'\inparen{ \frac{1}{\sqrt{N}}\sum_{k+1 \leq i \leq N} g_{i,m} x_i}  +  \frac{\sigma^-}{2} \sum_{j \leq k} x_j^2,
\end{align*}
which is associated to $\nu_0 f$; an intermediate Hamiltonian \eqref{eq:ST_tilde_H_N,M,0lim} between the above two is
\begin{align*}
    -\tilde{H}_{N,M,0\textnormal{-lim}}(\vec{x}) &= \sum_{m \leq M} u\inparen{ \frac{1}{\sqrt{N}}\sum_{k+1 \leq i \leq N} g_{i,m} x_i   } - \kappa\norm{\vec{x}}^2 + h\sum_{i \leq N} g_i x_i \nonumber\\
    &\quad\quad\quad\quad\quad + \sqrt{r^-}\sum_{j \leq k} x_j z_j  + \sum_{j \leq k} \log \Eb_\xi \insquare{ \exp x_j \sqrt{\tau^- - r^-}\xi_j  } +  \frac{\sigma^-}{2} \sum_{j \leq k} x_j^2,
\end{align*}
which is associated to $\tilde{\nu}_{0\textnormal{-lim}} f$.

\begin{proposition}
Let $f$ be a function on $\inparen{\RR^N}^n$. Suppose that $\norm{f}_{\infty} \leq 1$, and that $f$ is independent of the randomness of $\inparen{ g_{i,M} }_{i \leq N}$, $z$, $\inparen{\xi^\ell}_{\ell \leq n}$, then
\begin{align}
    \abs{\nu\!\insquare{f} - \nu_{0\textnormal{-lim}}\insquare{f}} \leq \frac{K}{\sqrt{N}}.
\end{align}
\label{proposition:ST_nuf_to_nu0limf}
\end{proposition}

Consider functions $f$ on $\inparen{ \RR^{N} }^{n}$ that are replicas of indicators of cylinder sets; that is, let $B$ be a measurable subset of $\RR^k$, and recall that $\rchi_B = \rchi_{\inbraces{ \vec{x} \in \RR^N :   \inparen{  x_1, \dots, x_k } \in B  }}$, and let 
\begin{align}
    f(\vec{x}^1,\dots,\vec{x}^n) = \rchi_{B}(\vec{x}^1) \cdots \rchi_{B}(\vec{x}^n).
    \label{eq:ST_f_cylinderIndicator}
\end{align}

\begin{proposition}
Let $0 < \epsilon < 1/4$, and let $f$ be a function as in \eqref{eq:ST_f_cylinderIndicator} on $\inparen{ \RR^N }^{2p}$. Then whenever $N \geq k + \exp \inparen{ \sqrt{8kp}/\epsilon }$,
\begin{align*}
    \abs{\nu_0 \insquare{f} - \tilde{\nu}_{0\textnormal{-lim}} \insquare{f} }  \leq \frac{K \sqrt{C}}{(N-k)^{1/4 - \epsilon}},
\end{align*}
where 
\begin{align}
    C &:= \inparen{ \sqrt{ \frac{2\kappa - \sigma^-}{2\pi}  } \int \abs{x} \exp \inparen{ -\frac{1}{2}\inparen{ 2\kappa - \sigma^- - 16k^2p^3 (1 + 8\bar{\epsilon}^2)(\alpha D^2 + h^2) } x^2  }  }^{2kp},
\end{align}
where $\bar{\epsilon} = \frac{1}{4\epsilon} - 1$. Moreover, under \eqref{eq:ST_asmpt_on_kappa}, $C < \infty$.
\label{proposition:ST_nu0_to_tildenu0lim}
\end{proposition}

\begin{proposition}
Under \ref{eq:ST_asmpt_on_kappa_FOR_JACOBIAN}, for any function $f$ on $(\RR^N)^n$ such that $\norm{f}_{\infty} \leq 1$, we have
\begin{align}
    \abs{\nu_{0\textnormal{-lim}}\insquare{f} - \tilde{\nu}_{0\textnormal{-lim}}\insquare{f}} \leq \frac{K}{N}.
\end{align}
\label{proposition:ST_nu0lim_to_tildenu0lim}
\end{proposition}

Propositions \ref{proposition:ST_nuf_to_nu0limf}, \ref{proposition:ST_nu0_to_tildenu0lim}, and \ref{proposition:ST_nu0lim_to_tildenu0lim} are proved in Sections \ref{sec:ST_nuf_to_nu0limf}, \ref{sec:ST_nu0_to_tildenu0lim}, and \ref{sec:ST_nu0lim_to_tildenu0lim} respectively, whence Theorem \ref{thm:ST_Hamiltonian_decomposition} will follow from triangle inequalities and the similar arguments as in the proof of Theorem \ref{thm:SK_classical_decomposition}. Strictly speaking, additional work is required to address the limitation in Proposition \ref{proposition:ST_nuf_to_nu0limf} that the function must be independent of the `randomness in $M$'---this will be resolved as in Corollary \ref{thm:perceptron_decomposeHamiltonian_annealed_M+1} and Theorem \ref{thm:perceptron_+_to_originalSystem} in the Perceptron case: we will go to the $(N,M+1)$-system; this is done in Section \ref{sec:ST_Hamiltonian_decomposition}. This does not bring much additional difficulty except tediousness.

\subsubsection{Proof of Theorem \ref{thm:ST_Hamiltonian_decomposition}}
\label{sec:ST_Hamiltonian_decomposition}

The argument is largely similar to that of Theorem \ref{thm:perceptron_Hamiltonian_decomposition}. In the process we have to address the limitation in Proposition \ref{proposition:ST_nuf_to_nu0limf} that the function $f$ must be independent of the randomness in $M$. As in the Perceptron case, we move to the $(N,M+1)$-system. Define the $(N,M+1)$ analog of \eqref{eq:ST_Hamiltonian_H_N,M,tlim}:
\begin{align}
    -H_{N,M+1,t\textnormal{-lim}}(\vec{x}) &:= \sum_{m \leq M+1} u\inparen{S_{m,t}} - \kappa \norm{\vec{x}}^{2} + h\sum_{i \leq N} g_i x_i + \sqrt{1 - t} \sum_{j \leq k} x_j \sqrt{r}z_j - (1-t)(r - \bar{r}) \sum_{j \leq k} \frac{x_j^2}{2},
    \label{eq:ST_Hamiltonian_H_N,M+1,tlim}
\end{align}
with associated expectations $\nu^+ = \Eb \inangle{\cdot}^+$. Let $\alpha^+ := (M+1)/N$, and let $r^+, \tau^+, \sigma^+, \rho^+, q^+$ be the RS quantities \eqref{eq:ST_RS_r_tau_sigma}, \eqref{eq:ST_RS_rho_q} extended to this setting. The $t\textnormal{-lim}$ and $v$-interpolations in Section \ref{sec:ST_nuf_to_nu0limf} go through as before, just that now we create a ``cavity-in-$(M+1)$, and we arrive at the following.

\begin{corollary}
Let $f$ be a function on $\inparen{\RR^N}^n$. Suppose that $\norm{f}_{\infty} \leq 1$ and that $f$ is independent of the randomness in $\inparen{ g_{i,M+1} }_{i \leq N}$, $\inparen{z_j}_{j \leq k}$, $\inparen{\xi^\ell}_{\ell \leq n}$, then
\begin{align}
    \abs{\nu^+\insquare{f} - \nu^+_{0\textnormal{-lim}}\insquare{f}} \leq \frac{K}{\sqrt{N}}.
\end{align}
\label{corollary:ST_nuf_to_nu0limf_+_system}
\end{corollary}

\begin{proposition}
Let $f$ be a function on $\inparen{\RR^N}^n$. Suppose that $\norm{f}_{\infty} \leq 1$ and that $f$ is independent of the randomness in $\inparen{ g_{i,M+1} }_{i \leq N}$, $\inparen{z_j}_{j \leq k}$, $\inparen{\xi^\ell}_{\ell \leq n}$, then
\begin{align}
    \abs{\nu^+\insquare{f} - \nu\insquare{f}} \leq \frac{K}{\sqrt{N}}; \quad \textnormal{and} \quad \abs{ \nu^+_{0\textnormal{-lim}}\insquare{f} - \nu_{0\textnormal{-lim}}\insquare{f}} \leq \frac{K}{\sqrt{N}}.
\end{align}
\label{proposition:ST_+_to_originalSystem}
\end{proposition}

The approach is similar to that of Theorem \ref{thm:perceptron_+_to_originalSystem} and given in Appendix \ref{sec:suppProofs_for_ST}.

\begin{proofof}{Theorem \ref{thm:ST_Hamiltonian_decomposition}} Denote $\rchi_B^{\ell} := \rchi_B(\vec{x}^{\ell})$ (not to be confused with powers of the indicator function). Using \eqref{eq:(1-1)^n_expand_identity} and replicas, we have
\begin{align}
    &\Eb_{\ud} \insquare{ \inparen{ G_{N}^{(k)}\insquare{B} - G_{N,0}^{(k)}\insquare{B} }^{2p}  } = \Eb_{\textnormal{d}} \insquare{\inparen{ \inangle{\rchi_{B}} - \inangle{\rchi_B}_{0}  }^{2p}  } \nonumber\\
    &= \sum_{1 \leq r \leq 2p} \inparen{-1}^{2p - r} \binom{2p}{r} \Eb_{\textnormal{d}} \insquare{ \inangle{\rchi_{B}^{1} \cdots \rchi_{B}^{r} } \inangle{ \rchi_B^{1}\dots \rchi_{B}^{2p - r}  }_{0} - \inangle{ \rchi_{B}^{1} \cdots \rchi_B^{2p} }_{0}  }.
    \label{eq:ST_quenchedCloseness_betweenoriginal_and_decomposed_expansion}
\end{align}
For every $1 \leq r \leq 2p$, write
\begin{align*}
    &\abs{ \Eb_{\textnormal{d}} \insquare{ \inangle{\rchi_{B}^{1} \cdots \rchi_{B}^{r} } \inangle{ \rchi_B^{1}\dots \rchi_{B}^{2p - r}  }_{0} - \inangle{ \rchi_{B}^{1} \cdots \rchi_B^{2p} }_{0}  } }\\
    &\quad = \abs{ \Eb_{\textnormal{d}} \insquare{ \inangle{\rchi_{B}^{1} \cdots \rchi_{B}^{r} \inangle{ \rchi_B^{1}\dots \rchi_{B}^{2p - r}  }_{0} }  - \inangle{ \rchi_{B}^{1} \cdots \rchi_B^{r} \inangle{ \rchi_B^{1}\dots \rchi_{B}^{2p - r}  }_{0} }_{0}  } }\\
    &\quad = \abs{\nu\insquare{f} - \nu_0\insquare{f}},
\end{align*}
where $f$ is the function on $\Sigma_{N}^{r}$ defined by $f := \rchi_{B}^{1} \cdots \rchi_{B}^{r} \inangle{ \rchi_B^{1}\dots \rchi_{B}^{2p - r}  }_{0}$. Note that this function $f$ depends on the randomness in $g_{i,M}$'s. Therefore, on the $\nu f \simeq \nu_{0\textnormal{-lim}}f$ path, we will need to move to the $\nu^+$ system, and then we can exploit the fact that $f$ is independent of the randomness in $g_{i,M+1}$'s.

We have
\begin{align}
    \abs{\nu\insquare{f} - \nu_0\insquare{f}} &\leq \abs{\nu\insquare{f} - \nu^+\insquare{f}} + \abs{\nu^+\insquare{f} - \nu^+_{0\textnormal{-lim}}\insquare{f}} + \abs{\nu^{+}_{0\textnormal{-lim}}\insquare{f} - \nu_{0\textnormal{-lim}}\insquare{f}} \nonumber \\
    &\quad\quad\quad  + \abs{\nu_{0\textnormal{-lim}}\insquare{f} - \tilde{\nu}_{0\textnormal{-lim}}\insquare{f}} + \abs{\tilde{\nu}_{0\textnormal{-lim}}\insquare{f} - \nu_0\insquare{f}} \label{eq:ST_decomposition_nuf_minus_nu0f_triangleDecomp} \\
    &\leq \frac{K}{(N-k)^{1/4 - \epsilon}},
    \label{eq:ST_decomposition_nuf_minus_nu0f_bound}
\end{align} 
where in \eqref{eq:ST_decomposition_nuf_minus_nu0f_triangleDecomp}, we have used, from left to right, Proposition \ref{proposition:ST_+_to_originalSystem}, Corollary \ref{corollary:ST_nuf_to_nu0limf_+_system}, Proposition \ref{proposition:ST_+_to_originalSystem} (again), Proposition \ref{proposition:ST_nu0lim_to_tildenu0lim}, and Proposition \ref{proposition:ST_nu0_to_tildenu0lim}. Plugging \eqref{eq:ST_decomposition_nuf_minus_nu0f_bound} back into \eqref{eq:ST_quenchedCloseness_betweenoriginal_and_decomposed_expansion} completes the proof.
\end{proofof}

\subsection{Local independence for Shcherbina-Tirozzi model}

By Theorem \ref{thm:ST_Hamiltonian_decomposition}, the ST Hamiltonian \eqref{eq:ST_H_N,M,0} is $K/(N^{1/4 - \epsilon})$-decomposable with
\begin{align}
    -H_{N,M,0}(\vec{x}) = -H_{N-k, M}(\vec{y}) - \inparen{\kappa + \frac{\sigma^-}{2}} \sum_{j \leq k} x_j^2 + h\sum_{j \leq k} g_j x_j + \sum_{j \leq k} x_j \sum_{m \leq M} \frac{g_{j,m}}{\sqrt{M}} \tilde{u}^{\, \prime}(S_m^-),
\end{align}
where
\begin{align}
    -H_{N-k,M}^{-}(\vec{y}) &= \sum_{m \leq M} \tilde{u}(S_m^-) - \kappa\sum_{k+1 \leq i \leq N} x_i^2 + h \sum_{k+1 \leq i \leq N} g_i x_i.
\end{align}

\begin{theorem}[Local independence, Shcherbina-Tirozzi model]
Let $\inangle{\cdot}$ be the ST Gibbs measure that satisfies \eqref{eq:ST_u_asmpts}, \eqref{eq:ST_other_asmpts}, \eqref{eq:ST_kappa0_assumption}. Let $B$ be a measurable subset of $\RR^k$. Denote by $G_{N}^{(k)}$ the Gibbs marginal on the first $k$-coordinates. Then for every $p \geq 1$, for every $0 < \epsilon < 1/4$, whenever $N \geq \max \inparen{k + \exp \inparen{ \sqrt{8kp}/\epsilon }, \,  \exp\inparen{\sqrt{16k}/\epsilon} }$,
\begin{enumerate}
    \item we have
    \begin{align}
        \Eb_{\ud}\insquare{ \sup_{B \subseteq \RR^k} \inparen{ G_{N}^{(k)}\!\insquare{B} -  \frac{1}{Z_k} \int_B \prod_{j \leq k} \exp\inparen{ x_j \vec{a}_j^{\top} \inangle{ \Ub }^-  - \frac{x_j^2}{2}(2\kappa + r^- - \tau^- - \sigma^- ) + hg_j x_j } \ud x_j  }^{2p}  } \leq \frac{K}{N^{\frac{1}{4} - \epsilon}  },
        \label{eq:ST_localindependence}
    \end{align}
    where $\vec{a}_j := \inparen{\frac{1}{\sqrt{N}} g_{j,m}  }_{m \leq M}$ and $\Ub := \inparen{ u'(S_m^0)  }_{m \leq M}$;
    \item furthermore, for independent standard Gaussians $(z_j)_{j \leq k}$ independent of everything else,
    \begin{align}
        \Eb_{\ud}\Eb_{\vec{z}}\insquare{ \sup_{B \subseteq \RR^k} \inparen{ G_{N}^{(k)}\!\insquare{B} -  \frac{1}{Z_k} \int_B \prod_{j \leq k} \exp\inparen{ x_j \sqrt{r^-} z_j  - \frac{x_j^2}{2}(2\kappa + r^- - \tau^- - \sigma^- ) + hg_j x_j } \ud x_j  }^{2p}  } \leq \frac{K}{N^{\frac{1}{4} - \epsilon}  },
        \label{eq:ST_localindependence_limiting}
    \end{align}
\end{enumerate}
where $Z_k$ is a normalizing constant, and where $r^-$, $\tau^-$, $\sigma^-$ satisfy \eqref{eq:ST_truncated_RS_r_tau_sigma}, and for constants $K$ depending on $k, p, D, \epsilon, h, \kappa$.
\end{theorem}

\begin{proof}
Apply Theorem \ref{thm:LI_abstract_theorem} with $\varrho \equiv 1$, $w_m = \sqrt{\alpha^-} \tilde{u}^{\, \prime} (S_m^-)$, $\Xi = \tau^-$, $\Upsilon = r^-$, $f_j(x_j) = -(\kappa + \frac{\sigma^-}{2}) x_j^2 + h g_j x_j $, and with $\vec{A}_j = \inparen{ \frac{1}{\sqrt{M}} g_{j,m} }_{m \leq M}$, and $\mu$ is the measure proportional to Lebesgue measure, i.e.~$\mu(\ud x_j) = C_0 (\ud x_j)$, so that
\begin{align*}
    \int_{\RR} \exp\inparen{ - \frac{x_j^2}{2}(2\kappa + r^- - \tau^- - \sigma^-)   } \mu(\ud x_j) = 1,
\end{align*}
that is, $C_0$ is the normalizing constant of the centered Gaussian measure of variance 
\begin{align}
    V^2 := 1/(2\kappa + r^- - \tau^- -\sigma^-).
    \label{eq:ST_V^2}
\end{align}
Note that $r^- \geq tau^- + \sigma^-$ by Lemma \ref{lemma:ST_r_>=_rbar} applied for the truncated system. It is verified that hypothesis 2 is satisfied since for any $P \in \RR$,
\begin{align*}
    \int_{\RR} \exp\inparen{ - \frac{x_j^2}{2}(2\kappa + r^- - \tau^- - \sigma^-) + x_j(h g_j + P)   } \mu(\ud x_j) = \exp\inparen{ \frac{1}{2}V^2(h g_j + P)^2   } \geq 1.
\end{align*}
As in the proof of Theorem \ref{thm:perceptron_LI}, it is more convenient to apply Theorem \ref{thm:LI_abstract_theorem} with $\vec{A}_j$ and $\vec{w}$ so defined, then later we use the relation
\begin{align*}
    \vec{A}_j^{\top} \inangle{\vec{w}}^-  = \sum_{m \leq M} \frac{g_{j,m}}{\sqrt{M}} \sqrt{\alpha^-} \tilde{u}^{\, \prime }(S_m^-) = \sum_{m \leq M} \frac{g_{j,m}}{\sqrt{N}} u^{\, \prime }(S_m^0)  =\vec{a}_j^{\top} \inangle{\Ub}^-
\end{align*}
to convert interchangeably. Hypothesis 1 in Theorem \ref{thm:LI_abstract_theorem} is satisfied by Theorem \ref{corollary:ST_TSOC_auxiliary_truncated}; note that the accuracy there is $K/(N-k)$, but since the ratio $M/(N-k) = \alpha^-$ is of constant order, this is really $K/M$. It remains to verify that the quantity $C$ defined by Theorem \ref{thm:LI_abstract_theorem} is finite. We have
\begin{align}
    C &\leq \sum_{j \leq k} \Eb \int_{\RR} \abs{x_j} \exp \inparen{ - \inparen{ \kappa + \frac{\sigma^-}{2} -  48 k^2 p ^2 (1 + 4\bar{\epsilon}^2)\alpha D^2 } x_j^2 + hg_j x_j  } \mu(\ud x_j),
    \label{eq:ST_LI_C}
\end{align}
using that $\inparen{ \sum_{j} \abs{x_j} }^2 \leq 2 \sum_{j} x_j^2$, and $\alpha^- \leq 2k \alpha$, and $\tau^- , r^- \leq \alpha D^2$. It follows that when \eqref{eq:ST_asmpt_on_kappa} is satisfied, i.e.~$\kappa$ is large enough, then the coefficient of $x_j^2$ in the exponent in \eqref{eq:ST_LI_C} is negative, and the expectation exists.
\end{proof}

\begin{theorem}[Converse to local independence, Shcherbina-Tirozzi] Suppose that assumptions \eqref{eq:ST_u_asmpts}, \eqref{eq:ST_other_asmpts}, \eqref{eq:ST_kappa0_assumption} are in force, and that \eqref{eq:ST_localindependence_limiting} holds for $k \leq 2$, $p = 1$. Let $\nu_N\!\insquare{\cdot} = \Eb_{\ud} \inangle{\cdot}$, and choose $0 < \delta < 1/16$. Then
\begin{align}
    \nu_N\!\insquare{ \inparen{ R_{12} - \nu_N R_{12}  }^2   } \leq \frac{K}{N^{1/16 - \delta}}, \quad \textnormal{and} \quad \nu_N\!\insquare{ \inparen{ R_{11} - \nu_N R_{11}  }^2   } \leq \frac{K}{N^{1/16 - \delta}},
    \label{eq:ST_converseToLI}
\end{align}
for some constant $K$ that depends on $\kappa, D, h$ and $\delta$.
\end{theorem}

\begin{proof}
    Apply Theorem \ref{thm:LI=>TSOC_abstract} with $\eta = 1/8 - 2\delta$, $\HH_j(\ud x_j) \propto \exp \inparen{ x_j \sqrt{r^-} z_j - x_j^2(2\kappa + r^- - \tau^- - \sigma^-)/2 + h g_j x_j  } \ud x_j$ (defined conditionally on $z_j$ and $g_j$).

    With $\Eb \HH_j =: \HH$, it holds that $\HH x^4 \leq K < \infty$. To see this, observe that $\HH_j$ is a Gaussian distribution with variance $V^2$ as in \eqref{eq:ST_V^2}, and with mean depending on $z_j$, $g_j$, and $h$. Therefore $\HH_j x^4$ is a polynomial in $z_j$, $g_j$, $h$, and $V^2$. In turn, $z_j$ and $g_j$ are independent standard Gaussians and so $\Eb \HH_j x^4$ exists.

    We also have $\nu_N^{(1)} x^4 \leq K < \infty$ for some $K$ independent of $N$. This follows from \cite{talagrand2010mean} Lemma 3.2.6, which states that there exists a $K > 0$ independent of $N$ such that $\nu^{(1)}_N\insquare{ \exp \inparen{ x^2 / K  }  } \leq K$.

    That $\nu_N^{(1)}x^8$ and $\HH x^8$ are bounded independently of $N$ follows similarly.
\end{proof}

\appendix

\section{Supplementary proofs for Section \ref{sec:LI_abstract}}
\label{sec:suppProofs_LI_abstract}

We reproduce here the main result of \cite{wee2023random} Corollary 1.4.



In what follows, $\vec{x}$ will refer to a random vector in $\RR^N$, and $\Theta$ will refer to a $N \times k$ matrix with iid entries of zero-mean Gaussians with variance $1/N$, with $k < N$, and with $\Theta$ is independent of all other sources of randomness. As an additional layer of complexity seen in disordered systems, suppose that the law of $\vec{x}$, denoted by $\inangle{\cdot}$, is itself random wrt.~$\Eb_{\ud}$, the disorder.

\begin{theorem}
Let $0 \leq q < \rho$ be constants such that
\begin{align}
    \Eb_{\textnormal{d}}\inangle{ \inparen{\frac{\norm{\vec{x}}^2}{N} - \rho}^{2}  } \leq c_1, \quad \textnormal{and} \quad \Eb_{\textnormal{d}}\inangle{ \inparen{\frac{\vec{x}^1 \boldsymbol{\cdot} \vec{x}^2}{N} - q}^{2}  } \leq c_2,
    \label{eq:projResult_TSOC_hypotheses}
\end{align}
for some numbers $c_1$, $c_2$ that may depend on $N$. Let $\bxi$ be a standard Gaussian vector in $\RR^k$, independent of all other sources of randomness. Then for every integer $p \geq 1$, for constants $K_1, K_2 > 0$ that depend on $p, \rho, q$, and for $d_1, d_2$ defined in \eqref{eq:d1_d2_rates}, the following happens.
\begin{enumerate}
    \item We have
    \begin{align}
        &\sup_{\substack{\norm{g}_{\textnormal{Lip}} \leq L \\ \norm{g}_{\infty} \leq M} } \Eb_{\ud}\Eb_{\Theta}  \insquare{ \inparen{ \inangle{g(\Theta^\top \vec{x})  } - \Eb_{\bxi} \insquare{  g\inparen{\Theta^\top\!\inangle{\vec{x}} + \sqrt{\rho - q}\bxi  }  }  }^{2p}  } \nonumber\\
        &\quad\quad\quad\quad\quad\quad\quad\quad\quad\quad\quad\quad\quad \leq \frac{K_1 k L M^{2p - 1} }{N - 1} \inparen{ d_1(c_1) +  d_2(c_2) \boldsymbol{1}_{q > 0} + N\inparen{c_2}^{1/4}\boldsymbol{1}_{q = 0}  }.
        \label{eq:projResult_partialAsymptotic}
    \end{align}
    \item Let $\vec{z} \in \RR^k$ be a standard Gaussian vector, independent of all other sources of randomness, then 
    \begin{align}
        &\sup_{\substack{\norm{g}_{\textnormal{Lip}} \leq L \\ \norm{g}_{\infty} \leq M} } \Eb_{\ud}\Eb_{\Theta}\Eb_{\vec{z}}  \insquare{ \inparen{ \inangle{g(\Theta^\top \vec{x})  } - \Eb_{\bxi} \insquare{  g\inparen{\sqrt{q}\vec{z} + \sqrt{\rho - q}\bxi  }  }  }^{2p}  }  \leq \frac{K_2 k L M^{2p - 1} }{N - 1} \inparen{ d_1(c_1) + d_2(c_2)  }.
        \label{eq:projResult_asymptotic}
    \end{align}
    Here, for constants $0 \leq q < \rho$, $d_i : \RR \rightarrow \RR$, $i =1,2$ are defined by
    \begin{align}
        d_1(y) := \sqrt{3N^2 y + 4N\rho \sqrt{y} + 2N\rho^2 }, \quad 
        d_2(y) := \sqrt{3N^2 y + 4Nq\sqrt{y} + 2Nq^2 }.
        \label{eq:d1_d2_rates}
    \end{align}
\end{enumerate}
\label{thm:ProjResult_DisorderedCase_supOverBL_LM_2p}
\end{theorem}

The following lemmas are used many times throughout this work.

\begin{lemma}
Let $A$ and $B$ be random quantities. If $\Eb \insquare{ A^{2p}} \leq K_1$, and $\Eb \insquare{ B^{2p}} \leq K_2 $, then
    \begin{align*}
        \Eb \insquare{ (A + B)^{2p} } \leq 2^{2p} (K_1 + K_2).
    \end{align*}
\label{lemma:triangleIneq_E[()^2]}
\end{lemma}
\begin{proof}
Follows from $(A+B)^{2p} \leq 2^{2p - 1} (A^{2p} + B^{2p})$.
\end{proof}

\begin{lemma}[\cite{talagrand2010mean} Lemma 1.7.14] If $\abs{D'} \leq E'$ and $E \geq 1$, we have
\begin{align*}
    \abs{ \frac{D'}{E'} - \frac{D}{E} } \leq \abs{D' - D} + \abs{E' - E}.
\end{align*}
\label{lemma:TalVol1_Lemma1.7.14}
\end{lemma}

\begin{proofof}{Theorem \ref{thm:LI_abstract_theorem}}
We start with \eqref{eq:LI_abstract_GGj_statement}. 
We claim that the (random wrt.~disorder) variational problem
\begin{align*}
    \sup_{B \subseteq \RR^{k}} \inparen{ G_{N}^{(k)}\!\insquare{B} - \inparen{ \bigotimes_{j \leq k}  \GG_j} \insquare{B}    }^{2p}
\end{align*}
has maximum attained at the (random wrt.~disorder) measurable set $B^* = \inbraces{ \ud G_{N}^{(k)}/\ud \mu^{\otimes k} > \ud \bigotimes \GG_j / \ud \mu^{\otimes k}  }$. To see this, note that for every measurable $B$,
\begin{align*}
    \inparen{ G_{N}^{(k)}\!\insquare{B} - \inparen{ \bigotimes_{j \leq k}  \GG_j} \insquare{B}    }^{2p} \leq \frac{1}{2^{2p}} \norm{ \frac{ \ud G_{N}^{(k)} }{\ud \mu^{\otimes k}} - \frac{\ud  \bigotimes \GG_j }{\ud \mu^{\otimes k}} }_{1}^{2p}.
\end{align*}
It can then be checked that $B^*$ attains the upper bound on the RHS. Therefore, to prove \eqref{eq:LI_abstract_GGj_statement}, it suffices to prove 
\begin{align}
    \Eb_{\ud} \insquare{ \inparen{ G_{N}^{(k)}\!\insquare{B^*}  -  \inparen{ \bigotimes_{j \leq k}  \GG_j} \insquare{B^*}   }^{2p}  } \leq K\inparen{ R_0(N,p) +  \frac{C}{N^{1/2 - \epsilon}\rchi_{\Upsilon > 0} + N^{1/4 - \epsilon/2} \rchi_{\Upsilon = 0}}  },
    \label{eq:LI_abstract_GGj_statement_for_B^*}
\end{align}
where $K$ is a universal constant. From Lemma \ref{lemma:triangleIneq_E[()^2]} and the $R_0(N,p)$-decomposability of $-H_N$, to prove \eqref{eq:LI_abstract_GGj_statement_for_B^*} it suffices to show
\begin{align}
    \Eb \insquare{ \inparen{ G_{N,0}^{(k)}\!\insquare{B^*}  -  \inparen{ \bigotimes_{j \leq k}  \GG_j} \insquare{B^*}   }^{2p}  } \leq \frac{C}{N^{1/2 - \epsilon}\rchi_{\Upsilon > 0} + N^{1/4 - \epsilon/2} \rchi_{\Upsilon = 0}}.
    \label{eq:LI_abstract_goal}
\end{align}
The LHS of \eqref{eq:LI_abstract_goal} can be written as
\begin{align*}
    \Eb \insquare{ \inparen{ G_{N,0}^{(k)}\!\insquare{B^*}  -  \inparen{ \bigotimes_{j \leq k}  \GG_j} \insquare{B^*}   }^{2p}  } &= \Eb \insquare{ \inparen{  \frac{D'}{E'} - \frac{D}{E} }^{2p}  },
\end{align*}
where
\begin{align*}
    D' &:= \int_{B^*} \inangle{ \exp \sum_{j \leq k} x_j \varrho \vec{A}_{j} \boldsymbol{\cdot} \vec{w}  }^- \vec{\mu} \inparen{\ud x_1,\dots,\ud x_k} \\
    E' &:= \int_{\RR^k} \inangle{ \exp \sum_{j \leq k} x_j \varrho \vec{A}_{j} \boldsymbol{\cdot} \vec{w}  }^-\vec{\mu} \inparen{\ud x_1,\dots,\ud x_k} \\
    D &:= \int_{B^*} \Eb_\xi \insquare{ \exp \sum_{j \leq k} x_j \varrho \inparen{ \vec{A}_{j} \boldsymbol{\cdot} \inangle{\vec{w} }^- + \sqrt{\Xi - \Upsilon}\xi_j  } } \vec{\mu} \inparen{\ud x_1,\dots,\ud x_k} \\
    E &:= \int_{\RR^k} \Eb_\xi \insquare{ \exp \sum_{j \leq k} x_j \varrho \inparen{ \vec{A}_{j} \boldsymbol{\cdot} \inangle{\vec{w} }^- + \sqrt{\Xi - \Upsilon}\xi_j  } } \vec{\mu} \inparen{\ud x_1,\dots,\ud x_k},
\end{align*}
where $\inparen{\xi_j}_{j \leq k}$ are independent standard Gaussians, independent of all other sources of randomness, and $\vec{\mu}$ is the product probability measure on $\RR^k$ with $\vec{\mu}(\ud x_1,\dots, \ud x_k) \propto \bigotimes_{j \leq k} \exp f_j(x_j) \mu(\ud x_j)$. By hypothesis \eqref{eq:LI_abstract_E>=1_hypothesis}, it holds that $E \geq 1$, so that by Lemma \ref{lemma:TalVol1_Lemma1.7.14}, we have
\begin{align}
    \Eb \insquare{ \inparen{ G_{N,0}^{(k)}\!\insquare{B^*}  -  \inparen{ \bigotimes_{j \leq k}  \GG_j} \insquare{B^*}   }^{2p}  } &\leq K(p) \inparen{ \Eb \insquare{ \inparen{D' - D}^{2p}  } + \Eb \insquare{ \inparen{E' - E}^{2p}  } }.
    \label{eq:LI_abstract_D'-D,E'-E}
\end{align}
Denote
\begin{align*}
    D'_0 = D'_0\inparen{\inparen{x_j}} &:= \inangle{ \exp \sum_{j \leq k} x_j \varrho \vec{A}_{j} \boldsymbol{\cdot} \vec{w}  }^-; \quad D_0 = D_0\inparen{\inparen{x_j}} := \Eb_\xi \insquare{ \exp \sum_{j \leq k} x_j \varrho \inparen{ \vec{A}_{j} \boldsymbol{\cdot} \inangle{\vec{w} }^- + \sqrt{\Xi - \Upsilon}\xi_j  } }.
\end{align*}
Let $\Eb_k$ denote expectation wrt.~the randomness in $\vec{\mu}$, and let $\Eb_{-,\Theta}$ denote expectation wrt.~the disorder in $\inangle{\cdot}^-$ and the $\inparen{ \vec{A}_j }_{j \leq k}$ (the $\Theta$ subscript hints that these are now playing the role of projection directions). Note that $\Eb \equiv \Eb_k \Eb_{-,\Theta}$ captures all sources of disorder. By Jensen's inequality (for the convex function $y \mapsto y^{2p}$) and Fubini-Tonelli's theorem, we have
\begin{align}
    \Eb \insquare{ \inparen{D' - D}^{2p}  } &\leq \Eb \insquare{ \int_{B^*} \inparen{D'_0 - D_0}^{2p} \vec{\mu}\inparen{\ud x_1,\dots,\ud x_k}  } \nonumber\\
    &\leq \Eb \insquare{ \int_{\RR^k} \inparen{D'_0 - D_0}^{2p} \vec{\mu}\inparen{\ud x_1,\dots,\ud x_k}  } \nonumber\\
    &= \Eb_k \insquare{ \int_{\RR^{k}} \Eb_{-,\Theta} \insquare{\inparen{ D'_0 - D_0  }^{2p}}  \vec{\mu}\inparen{\ud x_1,\dots,\ud x_k} }.
    \label{eq:LI_abstract_D'-D_inTermsOf_D'0-D0}
\end{align}
Note that in the second inequality above, we have upper bounded the integral over the random (wrt.~disorder) set $B^*$ by an integral over $\RR^k$.

Let $A > 0$ be a parameter to be optimized later, and define $\tilde{\varphi}(x) \in C_b(\RR)$ by
\begin{align*}
\tilde{\varphi}(x) = 
    \begin{cases}
    1, & x \in [-A, A], \\
    0, & x \in [-2A, 2A]^c,\\
    \frac{1}{A}x + 2 & x \in [-2A, -A] \\
    -\frac{1}{A}x + 2 & x \in [A, 2A]. 
    \end{cases}
\end{align*}
Note that $\rchi_{[-A,A]} \leq \tilde{\varphi} \leq \rchi_{[-2A, 2A]}$. Define $\varphi \in C_b(\RR^k)$ by $\varphi\inparen{\inparen{Y_j}} := \prod_{j \leq k} \tilde{\varphi}\inparen{Y_j}$. Define the function $f : \RR^{k} \rightarrow \RR$ by
\begin{align*}
    f\inparen{\inparen{ Y_j  }} := \varphi\inparen{\inparen{Y_j}} \prod_{j \leq k} \exp\inparen{x_j \varrho Y_j}.
\end{align*}
It can be computed that there exists numbers $M$ and $L$ such that
\begin{align}
    \norm{f}_{\infty} &\leq M := \exp\inparen{2\varrho A \sum_{j} \abs{x_j}} \nonumber\\
    \norm{f}_{\textnormal{Lip}} &\leq L := 2 \exp\inparen{ 2\varrho A \sum_{j} \abs{x_j}  } \sqrt{\max \inparen{ \varrho^2\sum_{j}\abs{x_j}^2 , \; \frac{k}{A^2}  } },
    \label{eq:LI_abstract_M_L}
\end{align}
where the last follows from $\abs{ \partial f/\partial Y_{j}  } \leq \inparen{ \varrho \abs{x_j} + 1/A  } \exp \inparen{ 2\varrho A  \sum_{j} \abs{x_j} }$, which holds everywhere except for a finite number of points where $f$ is not differentiable. This yields a bound on $\norm{\nabla f}_2$ which holds almost everywhere, from which \eqref{eq:LI_abstract_M_L} follows.

Define
\begin{align}
    X_{N,j} &:= \vec{A}_j \boldsymbol{\cdot} \vec{w}; \quad\quad X_{0,j} := \vec{A}_j \boldsymbol{\cdot} \inangle{\vec{w}}^- + \sqrt{\Xi - \Upsilon} \xi_j.
\end{align}
By Lemma \ref{lemma:triangleIneq_E[()^2]}, we have
\begin{align}
    \Eb_{-,\Theta} \insquare{\inparen{ D'_0 - D_0  }^{2p}} &\leq K(p)\inparen{ \Eb_{-,\Theta}\insquare{F^{2p}} + \Eb_{-,\Theta} \insquare{G^{2p}}  },
    \label{eq:LI_abstract_D'0-D0_inTermsOf_F2p,G2p}
\end{align}
where 
\begin{align}
    F &:= \inangle{f\inparen{\inparen{ X_{N,j}  }}}^- - \Eb_\xi \insquare{  f\inparen{\inparen{ X_{0,j}  }}  } \nonumber\\
    G &:= \inangle{ \inparen{1 - \varphi\inparen{\inparen{X_{N,j}}}}\prod_{j} \exp\inparen{x_j \varrho X_{N,j}}   }^- - \Eb_\xi\insquare{ \inparen{1 - \varphi\inparen{\inparen{X_{0,j}}}}\prod_{j} \exp\inparen{x_j \varrho X_{0,j}}   }.
    \label{eq:LI_abstract_G}
\end{align}
Since $f$ is bounded and Lipschitz, a bound for $\Eb\insquare{F^{2p}}$ will follow from Theorem \ref{thm:ProjResult_DisorderedCase_supOverBL_LM_2p} and the thin-shell and overlap concentration hypothesis \eqref{eq:LI_abstract_TSOC_hypothesis}. From hypothesis \eqref{eq:LI_abstract_M=N_hypothesis}, we have
\begin{align}
    \Eb_{-,\Theta}\insquare{F^{2p}} \leq  \frac{ K\exp\inparen{ 4 p \varrho A \sum_j \abs{x_j}} \sqrt{\max \inparen{ \varrho^2\sum_{j}\abs{x_j}^2 , \; \frac{k}{A^2}  } }  }{\sqrt{N}\rchi_{\Upsilon > 0} + N^{1/4} \rchi_{\Upsilon = 0}}.
    \label{eq:LI_abstract_F2p_bound}
\end{align}
On the other hand, letting $G_1$ and $G_2$ be the first and second terms on the RHS of \eqref{eq:LI_abstract_G}, we have
\begin{align}
    \Eb_{-,\Theta}\insquare{G^{2p}} &= \Eb_{-,\Theta} \insquare{\inparen{ G_1 - G_2}^{2p}} = \sum_{0 \leq r \leq 2p} \binom{2p}{r} \Eb_{-,\Theta}\insquare{ G_1^{r}  G_2^{2p - r}   }.
    \label{eq:LI_abstract_G2p_expansion}
\end{align}
We now bound $\Eb\insquare{G_1^{2p}}$ and $\Eb\insquare{G_2^{2p}}$. Note that $1 - \varphi\inparen{\inparen{Y_j}} \leq \rchi_{\bigcup_{j} \abs{Y_j} > A}$. We have, for some parameter $\eta > 0$ to be optimized later,
\begin{align}
    \Eb_{-,\Theta}\insquare{G_{1}^{2p}} &\leq \Eb_{-,\Theta} \inangle{ \rchi_{\bigcup_j \abs{X_{N,j}} > A} \prod_{j} \exp\inparen{2p \varrho x_j X_{N,j} }  }^- \nonumber\\
    &= \Eb_{-,\Theta} \inangle{ \exp\inparen{ -2p\varrho\eta \sum_{j} \abs{ X_{N,j} }} \rchi_{\bigcup_j \abs{X_{N,j}} > A} \prod_{j} \exp\inparen{2p \varrho \inparen{ x_j X_{N,j} + \eta \abs{X_{N,j}} } }  }^- \nonumber\\
    &\leq \exp\inparen{-2p\varrho \eta A} \Eb_{-,\Theta}\inangle{  \prod_{j} \inparen{\exp\insquare{ 2p\varrho(x_j + \eta)X_{N,j}  } + \exp\insquare{2p\varrho(x_j - \eta)X_{N,j}} } }^- \nonumber\\
    &= \exp\inparen{-2p\varrho \eta A} \Eb_{-}\inangle{  \prod_{j} \inparen{\Eb_{\Theta}\exp\insquare{ 2p\varrho(x_j + \eta)X_{N,j}  } + \Eb_{\Theta} \exp\insquare{2p\varrho(x_j - \eta)X_{N,j}} } }^-,
    \label{eq:LI_abstract_G1^2p_firstExpansion}
\end{align}
where the first inequality follows from Jensen's inequality, and in the second inequality we have used $e^{t\abs{y}} \leq e^{ty} + e^{-ty}$, and the last equality follows from independence of the $\vec{A}_j$'s. For each $j$, we have
\begin{align*}
    \Eb_{\Theta}\exp\insquare{ 2p\varrho(x_j \pm \eta)X_{N,j}  } = \exp \inparen{ \frac{2p^2\varrho^2 (x_j \pm \eta)^2}{M}  \sum_{m \leq M} w_m^2 } \leq \exp \inparen{4p^2 \varrho^2 D^2 (x_* + \eta^2)},
\end{align*}
where we used Jensen's inequality and the hypothesis on the bound on $\norm{w_m^2}$, and we set $x_* := \max_j x_j^2$. From \eqref{eq:LI_abstract_G1^2p_firstExpansion} we have
\begin{align}
    \Eb_{-,\Theta}\insquare{G_{1}^{2p}} \leq 2^k \exp\inparen{-2p\varrho \eta A} \exp \inparen{4kp^2 \varrho^2 D^2 (x_* + \eta^2)}.
    \label{eq:LI_abstract_G1^2p_finalExpansion}
\end{align}
Repeat the steps for \eqref{eq:LI_abstract_G1^2p_firstExpansion} for $\Eb_{-,\Theta}\insquare{G_{2}^{2p}}$, and using that 
\begin{align*}
    \Eb_\xi \Eb_{\Theta}\exp\insquare{ 2p\varrho(x_j \pm \eta)X_{0,j}  } &= \Eb_\Theta\insquare{ \exp \inparen{ 2p\varrho(x_j \pm \eta)\vec{A}_j \boldsymbol{\cdot} \inangle{\vec{w}}^{-} } } \Eb_{\xi}\insquare{ \exp \inparen{ 2p\varrho(x_j \pm \eta)\sqrt{\Xi - \Upsilon}\xi_j } }\\
    &\leq \exp \inparen{4p^2 \varrho^2 (D^2 + \Xi - \Upsilon) (x_* + \eta^2)},
\end{align*}
we have
\begin{align}
    \Eb_{-,\Theta}\insquare{G_{2}^{2p}} \leq 2^k \exp\inparen{-2p\varrho \eta A} \exp \inparen{4kp^2 \varrho^2 (D^2 + \Xi - \Upsilon) (x_* + \eta^2)}.
    \label{eq:LI_abstract_G2^2p_finalExpansion}
\end{align}
For every $0 < r < 2p$, since $G_1, G_2 \geq 0$, H{\"o}lder's inequality gives
\begin{align*}
    \Eb_{-,\Theta}\insquare{G_1^{r} G_2^{2p - r}} \leq \inparen{ \Eb_{-,\Theta}\insquare{ G_1^{2p}  } }^{r/2p} \inparen{ \Eb_{-,\Theta} \insquare{G_{2}^{2p}} }^{1 - r/2p}.
\end{align*}
Therefore, from \eqref{eq:LI_abstract_G2p_expansion}, we have using \eqref{eq:LI_abstract_G1^2p_finalExpansion} and  \eqref{eq:LI_abstract_G2^2p_finalExpansion} that
\begin{align}
    \Eb_{-,\Theta}\insquare{G^{2p}} &\leq K(p,k) \exp\inparen{ -2p\varrho \eta A  }\exp \inparen{4k p^2 \varrho^2 (D^2 + \Xi - \Upsilon) (x_* + \eta^2)}.
    \label{eq:LI_abstract_G2p_bound}
\end{align}
Let $T := \sum_{j} \abs{x_j}$ and set for $0 < \epsilon < 1/2$,
\begin{align}
    A = \frac{1}{4 p \varrho T}\log N^\epsilon; \quad \eta = 2T\insquare{ \inparen{ \frac{1}{2\epsilon} - 1}\rchi_{\Upsilon > 0} + \inparen{ \frac{1}{4\epsilon} - \frac{1}{2}}\rchi_{\Upsilon = 0}  }.
    \label{eq:LI_abstract_A_eta}
\end{align}
Note that whenever $N \geq \exp\inparen{ \sqrt{16k}/\epsilon  }$, we have
\begin{align}
    \sqrt{\max\inparen{\varrho^2 \sum_{j} x_j^2, \, \frac{k}{A^2}  }} \leq \sqrt{\max \inparen{ \varrho^2 T^2, \, \frac{16k\varrho^2 T^2}{\inparen{\log N^\epsilon}^2}  } } \leq \varrho T.
    \label{eq:LI_abstract_maxTerm_bound}
\end{align}
From \eqref{eq:LI_abstract_D'0-D0_inTermsOf_F2p,G2p}, using \eqref{eq:LI_abstract_F2p_bound}, \eqref{eq:LI_abstract_G2p_bound}, \eqref{eq:LI_abstract_A_eta}, and \eqref{eq:LI_abstract_maxTerm_bound}, we get
\begin{align}
    \Eb_{-,\Theta} \insquare{ \inparen{D'_0 - D_0}^{2p}  } \leq \frac{K T \exp\inparen{  4kp^2\varrho^2(D^2 + \Xi - \Upsilon)(1+ 4\bar{\epsilon}^2) T^2  }}{N^{1/2 - \epsilon}\rchi_{\Upsilon > 0} + N^{1/4 - \epsilon/2} \rchi_{\Upsilon = 0}}.
\end{align}
Therefore, from \eqref{eq:LI_abstract_D'-D_inTermsOf_D'0-D0}, we have
\begin{align}
    \Eb\insquare{\inparen{D' - D}^{2p}} \leq \frac{K C}{N^{1/2 - \epsilon}\rchi_{\Upsilon > 0} + N^{1/4 - \epsilon/2} \rchi_{\Upsilon = 0}}.
\end{align}
The same bound holds for $\Eb\insquare{\inparen{E' - E}^{2p}}$. The proof of \eqref{eq:LI_abstract_GGj_statement_for_B^*} is complete by \eqref{eq:LI_abstract_D'-D,E'-E}.

The proof of \eqref{eq:LI_abstract_HHj_statement} is analogous. In the above, we replace all occurrences of $\Eb_{-,\Theta}$ with $\Eb_{-,\Theta,\vec{z}}$, and replace all occurrences of $\vec{A}_j \boldsymbol{\cdot} \inangle{ \vec{w} }^{-}$ to $\sqrt{\Upsilon} z_j$, and in \eqref{eq:LI_abstract_F2p_bound} we use the second part of Theorem \ref{thm:ProjResult_DisorderedCase_supOverBL_LM_2p}. 
\end{proofof}

\section{Supplementary proofs for Section \ref{sec:perceptron}}
\label{sec:suppProofs_for_perceptron}

\begin{proofof}{Proposition \ref{proposition:perceptron_t_interpolation}}
The proof is a mechanical application of Gaussian integration by parts. Note that
\begin{align*}
    \frac{\ud }{\ud t} \inangle{f}_t &= \sum_{\ell \leq n} \inangle{ \frac{\ud}{\ud t} \inparen{ -H_{N,M,t}^{\ell} } f  }_t - n\inangle{ \frac{\ud}{\ud t} \inparen{ -H_{N,M,t}^{n+1} } f  }_t,
\end{align*}
where 
\begin{align*}
    \frac{\ud}{\ud t} \inparen{-H_{N,M,t}^{\ell}} = \frac{1}{2\sqrt{tN}}\sum_{m \leq M} \inparen{\sum_{j \leq k}  x_j g_{j,m} } u'\inparen{ S_{m,t}^{\ell}  } - \frac{1}{2\sqrt{(1-t) N}} \sum_{m \leq M} \inparen{ \sum_{j \leq k} x_j \tilde{g}_{j,m}    } u'\inparen{ S_{m}^{0,\ell} }.
\end{align*}
It follows that, since $M/\sqrt{N} = \alpha \sqrt{N}$, and since by symmetry, each $u'(S_{m,t}^{\ell})$, $m \leq M$, and each $x_j^{\ell}$, $j \leq k$ bring the same contribution,
\begin{align*}
    \frac{\ud}{\ud t} \nu_t \insquare{f} &= \textnormal{IV} + \textnormal{V},
\end{align*}
where
\begin{align*}
    \textnormal{IV} &:= \frac{\alpha k }{2}\sqrt{\frac{N}{t}} \inparen{ \sum_{\ell \leq n} \nu_t \insquare{ x_1^\ell g_{1,M} u'\inparen{ S_{M,t}^{\ell}   } f  } - n \nu_t \insquare{ x_1^{n+1} g_{1,M} u'\inparen{ S_{M,t}^{n+1}   } f  } } \\
    \textnormal{V} &:= -\frac{\alpha k}{2} \sqrt{\frac{N}{1 - t}} \inparen{  \sum_{\ell \leq n} \nu_t \insquare{ x_1^\ell \tilde{g}_{1,M} u'\inparen{ S_{M}^{0,\ell}   } f  } - n \nu_t \insquare{ x_1^{n+1} \tilde{g}_{1,M} u'\inparen{ S_{M}^{0,n+1}   } f  }   }.
\end{align*}
Gaussian integration by parts will be applied on each of the terms in $\textnormal{IV}$ and $\textnormal{V}$ in the following way. Starting with term $\textnormal{IV}$, for any $\ell \leq n$, we write
\begin{align*}
    \nu_t \insquare{ x_1^\ell g_{1,M} u'\inparen{ S_{M,t}^{\ell}  } f  } = \Eb\insquare{ g_{1,M} \inangle{x_1^\ell u'\inparen{ S_{M,t}^\ell} f }_t  },
\end{align*}
and we note that $\inangle{x_1^\ell u'\inparen{ S_{M,t}^\ell} f }_t$ is a function of the Gaussian r.v.'s $(g_{i,m})_{i \leq N, m \leq M}$, $(\tilde{g}_{j,m})_{j \leq k, m \leq M}$. However, by the independence and zero-mean of all the $g_{i,m}$'s and $\tilde{g}_{j,m}$'s, we may as well regard $\inangle{x_1^\ell u'\inparen{ S_{M,t}^\ell} f }_t$ as a function of $g_{1,M}$ only, and use that $\Eb\insquare{g_{1,M} F(g_{1,M}) } = \Eb \insquare{ F'(g_{1,M}) }$. With the relations
\begin{align*}
    \frac{\partial}{\partial g_{1,M}} S_{M,t}^{\ell} = \sqrt{\frac{t}{N}}x_1^\ell; \quad\quad \frac{\partial}{\partial g_{1,M}}\inparen{- H_{N,M,t}^{\ell}} = \sqrt{\frac{t}{N}} x_1^\ell u'\inparen{ S_{M,t}^{\ell}   },
\end{align*}
we obtain
\begin{align}
    \nu_t \insquare{ x_1^\ell g_{1,M} u'\inparen{ S_{M,t}^{\ell}   } f  } &= \sqrt{\frac{t}{N}} \left( \nu_t\insquare{ u''\inparen{S_{M,t}^{\ell} } f  } + \sum_{\ell' \leq n} \nu_t\insquare{ x_1^\ell x_1^{\ell'} u'\inparen{S_{M,t}^{\ell}} u'\inparen{S_{M,t}^{\ell'}} f  } \right.  \nonumber\\
    &\quad\quad\quad\quad\quad \left. - n \nu_t\insquare{ x_1^\ell x_1^{n+1} u'\inparen{S_{M,t}^{\ell}} u'\inparen{S_{M,t}^{n+1}} f  } \!\phantom{\Bigg\rvert}\! \right). 
    \label{eq:perceptron_t_interpolation_I-1}
\end{align}
Similarly, for the $\ell = n+1$ case, we have
\begin{align}
    \nu_t \insquare{ x_1^{n+1} g_{1,M} u'\inparen{ S_{M,t}^{n+1}  } f  } &= \sqrt{\frac{t}{N}} \left( \nu_t\insquare{ u''\inparen{S_{M,t}^{n+1} } f  } + \sum_{\ell' \leq n+1} \nu_t\insquare{ x_1^{n+1} x_1^{\ell'} u'\inparen{S_{M,t}^{n+1}} u'\inparen{S_{M,t}^{\ell'}} f  } \right.  \nonumber\\
    &\quad\quad\quad\quad\quad \left. - (n+1) \nu_t\insquare{ x_1^{n+1} x_1^{n+2} u'\inparen{S_{M,t}^{n+1}} u'\inparen{S_{M,t}^{n+2}} f  } \!\phantom{\Bigg\rvert}\! \right).
    \label{eq:perceptron_t_interpolation_I-2}
\end{align}
We next evaluate the terms in \textnormal{V}. For $\ell \leq n$, writing $\nu_t\insquare{ x_1^{\ell} \tilde{g}_{1,M} u'\inparen{  S_{M}^{0,\ell}  } } = \Eb \insquare{ \tilde{g}_{1,M} \inangle{x_1^{\ell} u'\inparen{  S_{M}^{0,\ell}  }   }_t }$, and using the relations
\begin{align*}
    \frac{\partial}{\partial \tilde{g}_{1,M}} S_{M}^{0,\ell} = 0; \quad\quad \frac{\partial}{\partial \tilde{g}_{1,M}}\inparen{- H_{N,M,t}^{\ell}} = \sqrt{\frac{1-t}{N}} x_1^\ell u'\inparen{ S_{M}^{0,\ell}   },
\end{align*}
we obtain
\begin{align}
    \nu_t \insquare{ x_1^\ell \tilde{g}_{1,M} u'\inparen{ S_{M}^{0,\ell}   } f  } &= \sqrt{\frac{1-t}{N}} \left(  \sum_{\ell' \leq n} \nu_t\insquare{ x_1^\ell x_1^{\ell'} u'\inparen{S_{M}^{0,\ell}} u'\inparen{S_{M}^{0,\ell'}} f  } \right.  \nonumber\\
    &\quad\quad\quad\quad\quad\quad\quad\quad \left. - n \nu_t\insquare{ x_1^\ell x_1^{n+1} u'\inparen{S_{M}^{0,\ell}} u'\inparen{S_{M}^{0,n+1}} f  } \!\phantom{\Bigg\rvert}\! \right). 
    \label{eq:perceptron_t_interpolation_II-1}
\end{align}
Similarly, for the $\ell = n+1$ case,
\begin{align}
    \nu_t \insquare{ x_1^{n+1} \tilde{g}_{1,M} u'\inparen{ S_{M}^{0,n+1} } f } &= \sqrt{\frac{1-t}{N}} \left(  \sum_{\ell' \leq n} \nu_t\insquare{ x_1^{n+1} x_1^{\ell'} u'\inparen{S_{M}^{0,n+1}} u'\inparen{S_{M}^{0,\ell'}} f  } \right.  \nonumber\\
    &\quad\quad\quad\quad\quad\quad\quad\quad \left. - n \nu_t\insquare{ x_1^{n+1} x_1^{n+2} u'\inparen{S_{M}^{0,n+1}} u'\inparen{S_{M}^{0,n+2}} f  } \!\phantom{\Bigg\rvert}\! \right). 
    \label{eq:perceptron_t_interpolation_II-2}
\end{align}
Using \eqref{eq:perceptron_t_interpolation_I-1} and \eqref{eq:perceptron_t_interpolation_I-2} for term $\textnormal{IV}$, and using \eqref{eq:perceptron_t_interpolation_II-1} and \eqref{eq:perceptron_t_interpolation_II-2} for term $\textnormal{V}$, we see that upon regrouping, $\textnormal{IV} + \textnormal{V} = \textnormal{I} + \textnormal{II} + \textnormal{III}$.
\end{proofof}

\begin{proofof}{Lemma \ref{lemma:perceptron_gammaInterpolation}}
By symmetry, the contribution brought by each $x_j^{\ell} \tilde{g}_{j,M}$ in $T_M^\ell$ is the same. Therefore
\begin{align}
    \frac{\partial}{\partial \gamma} \nu_{t,\gamma}\insquare{f}   &= \frac{k}{2\sqrt{\gamma}} \sqrt{\frac{1-t}{N}} \inparen{ \sum_{\ell \leq n} \nu_{t,\gamma}\insquare{ \tilde{g}_{1,M}x_1^{\ell} u'\inparen{ S_{M}^{0,\ell}  } f } -n\nu_{t,\gamma} \insquare{ \tilde{g}_{1,M} x_1^{n+1} u'\inparen{S_{M}^{0,n+1}}  }    }.
    \label{eq:perceptron_gammaInterpolation-0}
\end{align}
For each $\ell \leq n$, we have
\begin{align}
    &\nu_{t,\gamma}\insquare{ \tilde{g}_{1,M}x_1^{\ell} u'\inparen{ S_{M}^{0,\ell}  } f } \nonumber\\
    &\quad\quad= \Eb\insquare{  \tilde{g}_{1,M} \frac{ \inangle{ x_1^{\ell} u'\inparen{ S_{M}^{0,\ell}  }  f \exp\inparen{ \sum_{\ell \leq n} u\inparen{S_{M,t}^{\ell}} + \sqrt{\gamma}\sqrt{\frac{1-t}{N}} T_{M}^{\ell} u'\inparen{ S_{M}^{0,\ell}  } } }_{t,\sim}  }{ \inangle{ \exp\inparen{  u\inparen{S_{M,t}^{1}} + \sqrt{\gamma}\sqrt{\frac{1-t}{N}} T_{M}^{1} u'\inparen{ S_{M}^{0,1}  } } }_{t,\sim}^{n}  }  } \nonumber\\
    &\quad\quad= \sqrt{\gamma}\sqrt{\frac{1-t}{N}} \inparen{  \sum_{\ell' \leq n} \nu_{t,\gamma} \insquare{ x_1^{\ell} x_1^{\ell'} u'\inparen{ S_{M}^{0,\ell}   } u'\inparen{ S_{M}^{0,\ell'}  } f  } - n \nu_{t,\gamma} \insquare{ x_1^{\ell} x_1^{n+1} u'\inparen{ S_{M}^{0,\ell}   } u'\inparen{ S_{M}^{0,n+1}  } f  }  },  
    \label{eq:perceptron_gammaInterpolation-1}
\end{align}
which follows from Gaussian integration by parts. Similarly, for the $\ell = n+1$ case,
\begin{align}
    &\nu_{t,\gamma}\insquare{ \tilde{g}_{1,M}x_1^{n+1} u'\inparen{ S_{M}^{0,n+1}  } f } \nonumber\\
    &= \sqrt{\gamma}\sqrt{\frac{1-t}{N}} \inparen{  \sum_{\ell' \leq n+1} \nu_{t,\gamma} \insquare{ x_1^{n+1} x_1^{\ell'} u'\inparen{ S_{M}^{0,n+1}   } u'\inparen{ S_{M}^{0,\ell'}  } f  } - (n+1) \nu_{t,\gamma} \insquare{ x_1^{n+1} x_1^{n+2} u'\inparen{ S_{M}^{0,n+1}   } u'\inparen{ S_{M}^{0,n+2}  } f  }  }.
    \label{eq:perceptron_gammaInterpolation-2}
\end{align}
Using formulas \eqref{eq:perceptron_gammaInterpolation-1} and \eqref{eq:perceptron_gammaInterpolation-2} in \eqref{eq:perceptron_gammaInterpolation-0}, we obtain
\begin{align}
    \frac{\partial}{\partial \gamma} \nu_{t,\gamma}\insquare{f} = \frac{k(1-t)}{N} &\left( \sum_{\ell < \ell' \leq n}  \nu_{t,\gamma} \insquare{ x_1^{\ell} x_1^{\ell'} u'\inparen{ S_{M}^{0,\ell}   } u'\inparen{ S_{M}^{0,\ell'}  } f  }    \right. \nonumber\\
    &\quad -n \sum_{\ell \leq n} \nu_{t,\gamma} \insquare{ x_1^{\ell} x_1^{n+1} u'\inparen{ S_{M}^{0,\ell}   } u'\inparen{ S_{M}^{0,n+1}  } f  } \nonumber\\
    &\quad +\frac{1}{2}\sum_{\ell \leq n} \nu_{t,\gamma} \insquare{ u^{' 2}\inparen{ S_{M}^{0,\ell}   } f  } - \frac{n}{2} \nu_{t,\gamma} \insquare{ u^{'2}\inparen{ S_{M}^{0,n+1}   }  f  } \nonumber\\
    &\quad \left. +\frac{n(n+1)}{2} \nu_{t,\gamma} \insquare{ x_1^{n+1} x_1^{n+2} u'\inparen{ S_{M}^{0,n+1}   } u'\inparen{ S_{M}^{0,n+2}  } f  } \!\phantom{\Bigg\rvert}\!\right).
    \label{eq:perceptron_gammaInterpolation_finalEquality}
\end{align}
The $1/N$ in front saves us from further intricate computations in each term. We simply bound each term by H{\"o}lder's inequality as follows: $\abs{ \nu_{t,\gamma} \insquare{ x_1^{\ell} x_1^{\ell'} u'\inparen{ S_{M}^{0,\ell}   } u'\inparen{ S_{M}^{0,\ell'}  } f  } } \leq D^2 \nu_{t,\gamma}\abs{f}$. The result follows by writing
\begin{align*}
    \sup_{0 \leq t \leq 1} \abs{\nu_{t,1}\insquare{f} - \nu_{t,0}\insquare{f}} \leq \sup_{0 \leq t \leq 1} \sup_{0 \leq \gamma \leq 1} \abs{ \frac{\partial}{\partial \gamma} \nu_{t,\gamma} \insquare{f} }.
\end{align*}
\end{proofof}

\begin{proofof}{Lemma \ref{lemma:perceptron_vInterpolation}}
We start with the case $B_v = u'(S_v^1)u'(S_v^2)$. Define
\begin{align*}
    S_{v}^{\ell\, \prime} := \frac{\partial }{\partial v} S_{v}^{\ell} = \frac{1}{2\sqrt{v}} S_{M,t}^{\ell} - \frac{1}{2\sqrt{1-v}} \theta^\ell.
\end{align*}
Direct computation yields
\begin{align}
    \frac{\partial}{\partial v} \nu_{t,0,v}\insquare{  u'(S_v^1)u'(S_v^2) f    } &= \nu_{t,0,v} \insquare{ S_{v}^{1\, \prime} u''(S_v^1) u'(S_v^2) f  } \nonumber\\
    &\quad + \nu_{t,0,v} \insquare{ S_{v}^{2\, \prime} u'(S_v^1) u''(S_v^2) f  } \nonumber\\
    &\quad + \sum_{\ell \leq n} \nu_{t,0,v} \insquare{ S_{v}^{\ell\, \prime} u'(S_v^1) u'(S_v^2) u'(S_v^{\ell}) f  } \nonumber \\
    &\quad -n \nu_{t,0,v} \insquare{ S_{v}^{n+1\, \prime} u'(S_v^1) u'(S_v^2) u'(S_{v}^{n+1}) f  }.
    \label{eq:perceptron_vInterpolation_u'(S_v^1)u'(S_v^2)f}
\end{align}
We next apply Gaussian integration by parts on each term in \eqref{eq:perceptron_vInterpolation_u'(S_v^1)u'(S_v^2)f}. We start with the first term in \eqref{eq:perceptron_vInterpolation_u'(S_v^1)u'(S_v^2)f}. The approach here largely follows that of \cite{talagrand2010mean} Lemma 2.3.2. Define
\begin{align*}
    w = w(\vec{x}^1,\dots,\vec{x}^n) &:= \exp\inparen{ -\sum_{\ell \leq n} H_{N,M-1,t}^{\ell} } \\
    w_{*}^{\ell} = w_*(\vec{x}^\ell) &:= \exp\inparen{ -H_{N,M-1,t}^{\ell}},
\end{align*}
so that in fact $w = \prod_{\ell \leq n} w_{*}^{\ell}$. Note that $w$ and $w_{*}^{\ell}$ are independent of the ``randomness in $M$'', that is, the randomness in $g_{i,M}$'s, $\tilde{g}_{j,M}$'s, $\xi^\ell$'s, $z$. The first term in \eqref{eq:perceptron_vInterpolation_u'(S_v^1)u'(S_v^2)f} can be written as
\begin{align}
    &\nu_{t,0,v} \insquare{ S_{v}^{1\, \prime} u''(S_v^1) u'(S_v^2) f  } \nonumber \\
    &\quad=  \Eb \frac{ \Eb_{\xi} \sum_{\vec{x}^1,\dots,\vec{x}^n}  S_{v}^{1\, \prime}(\vec{x}^1,\xi^1) f(\vec{x}^1,\dots,\vec{x}^n) u''(S_v^1(\vec{x}^1, \xi^1)) u'(S_v^2(\vec{x}^2, \xi^2)) \exp\inparen{\sum_{\ell \leq n} u(S_v^\ell)} w  }{ \inparen{ \Eb_{\xi} \sum_{\vec{x}^1} \exp\inparen{ u(S_v^1(\vec{x}^1, \xi^1))  w_{*}^1  }  }^n  } \nonumber\\
    &\quad= \Eb \sum_{\vec{x}^1,\dots,\vec{x}^n} w \Eb_{M} S_v^{1\, \prime} \frac{C}{Z_{t,0,v}^n}, 
    \label{eq:perceptron_vInterpolation_S_v^1'_u''(S_v^1)u'(S_v^2)f}
\end{align}
where $\Eb_M$ denotes an expectation over the `randomness in $M$'', that is, the randomness in $g_{i,M}$'s, $\xi^\ell$'s, $z$ (one notes that in $\nu_{t,0}$, the dependence on $\tilde{g}_{j,M}$'s is removed), and where
\begin{align*}
    C &= C_{\vec{x}^1,\dots,\vec{x}^n} := f(\vec{x}^1,\dots,\vec{x}^n) u''(S_v^1(\vec{x}^1, \xi^1)) u'(S_v^2(\vec{x}^2, \xi^2)) \exp\inparen{\sum_{\ell \leq n} u(S_v^\ell)}, \\
    Z_{t,0,v} &:= \Eb_{\xi} \sum_{\vec{x}^1} \exp\inparen{ u(S_v^1(\vec{x}^1,\xi^1) }  w_{*}^1  .
\end{align*}
Gaussian integration by parts will be applied on the expression involving $\Eb_M$ in \eqref{eq:perceptron_vInterpolation_S_v^1'_u''(S_v^1)u'(S_v^2)f}, where $\vec{x}^1,\dots,\vec{x}^n$ are held fixed. Note that, with $\vec{x}^1$ fixed, then $S_v^1$, as a linear function of $(g_{i,M})$, $z$, $\xi^1$, is a Gaussian r.v. The same holds true for $S_v^{1\, \prime}$. It is natural to define
\begin{align*}
    z_{\vec{x}}^{\ell} &:= \sqrt{v} S_{M,t}(\vec{x}) + \sqrt{1-v} \inparen{\sqrt{q}z + \sqrt{1 - q} \xi^\ell} \\
    z_{\vec{x}}^{\ell\, \prime} &:= \frac{1}{2\sqrt{v}} S_{M,t}(\vec{x}) - \frac{1}{2\sqrt{1-v}} \inparen{\sqrt{q} z + \sqrt{1-q}\xi^\ell},
\end{align*}
so that $z_{\vec{x}^{\ell}}^{\ell} = S_v^{\ell}$, and $z_{\vec{x}^{\ell}}^{\ell\, \prime} = S_{v}^{\ell\, \prime}$. The notation $z_{\vec{x}}^{\ell}$'s emphasizes that $\vec{x}$ is held fixed, and we are considering a family of Gaussian r.v.'s $(z_{\vec{x}}^{\ell})_{\vec{x} \in \Sigma_N, \ell \leq n}$. We may then regard $C$ as a function of these $z_{\vec{x}}^{\ell}$'s by writing
\begin{align*}
    F_{\vec{x}^1,\dots,\vec{x}^n} \inparen{\inparen{ z_{\vec{x}}^{\ell}  }} := C = C_{\vec{x}^1,\dots,\vec{x}^n},
\end{align*}
where $F_{\vec{x}^1,\dots,\vec{x}^n}$ is the function on $\RR^{\textnormal{card} (\Sigma_N) \times n}$ defined by
\begin{align*}
    F_{\vec{x}^1,\dots,\vec{x}^n} \inparen{\inparen{ g_{\vec{x}}^{\ell}  }} &= f(\vec{x}^1,\dots,\vec{x}^n) u''\!\inparen{ g_{\vec{x}^1}^{1}  } u'\!\inparen{ g_{\vec{x}}^{2} } \exp\inparen{ \sum_{\ell \leq n} u\!\inparen{ g_{\vec{x}^\ell}^{\ell}  }  }.
\end{align*}
On the other hand, the function $Z_{t,0,v}$ is independent of the randomness in $\xi$, since it has been integrated out. Therefore we cannot regard $Z_{t,0,v}$ as a function of the $(z_{\vec{x}}^{\ell})$. Instead, define
\begin{align*}
    y_{\vec{x}} &:= \sqrt{v} S_{M,t}(\vec{x}) + \sqrt{1-v} \sqrt{q} z,
\end{align*}
so that $y_{\vec{x}}$ is the part that is independent of $\xi^\ell$ in $z_{\vec{x}}^{\ell}$. Further define $\xi_{*}^{\ell} = \sqrt{1-v}\sqrt{1-q} \xi^\ell$, so that $z_{\vec{x}}^{\ell} = y_{\vec{x}} + \xi_{*}^{\ell}$. We may then write
\begin{align*}
    F_1\inparen{ \inparen{y_{\vec{x}}}   } := Z_{t,0,v},
\end{align*}
where $F_1$ is the function on $\RR^{\textnormal{card}(\Sigma_N)}$ defined by
\begin{align*}
    F_1\inparen{\inparen{g_{\vec{x}}}} &= \Eb_{\xi} \sum_{\vec{x}} w_*(\vec{x}) \exp u\inparen{ g_{\vec{x}} + \xi_*^1  }.
\end{align*}
The family of Gaussian r.v.'s $(z_{\vec{x}}^{\ell})$, $(z_{\vec{x}}^{\ell\, \prime})$, $(y_{\vec{x}})$ defines the Gaussian space on which we will use Gaussian integration by parts. Before proceeding, we make several straightforward computations involving their correlations. Define $p(N,k,t) = -k(1-t)/(2N)$. 

Note that
\begin{align*}
    \Eb \insquare{ (\theta^\ell)^2  } = 1; \quad\quad \ell \neq \ell' \Rightarrow \Eb\insquare{\theta^\ell \theta^{\ell'}} = q
\end{align*}
and 
\begin{align}
    \Eb\insquare{ S_{M,t}(\vec{x}) S_{M,t}(\vec{v})  } &= R^t(\vec{x},\vec{v}) := \frac{1}{N} \sum_{k \leq i \leq N} x_i v_i + \frac{t}{N}\inparen{ \sum_{j \leq k} x_j v_j  } \label{eq:perceptron_R^t_ell,ell'_definition}\\
    \Eb\insquare{ \inparen{ S_{M,t}(\vec{x}) }^2 } &= R^t(\vec{x},\vec{x}) = 1 - p(N,k,t). \nonumber
\end{align}
Consequently,
\begin{align}
    \Eb \insquare{ z_{\vec{x}}^{\ell\, \prime} z_{\vec{x}}^{\ell}  } &= p(N,k,t); \quad \ell \neq \ell' \Rightarrow  \Eb \insquare{ z_{\vec{x}}^{\ell\, \prime} z_{\vec{v}}^{\ell'}  } = \frac{1}{2}\inparen{ R^{t}(\vec{x}, \vec{v}) - q }; \quad \Eb\insquare{ z_{\vec{x}}^{\ell\, \prime} y_{\vec{v}} } = \frac{1}{2}\inparen{ R^{t}(\vec{x}, \vec{v}) - q }
    \label{eq:perceptron_vInterpolation_smallComputations_S_M,t_case}
\end{align}
Returning to the expression involving $\Eb_M$ in \eqref{eq:perceptron_vInterpolation_S_v^1'_u''(S_v^1)u'(S_v^2)f}, we have for fixed $\vec{x}^1,\dots, \vec{x}^n$, by Gaussian integration by parts,
\begin{align*}
    \Eb_M S_v^{1\,\prime} \frac{C}{Z_{t,0,v}^n} &\equiv \Eb_M z_{\vec{x}^1}^{1\, \prime} \frac{F_{\vec{x}^1,\dots,\vec{x}^n}\inparen{ \inparen{z_{\vec{x}}^{\ell}}  }}{ F_{1}\inparen{\inparen{y_{\vec{x}}}}^n  } =\textnormal{TCBN} + \textnormal{TCBD}
\end{align*}
which stands for ``term created by numerator'' and ``term created by denominator'' respectively, where
\begin{align*}
    \textnormal{TCBN} &:= \sum_{\vec{v}, \ell} \Eb_{M}\insquare{  z_{\vec{x}^1}^{1\, \prime} z_{\vec{v}}^{\ell}  } \Eb_{M} \insquare{ \frac{\partial F_{\vec{x}^1,\dots,\vec{x}^n}}{ \partial g_{\vec{v}}^{\ell} }\inparen{\inparen{  z_{\vec{x}}^{\ell}  }} \frac{1}{F_{1}\inparen{ \inparen{y_{\vec{x}}}  }^n} }\\
    \textnormal{TCBD} &:= -n \sum_{\vec{v}} \Eb_{M}\insquare{  z_{\vec{x}^1}^{1\, \prime} y_{\vec{v}}  } \Eb_{M} \insquare{ \frac{\partial F_1}{ \partial g_{\vec{v}} } \inparen{ \inparen{y_{\vec{x}}}  } \frac{F_{\vec{x}^1,\dots,\vec{x}^n} \inparen{\inparen{  z_{\vec{x}}^{\ell}  }} }{F_{1}\inparen{ \inparen{y_{\vec{x}}}  }^{n+1}} }
\end{align*}
Working on $\textnormal{TCBD}$ first, we have
\begin{align}
    \frac{\partial F_1}{\partial g_{\vec{v}}} = \Eb_{\xi} w_*(\vec{v}) u'\inparen{ g_{\vec{v}} + \xi_*^1  } \exp u\inparen{  g_{\vec{v}} + \xi_*^1   } = \Eb_{\xi} w_*(\vec{v}, \xi^{n+1}) u'\inparen{ g_{\vec{v}} + \xi_*^{n+1}  } \exp u\inparen{  g_{\vec{v}} + \xi_*^{n+1}   },
    \label{eq:perceptron_vInterpolation_partialF_1_partialgv}
\end{align}
where in the second equality, we simply re-labled $\xi^1_*$ to be $\xi^{n+1}_*$; this is possible because the $\xi^1_*$ is encased in the integral $\Eb_{\xi}$. Plugging the above into $\textnormal{TCBD}$, using that $z_{\vec{v}}^{n+1} = y_{\vec{v}} + \xi_*^{n+1}$, and \eqref{eq:perceptron_vInterpolation_smallComputations_S_M,t_case}, we find that $\textnormal{TCBD}$'s contribution to \eqref{eq:perceptron_vInterpolation_S_v^1'_u''(S_v^1)u'(S_v^2)f} is
\begin{align}
    &-\frac{n}{2} \Eb \sum_{\vec{x}^1,\dots, \vec{x}^n, \vec{v}} w\inparen{ R^t(\vec{x}^1, \vec{v}) - q  } \frac{ F_{\vec{x}^1,\dots,\vec{x}^n}\inparen{\inparen{  z_{\vec{x}}^{\ell}  }} \Eb\insquare{ w_{*}(\vec{v}, \xi^{n+1}) u'(z_{\vec{v}}^{n+1}) \exp \inparen{u(z_{\vec{v}}^{n+1})}   } }{F_{1}\inparen{\inparen{y_{\vec{x}}}}^{n+1}} \nonumber\\
    &\quad= -\frac{n}{2} \Eb \sum_{\vec{x}^1,\dots, \vec{x}^{n+1}} \frac{1}{Z_{t,0,v}^{n+1}} w w_{*}^{n+1} C \inparen{R_{1,n+1}^{t} - q} u'(S_v^{n+1}) \exp u\inparen{ S_v^{n+1} } \nonumber \\
    &\quad= -\frac{n}{2} \nu_{t,0,v} \insquare{ f\inparen{R_{1,n+1}^{t} - q} u''(S_v^1) u'(S_v^2) u'(S_v^{n+1})  },
    \label{eq:perceptron_vInterpolation_TCBD_contribution}
\end{align}
where in the second line, we rename $\vec{v}$ to $\vec{x}^{n+1}$, and recall that $z_{\vec{x}^{n+1}}^{n+1} = S_v^{n+1}$; we also adopt the notation $R^t_{\ell, \ell'} := R^t(\vec{x}^\ell, \vec{x}^{\ell'})$.

We now turn to $\textnormal{TCBN}$. First note that
\begin{align*}
    \frac{\partial F_{\vec{x}^1,\dots,\vec{x}^n}}{\partial g_{\vec{v}}^{\ell}} \inparen{ \inparen{z_{\vec{x}}^{\ell}}  } = 0
\end{align*}
whenever $\vec{v} \neq \vec{x}^{\ell}$. Therefore, we see a reduction in the number of terms in $\textnormal{TCBN}$:
\begin{align*}
    \textnormal{TCBN} &= \sum_{\ell \leq n} \Eb_{M} \insquare{ z_{\vec{x}}^{1\,\prime} z_{\vec{x}^{\ell}}^{\ell}  } \Eb_{M} \insquare{ \frac{\partial F_{\vec{x}^1,\dots,\vec{x}^n}}{ \partial g_{\vec{x}^{\ell}}^{\ell} }\inparen{\inparen{  z_{\vec{x}}^{\ell}  }} \frac{1}{F_{1}\inparen{ \inparen{y_{\vec{x}}}  }^n} }.
\end{align*}
We further evaluate the partial derivative above, noting the extra terms that appear for $\ell = 1, 2$:
\begin{align}
    \frac{\partial F_{\vec{x}^1,\dots,\vec{x}^n}}{ \partial g_{\vec{x}^{\ell}}^{\ell} }\inparen{\inparen{  z_{\vec{x}}^{\ell}  }} &= f(\vec{x}^1,\dots,\vec{x}^n) u''(z_{\vec{x}^1}^1) u'(z_{\vec{x}^2}^2) u'(z_{\vec{x}^\ell}^{\ell}) \exp \inparen{  \sum_{\ell \leq n} u(z_{\vec{x}^\ell}^{\ell}) } \nonumber\\
    &\quad\quad + f(\vec{x}^1,\dots,\vec{x}^n) u'''(z_{\vec{x}^1}^1) u'(z_{\vec{x}^2}^2) \exp \inparen{  \sum_{\ell \leq n} u(z_{\vec{x}^\ell}^{\ell}) } \rchi_{\ell = 1} \nonumber\\
    &\quad\quad + f(\vec{x}^1,\dots,\vec{x}^n) u''(z_{\vec{x}^1}^1) u''(z_{\vec{x}^2}^2) \exp \inparen{  \sum_{\ell \leq n} u(z_{\vec{x}^\ell}^{\ell}) } \rchi_{\ell = 2}.
    \label{eq:perceptron_vInterpolation_partialF_partial_g_xl^l}
\end{align}
Using \eqref{eq:perceptron_vInterpolation_smallComputations_S_M,t_case} and the above, we can determine the contribution of $\textnormal{TCBN}$ in \eqref{eq:perceptron_vInterpolation_S_v^1'_u''(S_v^1)u'(S_v^2)f}. Combining with $\textnormal{TCBD}$'s contribution \eqref{eq:perceptron_vInterpolation_TCBD_contribution}, we have that the first term in \eqref{eq:perceptron_vInterpolation_u'(S_v^1)u'(S_v^2)f} is given by
\begin{align}
    \nu_{t,0,v}\insquare{ S_v^{1\, \prime} u''(S_v^1) u'(S_v^2) f  } &= p(N,k,t) \inparen{  \nu_{t,0,v}\insquare{f u''(S_v^1) u'(S_v^1) u'(S_v^2)} + \nu_{t,0,v}\insquare{f u'''(S_v^1) u'(S_v^2)}  } \nonumber\\
    &\quad\quad + \frac{1}{2} \nu_{t,0,v}\insquare{f \inparen{R^t_{1,2} - q} u''(S_v^1) u''(S_v^2)} \nonumber\\
    &\quad\quad + \frac{1}{2} \sum_{2 \leq \ell \leq n} \nu_{t,0,v} \insquare{f \inparen{R^t_{1,\ell} - q} u''(S_v^1) u'(S_v^2) u'(S_v^\ell) } \nonumber\\
    &\quad\quad -\frac{n}{2} \nu_{t,0,v} \insquare{f \inparen{R^t_{1,n+1} - q} u''(S_v^1) u'(S_v^2) u'(S_v^{n+1}) }
    \label{eq:perceptron_vInterpolation_u'(S_v^1)u'(S_v^2)f_FIRST_TERM}
\end{align}
Clearly, there is an analogous expression for the second term in \eqref{eq:perceptron_vInterpolation_u'(S_v^1)u'(S_v^2)f}, with the roles of $1$ and $2$ reversed. For the third term in \eqref{eq:perceptron_vInterpolation_u'(S_v^1)u'(S_v^2)f}, there are two cases to consider: $\ell = 1,2$ and $\ell \geq 3$. 

When $\ell = 1$, the same mechanism as above yields
\begin{align}
    \nu_{t,0,v} \insquare{ S_{v}^{1\, \prime} u^{\prime 2}(S_{v}^1) u^{\prime}(S_{v}^2)  f } &= p(N,k,t) \inparen{  2 \nu_{t,0,v}\insquare{f u''(S_v^1) u'(S_v^2)} + \nu_{t,0,v}\insquare{f u^{\prime 3}(S_v^1) u'(S_v^2)}  } \nonumber\\
    &\quad\quad + \frac{1}{2} \nu_{t,0,v}\insquare{f \inparen{R^t_{1,2} - q} u^{\prime 2}(S_v^1) u''(S_v^2)} \nonumber\\
    &\quad\quad + \frac{1}{2} \sum_{2 \leq \ell \leq n} \nu_{t,0,v} \insquare{f \inparen{R^t_{1,\ell} - q} u^{\prime 2}(S_v^1) u'(S_v^2) u'(S_v^\ell) } \nonumber\\
    &\quad\quad -\frac{n}{2} \nu_{t,0,v} \insquare{f \inparen{R^t_{1,n+1} - q} u^{\prime 2}(S_v^1) u'(S_v^2) u'(S_v^{n+1}) }
    \label{eq:perceptron_vInterpolation_u'(S_v^1)u'(S_v^2)f_THIRD_TERM_l=1},
\end{align}
and an analogous expression holds for $\ell = 2$, with roles of $1$ and $2$ reversed. When $\ell \geq 3$, say $\ell = 3$,
\begin{align}
    \nu_{t,0,v} \insquare{ S_{v}^{3\, \prime} u^{\prime}(S_{v}^1) u^{\prime}(S_{v}^2) u^{\prime}(S_{v}^3)  f } &= p(N,k,t) \inparen{ \nu_{t,0,v}\insquare{f u'(S_v^1) u'(S_v^2) u^{\prime\, 2}(S_v^{3}) } + \nu_{t,0,v}\insquare{f u^{\prime}(S_v^1) u'(S_v^2) u''(S_v^{3}) }  } \nonumber\\
    &\quad\quad + \frac{1}{2} \nu_{t,0,v}\insquare{f \inparen{R^t_{1,3} - q} u^{\prime \prime}(S_v^1) u'(S_v^2) u'(S_v^3)} \nonumber\\
    &\quad\quad + \frac{1}{2} \nu_{t,0,v}\insquare{f \inparen{R^t_{2,3} - q} u^{\prime}(S_v^1) u''(S_v^2) u'(S_v^3)} \nonumber\\
    &\quad\quad + \frac{1}{2} \sum_{\ell' \leq n, \ell' \neq 3} \nu_{t,0,v} \insquare{f \inparen{R^t_{3,\ell'} - q} u^{\prime}(S_v^1) u'(S_v^2) u'(S_v^3) u'(S_v^{\ell'}) } \nonumber\\
    &\quad\quad -\frac{n}{2} \nu_{t,0,v} \insquare{f \inparen{R^t_{3,n+1} - q} u^{\prime}(S_v^1) u'(S_v^2) u'(S_v^{3}) u'(S_v^{n+1}) }
    \label{eq:perceptron_vInterpolation_u'(S_v^1)u'(S_v^2)f_THIRD_TERM_l>=3}.
\end{align}
For the last term in \eqref{eq:perceptron_vInterpolation_u'(S_v^1)u'(S_v^2)f}, we have
\begin{align}
    \nu_{t,0,v} \insquare{ S_{v}^{n+1\, \prime} u^{\prime}(S_{v}^1) u^{\prime}(S_{v}^2) u^{\prime}(S_{v}^{n+1})  f } &= \frac{1}{2} \nu_{t,0,v}\insquare{f \inparen{R^t_{1,n+1} - q} u^{\prime \prime}(S_v^1) u'(S_v^2) u'(S_v^{n+1})} \nonumber\\
    &\quad\quad + \frac{1}{2} \nu_{t,0,v}\insquare{f \inparen{R^t_{2,n+1} - q} u^{\prime}(S_v^1) u''(S_v^2) u'(S_v^{n+1})} \nonumber\\
    &\quad\quad + \frac{1}{2} \sum_{\ell \leq n}  \nu_{t,0,v}\insquare{f \inparen{R^t_{\ell,n+1} - q} u^{\prime}(S_v^1) u'(S_v^2) u'(S_v^{n+1}) u'(S_v^{\ell})} \nonumber\\
    &\quad\quad -\frac{n+1}{2} \nu_{t,0,v}\insquare{f \inparen{R^t_{n+1,n+2} - q} u^{\prime}(S_v^1) u'(S_v^2) u'(S_v^{n+1}) u'(S_v^{n+1})}.
    \label{eq:perceptron_vInterpolation_u'(S_v^1)u'(S_v^2)f_FOURTH_TERM}
\end{align}
Observe that \eqref{eq:perceptron_vInterpolation_u'(S_v^1)u'(S_v^2)f_FIRST_TERM}, \eqref{eq:perceptron_vInterpolation_u'(S_v^1)u'(S_v^2)f_THIRD_TERM_l=1}, \eqref{eq:perceptron_vInterpolation_u'(S_v^1)u'(S_v^2)f_THIRD_TERM_l>=3}, and \eqref{eq:perceptron_vInterpolation_u'(S_v^1)u'(S_v^2)f_FOURTH_TERM} all involve only two types of terms:
\begin{align*}
    p(N,k,t) \, \nu_{t,0,v}\insquare{ f A  }; \quad \textnormal{ and } \quad \nu_{t,0,v}\insquare{ f (R_{\ell, \ell'}^t - q) A  },
\end{align*}
where $A$ is some monomial in $u'(S_v^{\ell})$, $u''(S_v^{\ell})$, $u'''(S_v^{\ell})$, $\ell \leq n+2$. Using that, for every $0 \leq t \leq 1$,
\begin{align}
    \abs{p(N,k,t)} \leq \frac{k}{N}; \quad \textnormal{ and } \quad \abs{R^{t}_{\ell, \ell'} - q } \leq \abs{R_{\ell,\ell'} - q} + \frac{k}{N}
    \label{eq:perceptron_vInterpolation_p(N,k,t)_R^t_l,l'}
\end{align}
we bound the two types of terms respectively, for every $t$, as
\begin{align}
    \abs{p(N,k,t) \nu_{t,0,v} \insquare{fA}} &\leq \frac{K(D) k}{N} \nu_{t,0,v}\abs{f}, \nonumber\\
    \abs{ \nu_{t,0,v}\insquare{ f (R_{\ell, \ell'}^t - q) A  } } &\leq K(D)\inparen{  \inparen{ \nu \abs{f}^{\tau_1} }^{1/\tau_1} \inparen{\nu \abs{ R_{1, 2}- q  }^{\tau_2} }^{1/\tau_2} + \frac{k}{N} \nu_{t,0,v}\abs{f}  },
    \label{eq:perceptron_vInterpolation_p(N,k,t)_f(R12-q)A_bounds}
\end{align}
where $K(D)$ is some constant that depends on $D$, and where in the second line we used H{\"o}lder's inequality. Combining \eqref{eq:perceptron_vInterpolation_u'(S_v^1)u'(S_v^2)f_FIRST_TERM}, \eqref{eq:perceptron_vInterpolation_u'(S_v^1)u'(S_v^2)f_THIRD_TERM_l=1}, \eqref{eq:perceptron_vInterpolation_u'(S_v^1)u'(S_v^2)f_THIRD_TERM_l>=3}, \eqref{eq:perceptron_vInterpolation_u'(S_v^1)u'(S_v^2)f_FOURTH_TERM} into \eqref{eq:perceptron_vInterpolation_u'(S_v^1)u'(S_v^2)f}, and using the bounds \eqref{eq:perceptron_vInterpolation_p(N,k,t)_f(R12-q)A_bounds}, we obtain
\begin{align*}
    \abs{ \frac{\partial}{\partial v} \nu_{t,0,v}\insquare{ u'(S_v^1) u'(S_v^2)f  }  } &\leq K(n,D) k \insquare{ \inparen{ \nu_{t,0,v}\abs{f}^{\tau_1} }^{1/\tau_1} \inparen{ \nu_{t,0,v}\abs{R_{1,2} - q}^{\tau_2} }^{1/\tau_2}  + \frac{1}{N} \nu_{t,0,v}\abs{f}  },
\end{align*}
which finishes the case $B_v = u'(S_v^1) u'(S_v^2)$. The cases $u'(S_v^1)u'(S_v^{n+1})$, $u'(S_v^{n+1})u'(S_v^{n+2})$, $u'^{2}(S_v^1)$, $u'^{2}(S_v^{n+1})$, $u''(S_v^1)$, $u''(S_v^{n+1})$ will follow similarly---taking derivative with respect to $v$ and using Gaussian integration by parts will yield terms of the type in \eqref{eq:perceptron_vInterpolation_p(N,k,t)_f(R12-q)A_bounds}, from which the conclusion follows.

We next turn to the cases involving $S_v^{0}$'s. Unsurprisingly, the mechanism is basically the same, we illustrate this with the case $B_v = u'(S_v^{0,1})u'(S_v^{0,2})$. Define
\begin{align*}
    S_{v}^{0, \ell\, \prime} := \frac{\partial }{\partial v} S_{v}^{0,\ell} = \frac{1}{2\sqrt{v}} S_{M}^{0,\ell} - \frac{1}{2\sqrt{1-v}} \theta^\ell.
\end{align*}
Direct computation yields
\begin{align}
    \frac{\partial}{\partial v} \nu_{t,0,v}\insquare{  u'(S_v^{0,1})u'(S_v^{0,2}) f    } &= \nu_{t,0,v} \insquare{ S_{v}^{0,1\, \prime} u''(S_v^{0,1}) u'(S_v^{0,2}) f  } \nonumber\\
    &\quad + \nu_{t,0,v} \insquare{ S_{v}^{0,2\, \prime} u'(S_v^{0,1}) u''(S_v^{0,2}) f  } \nonumber\\
    &\quad + \sum_{\ell \leq n} \nu_{t,0,v} \insquare{ S_{v}^{\ell\, \prime} u'(S_v^{0,1}) u'(S_v^{0,2}) u'(S_v^{\ell}) f  } \nonumber \\
    &\quad -n \nu_{t,0,v} \insquare{ S_{v}^{n+1\, \prime} u'(S_v^{0,1}) u'(S_v^{0,2}) u'(S_{v}^{n+1}) f  }.
    \label{eq:perceptron_vInterpolation_u'(S_v^0,1)u'(S_v^0,2)f}
\end{align}
Compared to \eqref{eq:perceptron_vInterpolation_u'(S_v^1)u'(S_v^2)f}, the only difference is that some of the $S_v^{\ell}$'s have been replaced by $S_v^{0,\ell}$'s. Recalling the definitions for $z_{\vec{x}}^{\ell}$, $y_{\vec{x}}$, we are motivated to introduce
\begin{align*}
    z_{\vec{x}}^{0,\ell} &:= \sqrt{v} S_{M}^{0,\ell} + \sqrt{1-v} \inparen{\sqrt{q}z + \sqrt{1-q}\xi^{\ell}} \\
    z_{\vec{x}}^{0,\ell\, \prime} &:= \frac{1}{2\sqrt{v}} S_{M}^{0,\ell} - \frac{1}{2\sqrt{1-v}} \inparen{\sqrt{q}z + \sqrt{1-q}\xi^{\ell}}.
\end{align*}
The correlation computations in \eqref{eq:perceptron_vInterpolation_smallComputations_S_M,t_case} are updated to include the following: 
\begin{align}
    \Eb\insquare{  z_{\vec{x}}^{0, \ell\, \prime} z_{\vec{x}}^{0,\ell}  } &= p(N,k,0); \quad \Eb\insquare{  z_{\vec{x}}^{0, \ell\, \prime} z_{\vec{x}}^{\ell}  } = p(N,k,0); \nonumber\\
    \ell \neq \ell' \Rightarrow \Eb\insquare{  z_{\vec{x}}^{0,\ell\,\prime} z_{\vec{v}}^{0, \ell'}  } &= \frac{1}{2}\inparen{ R^0(\vec{x}, \vec{v}) - q  }; \quad \textnormal{and} \quad \Eb\insquare{  z_{\vec{x}}^{0,\ell\,\prime} z_{\vec{v}}^{\ell'}  } = \frac{1}{2}\inparen{ R^0(\vec{x}, \vec{v}) - q  } \nonumber\\
    \Eb\insquare{ z_{\vec{x}}^{0,\ell\,\prime} y_{\vec{v}}  } &= \frac{1}{2}\inparen{ R^0(\vec{x}, \vec{v}) - q  }.
    \label{eq:perceptron_vInterpolation_smallComputations_S_M^0_case}
\end{align}
Consider the first term in \eqref{eq:perceptron_vInterpolation_u'(S_v^0,1)u'(S_v^0,2)f}---recall the notations $w$, $w_*^{\ell}$, $F_1$, and $\Eb_{M}$ and write
\begin{align}
    \nu_{t,0,v} \insquare{ S_{v}^{0,1\, \prime} u''(S_v^{0,1}) u'(S_v^{0,2}) f  } &= \Eb\sum_{\vec{x}^1,\dots,\vec{x}^n} w \Eb_{M} z_{\vec{x}^{1}}^{0,1\,\prime} 
    \frac{ \tilde{F}_{\vec{x}^1,\dots,\vec{x}^n}\inparen{ \inparen{ z_{\vec{x}}^{0, \ell} }, \inparen{ z_{\vec{x}}^{\ell}  }  }  }{  F_1\inparen{\inparen{ y_{\vec{x}}  }}^n  }
    \label{eq:perceptron_vInterpolation_S_v^01'_u''(S_v^01)u'(S_v^02)f}
\end{align}
where $\tilde{F}_{\vec{x}^1,\dots,\vec{x}^n}$ is the function given by
\begin{align*}
    \tilde{F}_{\vec{x}^1,\dots,\vec{x}^n}\inparen{ \inparen{g_{\vec{x}}^{0,\ell}}, \inparen{ g_{\vec{x}}^{\ell}  }  } &= f(\vec{x}^1,\dots,\vec{x}^n) u''(g_{\vec{x}^1}^{0,1}) u'( g_{\vec{x}^{2}}^{0,2} ) \exp\inparen{  \sum_{\ell \leq n}  u\inparen{z_{\vec{x}^{\ell}}^{\ell}}  } 
\end{align*}
For fixed $\vec{x}^1,\dots,\vec{x}^n$, we have
\begin{align*}
    \Eb_{M} z_{\vec{x}^{1}}^{0,1\,\prime} 
    \frac{ \tilde{F}_{\vec{x}^1,\dots,\vec{x}^n}\inparen{ \inparen{ z_{\vec{x}}^{0, \ell} }, \inparen{ z_{\vec{x}}^{\ell}  }  }  }{  F_1\inparen{\inparen{ y_{\vec{x}}  }}^n  } = \tilde{\textnormal{TCBN}} + \tilde{\textnormal{TCBD}},
\end{align*}
where now
\begin{align}
    \tilde{\textnormal{TCBN}} &:= \sum_{\vec{v}, \ell} \Eb_{M}\insquare{  z_{\vec{x}^1}^{0,1\,\prime} z_{\vec{v}}^{0,\ell}  } \Eb_{M} \insquare{ \frac{\partial \tilde{F}_{\vec{x}^1,\dots,\vec{x}^n}}{ \partial g_{\vec{v}}^{0,\ell} }\inparen{\inparen{  z_{\vec{x}}^{0,\ell}  }, \inparen{  z_{\vec{x}}^{\ell}  }} \frac{1}{F_{1}\inparen{ \inparen{y_{\vec{x}}}  }^n} } \nonumber\\
    &\quad\quad + \sum_{\vec{v}, \ell} \Eb_{M}\insquare{  z_{\vec{x}^1}^{0,1\,\prime} z_{\vec{v}}^{\ell}  } \Eb_{M} \insquare{ \frac{\partial \tilde{F}_{\vec{x}^1,\dots,\vec{x}^n}}{ \partial g_{\vec{v}}^{\ell} }\inparen{\inparen{  z_{\vec{x}}^{0,\ell}  }, \inparen{  z_{\vec{x}}^{\ell}  }} \frac{1}{F_{1}\inparen{ \inparen{y_{\vec{x}}}  }^n} } \nonumber\\
    \tilde{\textnormal{TCBD}} &:= -n \sum_{\vec{v}} \Eb_{M}\insquare{  z_{\vec{x}^1}^{0,1\, \prime} y_{\vec{v}}  } \Eb_{M} \insquare{ \frac{\partial F_1}{ \partial g_{\vec{v}} } \inparen{ \inparen{y_{\vec{x}}}  } \frac{F_{\vec{x}^1,\dots,\vec{x}^n} \inparen{\inparen{  z_{\vec{x}}^{\ell}  }} }{F_{1}\inparen{ \inparen{y_{\vec{x}}}  }^{n+1}} }. \nonumber
\end{align}
From the computations in \eqref{eq:perceptron_vInterpolation_smallComputations_S_M^0_case}, and reusing the calculation for $\partial F_1/ \partial g_{\vec{v}}$ in \eqref{eq:perceptron_vInterpolation_partialF_1_partialgv}, and the same mechanism of relabeling the $\xi_*^1$'s to $\xi_*^{n+1}$, and renaming $\vec{v}$ to $\vec{x}^{n+1}$, we obtain that $\tilde{\textnormal{TCBD}}$'s contribution to \eqref{eq:perceptron_vInterpolation_S_v^01'_u''(S_v^01)u'(S_v^02)f} is 
\begin{align}
    -\frac{n}{2} \nu_{t,0,v}\insquare{ f\inparen{ R^{0}_{1,n+1} - q  } u''(S_{v}^{0,1}) u'(S_v^{0,2}) u'(S_v^{n+1})  }.
    \label{eq:perceptron_vInterpolation_tildeTCBD_contribution}
\end{align}
For $\tilde{\textnormal{TCBN}}$, we also have that $\partial \tilde{F}_{\vec{x}^1,\dots,\vec{x}^n}/\partial g_{\vec{v}}^{\ell}$ and $\partial \tilde{F}_{\vec{x}^1,\dots,\vec{x}^n}/\partial g_{\vec{v}}^{0,\ell}$ equal zero whenever $\vec{v} \neq \vec{x}^{\ell}$. By similar derivative computations to \eqref{eq:perceptron_vInterpolation_partialF_partial_g_xl^l}, and using \eqref{eq:perceptron_vInterpolation_smallComputations_S_M^0_case}, we can determine $\tilde{\textnormal{TCBN}}$'s contribution to \eqref{eq:perceptron_vInterpolation_S_v^01'_u''(S_v^01)u'(S_v^02)f}. Together with \eqref{eq:perceptron_vInterpolation_tildeTCBD_contribution}, we obtain that \eqref{eq:perceptron_vInterpolation_S_v^01'_u''(S_v^01)u'(S_v^02)f} is given by
\begin{align*}
    \nu_{t,0,v}\insquare{  S_{v}^{0,1\,\prime} u''(S_v^{0,1}) u'(S_v^{0,2}) f  } &= p(N,k,0) \inparen{  \nu_{t,0,v}\insquare{f u''(S_v^{0,1}) u'(S_v^1) u'(S_v^{0,2})} + \nu_{t,0,v}\insquare{f u'''(S_v^{0,1}) u'(S_v^{0,2})}  } \nonumber\\
    &\quad\quad + \frac{1}{2} \nu_{t,0,v}\insquare{f \inparen{R^0_{1,2} - q} u''(S_v^{0,1}) u''(S_v^{0,2})} \nonumber\\
    &\quad\quad + \frac{1}{2} \sum_{2 \leq \ell \leq n} \nu_{t,0,v} \insquare{f \inparen{R^0_{1,\ell} - q} u''(S_v^{0,1}) u'(S_v^{0,2}) u'(S_v^\ell) } \nonumber\\
    &\quad\quad -\frac{n}{2} \nu_{t,0,v} \insquare{f \inparen{R^0_{1,n+1} - q} u''(S_v^{0,1}) u'(S_v^{0,2}) u'(S_v^{n+1}) }.
\end{align*}
One should compare this to \eqref{eq:perceptron_vInterpolation_u'(S_v^1)u'(S_v^2)f_FIRST_TERM}. It will come as no surprise that all the other terms in \eqref{eq:perceptron_vInterpolation_u'(S_v^0,1)u'(S_v^0,2)f} admit similar expressions. That is, they all consist of terms of the type:
\begin{align*}
    p(N,k,0) \nu_{t,0,v}\insquare{fA}; \quad \textnormal{ and } \quad \nu_{t,0,v}\insquare{ f(R_{\ell,\ell'}^0 - q) A  },
\end{align*}
where now $A$ is some monomial in $u'(S_v^{\ell})$, $u''(S_v^{\ell})$, $u'''(S_v^{\ell})$, $u'(S_v^{0,\ell})$, $u''(S_v^{0,\ell})$, $u'''(S_v^{0,\ell})$, with $\ell \leq n+2$. All the $S_v^\ell$'s and $S_v^{0,\ell}$'s are wrapped in $u$ or its derivatives, which are bounded in magnitude, therefore we may re-use \eqref{eq:perceptron_vInterpolation_p(N,k,t)_R^t_l,l'} to obtain
\begin{align*}
    \abs{ \frac{\partial}{\partial v} \nu_{t,0,v}\insquare{ u'(S_v^{0,1}) u'(S_v^{0,2})f  }  } &\leq K(n,D) k \insquare{ \inparen{ \nu_{t,0,v}\abs{f}^{\tau_1} }^{1/\tau_1} \inparen{ \nu_{t,0,v}\abs{R_{1,2} - q}^{\tau_2} }^{1/\tau_2}  + \frac{1}{N} \nu_{t,0,v}\abs{f}  }.
\end{align*}
This finishes the case for $B_v = u'(S_v^{0,1})u'(S_v^{0,2})$. The remaining cases for $B_v$ equal to one of $u'(S_v^{0,1})u'(S_v^{0,n+1})$, $u'(S_v^{0,n+1})u'(S_v^{0,n+2})$, $u'^{2}(S_v^{0,1})$, $u'^{2}(S_v^{0,n+1})$ follows similarly.

The case $B_v \equiv 1$ is in fact the easiest because we have the fewest terms:
\begin{align}
    \frac{\partial}{\partial v} \nu_{t,0,v}\insquare{   f    } =
    \sum_{\ell \leq n} \nu_{t,0,v} \insquare{ S_{v}^{\ell\, \prime} u'(S_v^{\ell}) f  } -n \nu_{t,0,v} \insquare{ S_{v}^{n+1\, \prime} u'(S_{v}^{n+1}) f  },
    \nonumber
\end{align}
and the mechanism described above yields the result.
\end{proofof}

\begin{proofof}{Lemma \ref{lemma:perceptron_ddt_nu_t_f_inTermsOf_nu_t_f}}
We use Proposition \ref{proposition:perceptron_t_interpolation} and the definition of terms $\textnormal{I}$, $\textnormal{II}$, and $\textnormal{III}$ there. There are two cases to consider---that of terms $\textnormal{II}$ and $\textnormal{III}$, and that of term $\textnormal{I}$. We start with the former---consider a generic term in $\textnormal{II}$ and $\textnormal{III}$ and use a sequence of triangle inequalities to write
\begin{align}
     \abs{  \nu_t\insquare{ x_1^\ell x_1^{\ell'} u'\inparen{ S_{M,t}^\ell  } u'\inparen{ S_{M,t}^{\ell'} } f  } - \nu_t\insquare{ x_1^\ell x_1^{\ell'} u'\inparen{ S_{M}^{0,\ell}  } u'\inparen{ S_{M}^{0,\ell'} } f  } }  \leq A_1 + A_2 + B_1 + B_2,
    \label{eq:perceptron_termsII,III_type_bound}
\end{align}
where 
\begin{align*}
    A_1 &:= \abs{  \nu_t\insquare{ x_1^\ell x_1^{\ell'} u'\inparen{ S_{M,t}^\ell  } u'\inparen{ S_{M,t}^{\ell'} } f  } - \nu_{t,0}\insquare{ x_1^\ell x_1^{\ell'} u'\inparen{ S_{M,t}^\ell  } u'\inparen{ S_{M,t}^{\ell'} } f  }   } \\
    A_2 &:= \abs{ \nu_t\insquare{ x_1^\ell x_1^{\ell'} u'\inparen{ S_{M}^{0,\ell}  } u'\inparen{ S_{M}^{0,\ell'} } f  } - \nu_{t,0}\insquare{ x_1^\ell x_1^{\ell'} u'\inparen{ S_{M}^{0,\ell}  } u'\inparen{ S_{M}^{0,\ell'} } f  }  } \\
    B_1 &:= \abs{  \nu_{t,0}\insquare{ x_1^\ell x_1^{\ell'} u'\inparen{ S_{M,t}^\ell  } u'\inparen{ S_{M,t}^{\ell'} } f  } - \nu_{t,0,0}\insquare{ x_1^\ell x_1^{\ell'} u'\inparen{ \theta^\ell  } u'\inparen{ \theta^{\ell'} } f  }   } \\
    B_2 &:= \abs{  \nu_{t,0}\insquare{ x_1^\ell x_1^{\ell'} u'\inparen{ S_{M}^{0,\ell}  } u'\inparen{ S_{M}^{0,\ell'} } f  } - \nu_{t,0,0}\insquare{ x_1^\ell x_1^{\ell'} u'\inparen{ \theta^\ell  } u'\inparen{ \theta^{\ell'} } f  }   },
\end{align*}
where we have used the relation $S_{0}^{0,\ell} = \theta^\ell = S_{0}^{\ell}$.

Let $\cR_1$, $\cR_2$ represent quantities
\begin{align*}
    \cR_1 &:= \frac{K(n,k,D)}{N} \nu_{t}\abs{f} \\
    \cR_2 &:= K(n,k,D) \inparen{ \inparen{ \nu_t\abs{f}^{\tau_1}}^{1/\tau_1} \inparen{ \nu_t\abs{R_{1,2} - q}^{\tau_2}}^{1/\tau_2} + \frac{1}{N}\nu_{t}\abs{f}   },
\end{align*}
up to changes in $K(n,k,D)$. Noting that $x_1^\ell x_1^{\ell'} u'\inparen{ S_{M,t}^\ell  } u'\inparen{ S_{M,t}^{\ell'} } f$ and $x_1^\ell x_1^{\ell'} u'\inparen{ S_{M}^{0,\ell}  } u'\inparen{ S_{M}^{0,\ell'} } f$ are functions independent of the randomness in $\inparen{\tilde{g}_{j,M}}_{j \leq M}$, Lemma \ref{lemma:perceptron_gammaInterpolation} and \eqref{eq:perceptron_getRidOfgamma} then yield that 
\begin{align*}
    A_1,\; A_2 \leq \cR_1.
\end{align*}
Furthermore, we apply Lemma \ref{lemma:perceptron_vInterpolation}, followed by \eqref{eq:perceptron_getRidOfv} to get rid of $v$ on the RHS, and then \eqref{eq:perceptron_getRidOfgamma} to get rid of $\gamma$ on the RHS, to obtain
\begin{align*}
    \abs{  \nu_{t,0}\insquare{ u'\inparen{ S_{M,t}^\ell  } u'\inparen{ S_{M,t}^{\ell'} } f  } - \nu_{t,0,0}\insquare{  u'\inparen{ \theta^\ell  } u'\inparen{ \theta^{\ell'} } f  }   } \leq \cR_2 \\
    \abs{  \nu_{t,0}\insquare{ u'\inparen{ S_{M}^{0,\ell}  } u'\inparen{ S_{M}^{0,\ell'} } f  } - \nu_{t,0,0}\insquare{  u'\inparen{ \theta^\ell  } u'\inparen{ \theta^{\ell'} } f  }   } \leq \cR_2.
\end{align*}
Using that $\abs{x_1^{\ell} x_1^{\ell'}} = 1$, we replace $f$ in the above display with $x_1^{\ell} x_1^{\ell'} f$ to obtain
\begin{align*}
    B_1,\; B_2 \leq \cR_2.
\end{align*}
Therefore, from \eqref{eq:perceptron_termsII,III_type_bound}, we deduce that
\begin{align}
    \textnormal{II} + \textnormal{III} \leq \alpha K(n,k,D)  \insquare{ \inparen{ \nu_{t}\abs{f}^{\tau_1} }^{1/\tau_1} \inparen{ \nu_{t}\abs{R_{1,2} - q}^{\tau_2} }^{1/\tau_2}  + \frac{1}{N} \nu_{t}\abs{f}  }.
    \label{eq:perceptron_termsII,III_final_bound}
\end{align}
It remains only to bound term $\textnormal{I}$ by the same expression. Observe that for any $\ell \leq n$, since $f$ is independent of the $\xi^{\ell}$'s, we have the identity
\begin{align*}
    \nu_{t,0,0}\insquare{  u''(\theta^\ell) f  } &= \Eb \frac{\Eb_\xi\inangle{ u''(\theta^\ell) f \exp u(\theta^\ell)  }_{t,\sim}}{  \Eb_\xi\inangle{ \exp u(\theta^1)   }_{t,\sim} } = \Eb \frac{\Eb_\xi\inangle{ u''(\theta^{n+1}) f \exp u(\theta^{n+1})  }_{t,\sim}}{  \Eb_\xi\inangle{ \exp u(\theta^1)   }_{t,\sim} } = \nu_{t,0,0}\insquare{  u''(\theta^{n+1}) f  }.
\end{align*}
Triangle inequalities then yield
\begin{align*}
    \abs{ \nu_{t}\insquare{ u''(S_{M,t}^{\ell}) f  } - \nu_{t}\insquare{ u''(S_{M,t}^{n+1}) f  }  } \leq C_1 + C_2 + D_1 + D_2,
\end{align*}
where
\begin{align*}
    C_1 &:= \abs{ \nu_{t}\insquare{ u''(S_{M,t}^{\ell}) f  } - \nu_{t,0}\insquare{ u''(S_{M,t}^{\ell}) f  }  } \\
    C_2 &:= \abs{ \nu_{t}\insquare{ u''(S_{M,t}^{n+1}) f  } - \nu_{t,0}\insquare{ u''(S_{M,t}^{n+1}) f  }  } \\
    D_1 &:= \abs{ \nu_{t,0}\insquare{ u''(S_{M,t}^{\ell}) f  } - \nu_{t,0,0}\insquare{ u''(\theta^{n+1}) f  }  } \\
    D_2 &:= \abs{ \nu_{t,0}\insquare{ u''(S_{M,t}^{n+1}) f  } - \nu_{t,0,0}\insquare{ u''(\theta^{n+1}) f  }  },
\end{align*}
and by similar reasoning as above, we have that $C_1, C_2 \leq \cR_1$ and $D_1, D_2 \leq \cR_2$ from which we can bound term $\textnormal{I}$ as in the RHS of \eqref{eq:perceptron_termsII,III_final_bound}. The conclusion follows.
\end{proofof}

\begin{proofof}{Theorem \ref{thm:perceptron_+_to_originalSystem}}
We start with \eqref{eq:perceptron_nu^+_to_nu}. Recall the notation $S_{M+1}^{\ell} = N^{-1/2} \sum_{i \leq N} g_{i,M+1} x_i^{\ell}$. Note the identity
\begin{align*}
    \nu^+\insquare{f} = \Eb   \frac{ \inangle{ f \exp \inparen{\sum_{\ell \leq n} u(S_{M+1}^{\ell}) } } }{ \inangle{ \exp u(S_{M+1}^1) }^n }   .
\end{align*}
With the RS quantities defined in \eqref{eq:perceptron_RS_equations_+}, set for $0 \leq v \leq 1$,
\begin{align*}
    S_v^{\ell} &:= \sqrt{v}S_{M+1}^{\ell} + \sqrt{1-v} \theta^{+,\ell} \\
    S_v^{\ell\, \prime} &:= \frac{1}{2\sqrt{v}}S_{M+1}^{\ell} - \frac{1}{2\sqrt{1-v}} \theta^{+,\ell}.
\end{align*}
Define
\begin{align*}
    \nu^{+,v}\insquare{f} &:= \Eb   \frac{ \Eb_\xi \inangle{ f \exp \inparen{\sum_{\ell \leq n} u(S_{v}^{\ell}) } } }{ \inparen{ \Eb_\xi \inangle{ \exp u(S_{v}^1) } }^n }.   
\end{align*}
One notes that $\nu^{+,1}\insquare{f} = \nu^{+}\insquare{f}$, and that $\nu^{+,0}\insquare{f} = \nu\insquare{f}$. The latter follows because $f$ is independent of any $\theta^{+,\ell}$'s. Therefore, it suffices to show that $\ud \nu^{+,v}\insquare{f}/\ud v$ is bounded by the RHS of \eqref{eq:perceptron_nu^+_to_nu}.

We are essentially performing the $v$-interpolation as in Lemma \ref{lemma:perceptron_vInterpolation}, but at the $t = 1$ endpoint. Compute
\begin{align}
    \frac{\ud}{\ud v} \nu^{+,v}\insquare{f} &= \sum_{\ell \leq n} \nu^{+,v}\insquare{ S_v^{\ell\, \prime} u'(S_v^\ell)f  } - n\nu^{+,v}\insquare{ S_{v}^{n+1\, \prime} u'(S_v^{n+1})f  }.
    \label{eq:perceptron_partial_nu^+v_partialv}
\end{align}
Gaussian integration by parts will be used as in the proof of Lemma \ref{lemma:perceptron_vInterpolation}. That is, we use the relations
\begin{align*}
    \Eb_{\textnormal{d}}\insquare{ S_v^{\ell\, \prime} S_v^{\ell}  } = 0; \quad \textnormal{ and } \quad \ell \neq \ell' \Rightarrow \Eb_{\textnormal{d}}\insquare{ S_v^{\ell\,\prime} S_v^{\ell'}  } = \frac{1}{2}\inparen{R_{\ell,\ell'} - q^+},
\end{align*}
and by the method exactly as in that in \eqref{eq:perceptron_vInterpolation_S_v^1'_u''(S_v^1)u'(S_v^2)f}, we obtain that \eqref{eq:perceptron_partial_nu^+v_partialv} consists of $K(n)$ terms of the form 
\begin{align*}
    \nu^{+,v}\insquare{f \inparen{R_{\ell,\ell'} - q^+}A},
\end{align*}
where $A$ is some monomial in $u'(S_v^{\ell})$, $u''(S_v^{\ell})$, $\ell \leq n+2$. By H{\"o}lder's inequality, we obtain
\begin{align*}
     \abs{ \frac{\ud}{\ud v} \nu^{+,v}\insquare{f} } \leq  K(n,D)\insquare{ \inparen{ \nu^{+,v}\abs{f}^{\tau_1} }^{1/\tau_1} \inparen{ \nu^{+,v}\abs{R_{1,2} - q^+}^{\tau_2} }^{1/\tau_2}   }.
\end{align*}
(There is no $k$ dependence in the constant because we are interpolating the original Hamiltonian $-H_{N,M+1}$.) It is straightforward to extend \eqref{eq:perceptron_getRidOfv} to this setting so that for $\bar{f} \geq 0$, $\nu^{+,v}\insquare{\bar{f}} \leq K(n,D) \nu^{+}\insquare{\bar{f}}$. This establishes \eqref{eq:perceptron_nu^+_to_nu}.  

The steps for \eqref{eq:perceptron_nu^+_0_to_nu_0} are almost the same, except that we have to perform a $\gamma$-interpolation first. Note the identity
\begin{align*}
    \nu^{+}_{0}\insquare{f} &= \Eb \inangle{f}^{+}_{0} = \Eb   \frac{ \inangle{ f \exp \inparen{\sum_{\ell \leq n} u(S_{M+1}^{0,\ell}) + \frac{1}{\sqrt{N}}T_{M+1,k}^{\ell} u'(S_{M+1}^{0,\ell}) } }_{0} }{ \inangle{ \exp \inparen{ u(S_{M+1}^{0,1}) + \frac{1}{\sqrt{N}}T_{M+1,k}^{1} u'(S_{M+1}^{0,1}) } }^n_{0} }, 
\end{align*}
where we recall the notation $T_{M+1,k}^{\ell} := \sum_{j \leq k} x_{j}^{\ell} \tilde{g}_{j,M+1}$. The following is similar to the proof of Lemma \ref{lemma:perceptron_gammaInterpolation}. Define
\begin{align*}
        \nu^{+,\gamma}_{0}\insquare{f} &= \Eb   \frac{ \inangle{ f \exp \inparen{\sum_{\ell \leq n} u(S_{M+1}^{0,\ell}) + \sqrt{\frac{\gamma}{N} }T_{M+1,k}^{\ell} u'(S_{M+1}^{0,\ell}) } }_{0} }{ \inangle{ \exp \inparen{ u(S_{M+1}^{0,1}) + \sqrt{\frac{\gamma}{N} }T_{M+1,k}^{1} u'(S_{M+1}^{0,1}) } }^n_{0} }.
\end{align*}
Differentiating, we find that because the contribution brought by each $x_j^{\ell}\tilde{g}_{j,M+1}$, $j \leq k$ is the same,
\begin{align*}
    \frac{\ud}{\ud \gamma}   \nu^{+,\gamma}_{0}\insquare{f} &= \frac{k}{2\sqrt{\gamma N}} \inparen{ \sum_{\ell \leq n} \nu^{+,\gamma}_{0}\insquare{ \tilde{g}_{1,M+1} x_{1}^{\ell} u'(S_{M+1}^{0,\ell}) f  } - n \nu^{+,\gamma}_{0}\insquare{ \tilde{g}_{1,M+1} x_{1}^{n+1} u'(S_{M+1}^{0,n+1}) f  }  } \\
    &= \frac{k}{N} \left( \sum_{\ell < \ell' \leq n}  \nu_{t,\gamma} \insquare{ x_1^{\ell} x_1^{\ell'} u'\inparen{ S_{M+1}^{0,\ell}   } u'\inparen{ S_{M+1}^{0,\ell'}  } f  }    \right. \nonumber\\
    &\quad\quad\quad\quad\quad\quad\quad -n \sum_{\ell \leq n} \nu_{t,\gamma} \insquare{ x_1^{\ell} x_1^{n+1} u'\inparen{ S_{M+1}^{0,\ell}   } u'\inparen{ S_{M+1}^{0,n+1}  } f  } \nonumber\\
    &\quad\quad\quad\quad\quad\quad\quad +\frac{1}{2}\sum_{\ell \leq n} \nu_{t,\gamma} \insquare{ u^{' 2}\inparen{ S_{M+1}^{0,\ell}   } f  } - \frac{n}{2} \nu_{t,\gamma} \insquare{ u^{'2}\inparen{ S_{M+1}^{0,n+1}   }  f  } \nonumber\\
    &\quad\quad\quad\quad\quad\quad\quad \left. +\frac{n(n+1)}{2} \nu_{t,\gamma} \insquare{ x_1^{n+1} x_1^{n+2} u'\inparen{ S_{M+1}^{0,n+1}   } u'\inparen{ S_{M+1}^{0,n+2}  } f  } \!\phantom{\Bigg\rvert}\!\right).
\end{align*}
As before, the $1/N$ in front removes the need for further elaborate computations on each term. We obtain
\begin{align}
    \abs{ \frac{\ud}{\ud \gamma}   \nu^{+,\gamma}_{0}\insquare{f} } &\leq \frac{1}{N} K(n,k,D) \nu^{+,\gamma}_{0}\abs{f}.
    \label{eq:perceptron_|partial_nu^+gamma_0_partialgamma|}
\end{align}

Next---the $v$-interpolation for the $\nu^{+,0}_{0}$ system. Define
\begin{align*}
    S_v^{0,\ell} &:= \sqrt{v}S_{M+1}^{0,\ell} + \sqrt{1-v} \theta^{+,\ell} \\
    S_v^{0,\ell\, \prime} &:= \frac{1}{2\sqrt{v}}S_{M+1}^{0,\ell} - \frac{1}{2\sqrt{1-v}} \theta^{+,\ell},
\end{align*}
and
\begin{align*}
    \nu^{+,0,v}_0\insquare{f} &= \Eb   \frac{ \Eb_\xi \inangle{ f \exp \inparen{\sum_{\ell \leq n} u(S_{v}^{0,\ell}) } }_0 }{ \inparen{ \Eb_\xi \inangle{ \exp u(S_{v}^{0,1}) }_0 }^n }.
\end{align*}
Differentiating, we have
\begin{align*}
    \frac{\ud}{\ud v} \nu^{+,0,v}_{0}\insquare{f} &= \sum_{\ell \leq n} \nu^{+,0,v}_0\insquare{ S_v^{0,\ell\, \prime} u'(S_v^{0,\ell})f  } - n\nu^{+,0,v}_0\insquare{ S_{v}^{0,n+1\, \prime} u'(S_v^{0,n+1})f  }.
\end{align*}
Using the relations 
\begin{align*}
    \Eb_{\textnormal{d}}\insquare{ S_v^{0,\ell\, \prime} S_v^{0,\ell}  } = p(N,k,0); \quad \textnormal{ and } \quad \ell \neq \ell' \Rightarrow \Eb_{\textnormal{d}}\insquare{ S_v^{0,\ell\,\prime} S_v^{0,\ell'}  } = \frac{1}{2}\inparen{R^{0}_{\ell,\ell'} - q^+},
\end{align*}
we integrate by parts as in the proof of Lemma \ref{lemma:perceptron_vInterpolation}. Compared to \eqref{eq:perceptron_partial_nu^+v_partialv}, there is an additional type of term $p(N,k,0) \nu^{+,0,v}_0\insquare{fA}$, which is easily dealt with since $\abs{p(N,k,0)} \leq k/N$. We thus obtain
\begin{align}
    \abs{\frac{\ud}{\ud v}\nu^{+,0,v}_0\insquare{f}} \leq K(n,k,D)\insquare{ \inparen{ \nu^{+,0,v}_{0}\abs{f}^{\tau_1} }^{1/\tau_1} \inparen{ \nu^{+,0,v}_{0}\abs{R_{1,2} - q^+}^{\tau_2} }^{1/\tau_2}  + \frac{1}{N}\nu^{+,0,v}_{0}\abs{f} }.
    \label{eq:perceptron_|partial_nu^+0v_0_partialv|}
\end{align}
We now put the $\gamma$ and $v$-interpolations together. Observe that, at the $\gamma = 0$, $v = 0$ endpoint,
\begin{align*}
    \nu^{+,0,0}_{0} \insquare{f} &= \Eb   \frac{ \Eb_\xi \inangle{ f \exp \inparen{\sum_{\ell \leq n} u(\theta^{+,\ell}) } }_0 }{ \inparen{ \Eb_\xi \inangle{ \exp u(\theta^{+,1}) }_0 }^n } = \nu_0 \insquare{f}
\end{align*}
Therefore,
\begin{align*}
    \abs{\nu^+_{0}\insquare{f} - \nu_0\insquare{f}} &\leq  \abs{\nu^{+}_{0}\insquare{f} - \nu^{+,0}_0\insquare{f}} + \abs{\nu^{+,0}_{0}\insquare{f} - \nu^{+,0,0}_0\insquare{f}} \leq \sup_{0 \leq \gamma \leq 1} \abs{ \frac{\ud}{\ud \gamma}   \nu^{+,\gamma}_{0}\insquare{f} } + \sup_{0 \leq v \leq 1} \abs{\frac{\ud}{\ud v}\nu^{+,0,v}_0\insquare{f}}.
\end{align*}
Extend \eqref{eq:perceptron_getRidOfv} to this setting to obtain, for $\bar{f} \geq 0$, that $\nu^{+,0,v}_0 \bar{f} \leq K(n,k,D) \nu^{+,0}_0 \bar{f}$; extend \eqref{eq:perceptron_getRidOfgamma} as well to get that $\nu^{+,0}_0\bar{f} \leq K(n,k,D) \nu^{+}_0 \bar{f}$; finally, extend \eqref{eq:perceptron_getRidOft} to get that there exists a constant $K^+(n,k,D)$ such that whenever $\alpha^+ K^+ \leq 1$ then $\nu^+_0 \bar{f} \leq 2\nu^+ \bar{f}$. Together with \eqref{eq:perceptron_|partial_nu^+gamma_0_partialgamma|} and \eqref{eq:perceptron_|partial_nu^+0v_0_partialv|}, this establishes \eqref{eq:perceptron_nu^+_0_to_nu_0}.
\end{proofof}

\begin{proofof}{Theorem \ref{thm:perceptron_TSOC_auxiliarySystem}}
The proof follows the strategy of Talagrand \cite{talagrand2010mean} Lemma 2.4.3. Define
\begin{align*}
    f &:= \frac{1}{N-k} \sum_{m \leq M} \tilde{u}^{\,\prime 2} \inparen{ S_m^{-}  } - \tau^-; \quad \textnormal{ and } \quad f^- := \frac{1}{N-k} \sum_{m \leq M-1} \tilde{u}^{\,\prime 2} \inparen{ S_m^{-}  } - \tau^-,
\end{align*}
and the important point is that $f^-$ is independent of the randomness in $(g_{i,M})_{k+1 \leq i \leq N}$. By symmetry among the $S_{m}^{-}$'s, 
\begin{align}
    \nu^-\insquare{f^2} &= \nu^-\insquare{ \inparen{ \alpha^- \tilde{u}^{\,\prime 2}(S_M^{-}) - \tau^-  } f  } \leq \nu^-\insquare{ \inparen{ \alpha^- \tilde{u}^{\,\prime 2}(S_M^{-}) - \tau^-  } f^-  } + \frac{K(k,D)}{N-k},
    \label{eq:perceptron_u(tilde)'_TS_f^2}
\end{align}
where in the inequality we used the fact that there exists a constant $K^-$ for which $\alpha^- K^- \leq 1$. 

In what follows, we have to extend several results of Section \ref{sec:perceptron_decomposeHamiltonian} to the $\nu^-$ setting. This brings no additional difficulty, we simply replace any $\nu$ by $\nu^-$, and all $\alpha, q, r, u$ by $\alpha^-$, $q^-$, $r^-$, $\tilde{u}$ respectively, and any $N$ by $N-k$. 

The following argument is very similar in spirit to that of \eqref{eq:perceptron_nu^+_to_nu}, in the sense that we do not need a $t$ or $\gamma$ interpolation, only a $v$ interpolation. For some function $\phi$ independent of the randomness in $(g_{i,M})_{i \leq M}$, define 
\begin{align*}
    \nu^{-,v} \insquare{ \tilde{u}^{\,\prime 2}(S_v^{-}) \phi } &:= \Eb \frac{\Eb_\xi \inangle{ \tilde{u}^{\,\prime 2}(S_v^{-}) \phi \exp \tilde{u}(S_v^{-}) }^{-}_{\sim} }{\Eb_\xi \inangle{  \exp \tilde{u}(S_v^{-}) }^{-}_{\sim}},
\end{align*}
where $S_v^- := \sqrt{v}S_{M}^- + \sqrt{1-v}\theta^-$.
Note that from the definition \eqref{eq:perceptron_tau^-},
\begin{align*}
    \nu^{-,0}\insquare{ \alpha^- \tilde{u}^{\,\prime 2}(S_0^{-}) f^- } &= \tau^- \nu^{-}\insquare{ f^- }.
\end{align*}
Therefore, in order to show that the first term on the RHS of \eqref{eq:perceptron_u(tilde)'_TS_f^2} is small, we need only show that $\partial \nu^{-,v} \insquare{ \alpha^-\tilde{u}^{\,\prime 2}(S_v^{-}) f } / \partial v$ is bounded. However, this follows by the standard $v$-interpolation argument, as in \eqref{eq:perceptron_partial_nu^+v_partialv}, and using the relations, 
\begin{align*}
    \Eb_{\textnormal{d}}\insquare{ S_v^{-,\ell\, \prime} S_v^{-\ell}  } = 0; \quad \textnormal{ and } \quad \ell \neq \ell' \Rightarrow \Eb_{\textnormal{d}}\insquare{ S_v^{-,\ell\,\prime} S_v^{-,\ell'}  } = \frac{1}{2}\inparen{R^{-}_{\ell,\ell'} - q^-},
\end{align*}
and later removing the dependence in $v$ in the RHS by extending \eqref{eq:perceptron_getRidOfv}. Thus, we obtain that for some $K^*(k,D)$,
\begin{align*}
    \abs{ \nu^-\insquare{ \inparen{ \alpha^- \tilde{u}^{\,\prime 2}(S_M^{-}) - \tau^-  } f  }  } &\leq \alpha^-K^*(k,D) \insquare{ \inparen{\nu^-\abs{R_{1,2}^- - q^-}^2}^{1/2} \inparen{ \nu^-\abs{f^-}^2 }^{1/2}  + \frac{1}{N-k} \nu^{-}\abs{f^-}}.
\end{align*}
Substituting the above into \eqref{eq:perceptron_u(tilde)'_TS_f^2} and using that $\sqrt{ab} \leq \frac{1}{2}a + \frac{1}{2}b$ and that $\abs{f^-} \leq \alpha^- K(D)$, we have that whenever $\alpha^- K^* \leq 1$, 
\begin{align}
    \nu^{-}\insquare{f^2} \leq \frac{1}{2} \nu^-\abs{R_{1,2}^- - q^-}^2 + \frac{1}{2} \nu^-\abs{f^-}^2  + \frac{K(k,D)}{N-k}.
    \label{perceptron_TSOC_auxiliary_contractionPart}
\end{align}
It is easily seen that $\abs{f^2 - f^{-\, 2}} \leq K(D)/(N-k)$, and this yields
\begin{align*}
    \nu^{-}\insquare{f^2} \leq \frac{1}{2} \nu^-\abs{R_{1,2}^- - q^-}^2 + \frac{1}{2} \nu^-\insquare{f^2}  + \frac{K(k,D)}{N-k},
\end{align*}
and the result follows upon rearranging and using \eqref{eq:perceptron_R12_OC_-}.

Note that the same argument will also work for \eqref{eq:perceptron_u(tilde)'(S_m^-)_OC}, simply by changing the definition of $f$ appropriately.
\end{proofof}

\section{Supplementary proofs for Section \ref{sec:Shcherbina-Tirozzi}}
\label{sec:suppProofs_for_ST}

\subsection{ Proofs of Theorem \ref{thm:ST_TSOC_auxiliary} }

Preparing for the proof of Theorem \ref{thm:ST_TSOC_auxiliary}, some notation is introduced. Let $z$, $\tilde{z}$, $(\xi^\ell)_{\ell \leq n}$, $(\tilde{\xi}_\ell)_{\ell \leq n}$ be independent standard Gaussians, and define for $\ell \leq n$,
\begin{align*}
    \eta^{\ell} := \sqrt{q}z + \sqrt{\rho - q} \xi^\ell; \quad\quad 
    \tilde{\eta}^{\ell} := \sqrt{q}\tilde{z} + \sqrt{\rho - q} \tilde{\xi}^{\ell},
\end{align*}
where $q $ and $\rho$ are given as in \eqref{eq:ST_RS_rho_q}. We have the identity
\begin{align}
    \tau^2 &= \alpha^2 \Eb \frac{ \Eb_\xi \insquare{  u^{\prime 2}(\eta) u^{\prime 2}(\tilde{\eta}) \exp\inparen{ u(\eta) + u(\tilde{\eta})   }  } }{ \Eb_\xi \exp\inparen{ u(\eta) + u(\tilde{\eta}) } }.
    \label{eq:ST_tau^2_identity}
\end{align}
Let us emphasize that we should \textbf{not} think of $\eta$ and $\tilde{\eta}$ as replicas (even as they are independent). Indeed, $\cL\inparen{ \eta, \tilde{\eta} \, | \, z, \tilde{z} }$ arises as the limit of $\cL\inparen{ \vec{A}_{M}^{\top} \vec{x}, \vec{A}_{M-1}^{\top} \vec{x} \, | \, \vec{A}_{M}, \vec{A}_{M-1}, \textnormal{disorder} }$---the independence of the projection directions is the reason for $z$ and $\tilde{z}$. Even though we are projecting the same $\vec{x}$, the independent copies $\xi$ and $\tilde{\xi}$ come about from the intuition given by the projection result Theorem \ref{thm:ProjResult_DisorderedCase_supOverBL_LM_2p}, which stipulates the independence of $\vec{A}_{M}^{\top} \vec{x}$ and $\vec{A}_{M-1}^{\top} \vec{x}$ as $N$ grows large, conditioned on the disorder. (Note however, that $\eta^1, \eta^2, \dots$ or $\tilde{\eta}^1, \tilde{\eta}^2,\dots$ are indeed replicas).

The idea for the proof is to create a cavity-in-$(M, M-1)$. That is, consider the Hamiltonian on $\RR^N$:
\begin{align*}
    -H_{N,M-2}(\vec{x}) &= \sum_{m \leq M-2} u\inparen{ \frac{1}{\sqrt{N}}\sum_{1 \leq i \leq N} g_{i,m} x_i   } - \kappa\norm{\vec{x}}^2 + h\sum_{i \leq N} g_i x_i 
\end{align*}
with associated Gibbs expectations denoted by $\inangle{\cdot}_{1,\sim\sim}$ (remember one `$\sim$' was for a cavity-in-$M$). The subscript $1$ is to not conflict with the $t$-interpolations. Let $0 \leq s \leq 1$. Define
\begin{align}
    S_{s}^{\ell} &:= \sqrt{s}S_{M}^{\ell} + \sqrt{1 - s} \eta^{\ell}; \quad\quad
    \tilde{S}_{s}^{\ell} := \sqrt{s}S_{M-1}^{\ell} + \sqrt{1 - s} \tilde{\eta}^{\ell}.
\end{align}
Observe the relation $-H_{N,M} = -H_{N,M-2,} + u(S_{M}) + u(S_{M-1})$. Let $B_s$ and $f$ be a functions on $\inparen{ \RR^{N}}^{n}$ ($f$ does not depend on $s$) and it is natural to define
\begin{align}
    \nu_{1,s}\insquare{B_s f} &:= \Eb \frac{\Eb_\xi \inangle{B_s f  \exp\inparen{ \sum_{\ell \leq n} u(S_s^{\ell}) + u(\tilde{S}_s^{\ell})  }  }_{1,\sim\sim} }{ \inparen{ \Eb_\xi \inangle{ \exp\inparen{  u(S_s^{1}) + u(\tilde{S}_s^{1})  }  }_{1,\sim\sim} }^{n}  },
    \label{eq:ST_cavityIn_M,M-1}
\end{align}
so that $\nu_{1,1}\insquare{B_1 f} = \nu_{1}\insquare{B_1 f}$. Observe from \eqref{eq:ST_tau^2_identity} that if we set $B_s = u^{\prime 2}(S_s)u^{\prime 2}(\tilde{S}_s)$ and $f \equiv 1$, we have the relation
\begin{align}
    \alpha^2 \nu_{1,0}\insquare{ u^{\prime 2}(S_0)u^{\prime 2}(\tilde{S}_0)  } = \alpha^2 \nu_{1,0}\insquare{ u^{\prime 2}(\eta)u^{\prime 2}(\tilde{\eta})  } = \tau^2.
    \label{eq:ST_nu_1,0_identityforTau^2}
\end{align}
The strategy for showing thin-shell for the auxiliary system $\inparen{ u'(S_m) }$ is then clear: show that $\abs{ \partial \nu_{1,s}\insquare{B_s f}/\partial s }$ is small.

Preparing for this, we need to record some useful facts about the implied Gibbs measure associated to $\nu_{1,s}$: $\inangle{\cdot}_{1,s}$. Remark that many of these useful properties mirrors those obtained by Talagrand \cite{talagrand2010mean} Section 3.2, just that here we have two cavities in $M$ and $M-1$. The Hamiltonian associated to $\inangle{\cdot}_{1,s}$ can be written as
\begin{align}
    -H_{N,M,1,s}(\vec{x}) &= \sum_{m \leq M-2} u\inparen{ S_{m}(\vec{x}) } + u_s\inparen{  \sqrt{s} S_{M}(\vec{x}) + \sqrt{1-s}\sqrt{q} z    }  \\
    &\quad\quad\quad\quad\quad\quad\quad + u^{\sim}_s\inparen{  \sqrt{s} S_{M-1}(\vec{x}) + \sqrt{1-s}\sqrt{q} \tilde{z}    } - \kappa \norm{\vec{x}}^2  + h\sum_{i \leq N} g_i x_i,
    \label{eq:ST_Hamiltonian_H_N,M,1,s}
\end{align}
where $u_s(X) := \log \Eb_\xi \exp u\inparen{ X + \sqrt{1-s}\sqrt{\rho - q} \xi  }$, and $u^{\sim}_s(X) := \log \Eb_\xi \exp u\inparen{ X + \sqrt{1-s}\sqrt{\rho - q} \tilde{\xi}  }$.

\begin{lemma}
The Hamiltonian $-H_{N,M,1,s}$ is $(2\kappa)$-strongly-concave (as in \eqref{eq:-H_N,M_strongConcavity}).
\label{lemma:ST_H_N,M,tlim,s_stronglyConcave}
\end{lemma}

\begin{proof}
It suffices to show that $u_s$ is concave, then the term $-\kappa \norm{\vec{x}}^2$ will guarantee that \eqref{eq:-H_N,M_strongConcavity} holds for $-H_{N,M,1,s}$. Observe that $\exp u_s(X)$ is the $X$-marginal of the log-concave function 
\begin{align*}
    (X, y) \mapsto \frac{1}{\sqrt{2\pi}} \exp \inparen{u\inparen{X + \sqrt{1-s}\sqrt{\rho - q}y} - \frac{y^2}{2}},
\end{align*}
and the result follows as a consequence of Pr{\'e}kopa-Leindler that states that marginals of log-concave functions are also log-concave. Of course the same holds for $u^{\sim}_s$.
\end{proof}

As a consequence of the strong-concavity in $\nu_{1,s}$ most of the helpful results that hold for $\nu$ will extend straightforwardly to this case---exponential integrability of $x_1^2$; that $\nu_{1,s}$ mostly lives on the centered sphere of radius $K\sqrt{N}$; and the strong concentration and boundedness in annealed expectation of $R_{1,1}$ and $R_{1,2}$. Most of these have been established by Talagrand, the only difference here is our two cavities in $M$ and $M-1$, and that $\rho$ and $q$ are the limiting objects satisfying \eqref{eq:ST_RS_rho_q}, not the $\rho_{N,M}, q_{N,M}$ as given by the system. However, it is obvious from Theorem \ref{thm:ST_TSOC_mainSystem} and Lemma \ref{lemma:ST_H_N,M,tlim,s_stronglyConcave} that the results carry forward to this setting.

\begin{lemma}[Essentially \cite{talagrand2010mean} Equations (3.100), (3.102), (3.97), (3.98), and (3.104)] In the following, there are constants $K$ depending on $D, \kappa_0, h_0$ (that possibly changes from line to line) such that
\begin{align}
    \nu_{1,s}\insquare{\exp \frac{x_1^2}{K}} \leq K; \, \quad \textnormal{and }\quad \nu_{1,s}\insquare{\exp \frac{\norm{\vec{x}^2}}{K} } \leq \exp LN,
    \label{eq:ST_nu_tlim,s_exponentialInteg_livingOnSphere}
\end{align}
where $L$ is a universal constant, and
\begin{align}
    \nu_{1,s}\insquare{ \rchi_{\inbraces{ \abs{R_{1,1} - \rho} \geq K }} } \leq \exp(-4N); \quad \nu_{1,s}\insquare{ \rchi_{\inbraces{ \abs{R_{1,2} - q} \geq K  }} } \leq \exp(-4N),
    \label{eq:ST_nu_tlim,s_TSOC_mainSystem}
\end{align}
and
\begin{align}
    \nu_{1,s}\insquare{ \inparen{R_{1,1} - \rho}^8 } \leq K; \quad \nu_{1,s}\insquare{ \inparen{R_{1,2} - q}^8 } \leq K.
    \label{eq:ST_nu_tlim,s_R11-rho,R12-q_BDD}
\end{align}
\label{lemma:ST_nicePropertiesOf_nu_1s}
\end{lemma}

\begin{lemma}
Let $f$ be a function on $\inparen{\RR^N}^n$, independent of $g_{i,M}$'s, $g_{i,M-1}$'s, $z_j$'s, $\xi_j^{\ell}$'s.
For $B_s$ equal to one of $\, 1$, or $\, u^{\prime 2} \inparen{ S^{1}_{s}  } u^{\prime 2}( \tilde{S}^1_{s}  )$ or $\, u'(S_s^{1}) u'(S_s^{2}) u'(\tilde{S}_s^{1}) u'(\tilde{S}_s^{2})$, then
\begin{align}
    \abs{ \frac{\partial}{\partial s} \nu_{1, s}\insquare{B_s f} } \leq K\inparen{ \sum_{\ell \leq n+1} \nu_{1,s}\insquare{\abs{f}\abs{R_{\ell, \ell} - \rho}}  + \sum_{\ell < \ell' \leq n+2} \nu_{1,s}\insquare{\abs{f}\abs{R_{\ell, \ell'} - q}}    },
    \label{eq:ST_s_interpolation_statement}
\end{align}
where $K = K(n,D, \kappa_0, h_0) > 0$ is some constant.
\label{lemma:ST_s_interpolation}
\end{lemma}

\begin{proof}
The proof is almost identical to the $v$-interpolation in Lemma \ref{lemma:perceptron_vInterpolation}, therefore we will be sparse with the details. Here we define $S_{s}^{\ell \, \prime} := \frac{1}{2\sqrt{s}}S_{M}^{\ell} - \frac{1}{2\sqrt{1-s}} \eta^\ell $, (similar for $\tilde{S}_{S}^{\ell \, \prime}$) and we have the relations
\begin{align}
    \Eb S_s^{\prime} \tilde{S}_s = 0 \,\textnormal{ (any replica indices)};\; \quad \Eb S_{s}^{\ell\,\prime} S_{s}^{\ell} = \frac{1}{2}\inparen{ R_{\ell, \ell} - \rho  };\; \quad \ell \neq \ell' \Rightarrow \Eb S_{s}^{\ell\,\prime} S_{s}^{\ell'} = \frac{1}{2}\inparen{ R_{\ell, \ell'} - q  },
    \label{eq:ST_s-interpolation_correlationRelations}
\end{align}
with analogous relations when $S_{s}^{\ell \, \prime}$ is swapped for $\tilde{S}_{S}^{\ell \, \prime}$. The same technique in Lemma \ref{lemma:perceptron_vInterpolation} will then yield that the expression for $\partial \nu_{1, s} \insquare{B_s f} / \partial s$ will be a sum of two types of terms: $\nu_{1, s}\insquare{f(R_{\ell,\ell} - \rho)A}$ and $\nu_{1, s}\insquare{f(R_{\ell,\ell'} - q)A}$, where $A$ is some monomial in $u'(S^{\ell}_s)$ or $u'(\tilde{S}^{\ell}_s)$ or their derivatives. 
\end{proof}

\begin{lemma}
Let $f^*$ be a function on $\inparen{ \RR^{N} }^n$ with $f^* \geq 0$, then
\begin{align}
     \nu_{1, s} \insquare{f^*} \leq K \nu\insquare{ f^* } +  K \exp(-N) \sup_{0 \leq s \leq 1} \inparen{\nu_{1,s} \insquare{f^2}}^{1/2}
\end{align}
\label{lemma:ST_getRidOf_s}
\end{lemma}

\begin{proof}
The idea is to use the fact that under $\nu_{1,s}$, for any $s$, $R_{\ell, \ell} - \rho$ and $R_{\ell, \ell'} - q$ are small except on a set of exponentially small measure. This allows the RHS of \eqref{eq:ST_s_interpolation_statement} to be bounded without $R_{\ell, \ell} - \rho$ or $R_{\ell, \ell'} - q$, then we can integrate the differential inequality.

Consider a generic summand on the RHS of \eqref{eq:ST_s_interpolation_statement}, we have
\begin{align*}
    \nu_{1,s}\insquare{ \abs{f} \abs{R_{\ell, \ell} - \rho}  } &\leq K\nu_{1,s}\abs{f} + \inparen{\nu_{1,s}\insquare{f^2}}^{1/2} \inparen{\nu_{1,s}\insquare{ (R_{1,1} - \rho)^4 }}^{1/4} \inparen{\nu_{1,s}\insquare{ \rchi_{\inbraces{ \abs{R_{1,1} - \rho}  \geq K } } }}^{1/4} \\
    &\leq K\nu_{1,s}\abs{f} + K \exp(-N) \inparen{\nu_{1,s}\insquare{f^2}}^{1/2}
\end{align*}
where in the first inequality we have split the integral into two, and applied Cauchy-Schwarz; the second inequality follows from \eqref{eq:ST_nu_tlim,s_TSOC_mainSystem} and \eqref{eq:ST_nu_tlim,s_R11-rho,R12-q_BDD}. The same argument applies for summands with $R_{\ell,\ell'} - q$. It follows that we have the differential inequality
\begin{align}
    \abs{\frac{\partial}{\partial s} \nu_{1, s} \insquare{f} } \leq K \nu_{1,s}\abs{f} +  K\exp(-N) \sup_s \inparen{\nu_{1,s} \insquare{f^2}}^{1/2},
\end{align}
and if $f \geq 0$, we can integrate using a quantitative variant of Gronwall's lemma (see \cite{talagrand2010mean} appendix A.13.1).
\end{proof}

\begin{proofof}{Theorem \ref{thm:ST_TSOC_auxiliary}}
We start with \eqref{eq:ST_TS_auxiliarySystem} first. Expand and use the symmetry among each $m \leq M$ to get
\begin{align}
    \nu \insquare{ \inparen{  \frac{1}{N} \sum_{m \leq M} u^{\prime 2} (S_m) - \tau }^{2}  } &= \nu \insquare{ \inparen{  \frac{1}{N} \sum_{m \leq M} u^{\prime 2} (S_m) }^{2} - \tau^2  } - 2\tau \nu \insquare{ \ \frac{1}{N} \sum_{m \leq M} u^{\prime 2} (S_m) - \tau   } \nonumber \\
    &= \nu\insquare{ \alpha^2 u^{\prime 2}(S_M) u^{\prime 2}(S_{M-1} ) - \tau^2 } - 2\tau\nu\insquare{ \alpha u^{\prime 2}(S_M) - \tau  }\nonumber\\
    &\quad\quad\quad + \frac{M}{N^2} \nu \insquare{  u^{\prime 4}(S_M)  } - \frac{M}{N^2} \nu\insquare{ u^{\prime 2}(S_M) u^{\prime 2}(S_{M-1})  } \nonumber \\
    &\leq \abs{\alpha^2 \nu\insquare{ u^{\prime 2}(S_M) u^{\prime 2}(S_{M-1} ) } - \tau^2  } + K(D)  \abs{ \alpha \nu\insquare{ u^{\prime 2}(S_M)  }  - \tau  } + \frac{K(D)}{N},
    \label{eq:ST_TS_auxiliary_expansion}
\end{align}
using that $M/N \leq L$ as in \ref{eq:ST_other_asmpts}. It will next be shown that $\alpha^2 \nu\insquare{ u^{\prime 2}(S_M) u^{\prime 2}(S_{M-1} ) } \simeq \tau^2$. With $B_s = u^{\prime 2}(S_s)u^{\prime 2}(\tilde{S}_s)$ and $f \equiv 1$, one checks from the definition of $\tau^2$ that (as also noted in \eqref{eq:ST_tau^2_identity}), 
\begin{align}
    \abs{\alpha^2 \nu\insquare{ u^{\prime 2}(S_M) u^{\prime 2}(S_{M-1} ) } - \tau^2  } = \alpha^2\abs{ \nu_{1,1}\insquare{ B_1 f  }  - \nu_{1,0}\insquare{ B_0 f  }  }.
    \label{eq:ST_TS_auxiliary_expansion_FIRSTTERM}
\end{align}
If we use a bound on $\abs{ \partial \nu_{1,s} \insquare{B_s f}/\partial s }$, we will be able to prove \eqref{eq:ST_TS_auxiliarySystem} at a rate $K/\sqrt{N}$. In order to attain a better rate, we need to use a second order estimate: Taylor's remainder theorem (integral form) gives
\begin{align}
    \abs{ \nu_{1,1} \insquare{B_1 f} - \nu_{1,0}\insquare{B_0 f} - \frac{\partial}{\partial s} \nu_{1,s}\insquare{B_s f} \bigg\rvert_{s = 0}   } \leq \sup_{0 \leq s \leq 1} \abs{ \frac{\partial^2}{\partial s^2} \nu_{1,s}\insquare{B_s f}  }.
    \label{eq:ST_TS_auxiliary_secondOrderEstimation}
\end{align}
Fortunately, the only difficulty associated with the second derivative is tedious computation. Rather than write explicitly all the details, we will reference the computations in the proof in Lemma \ref{lemma:perceptron_vInterpolation}. We first differentiate $\nu_{1,s}\insquare{B_s f}$ once in $s$, and integrate by parts, using the correlation relations exactly as in \eqref{eq:ST_s-interpolation_correlationRelations}, this will yield a sum of terms of the type $\nu_{1,s}\insquare{ f (R_{\ell,\ell} - \rho) A_s   }$ or $\nu_{1,s}\insquare{ f (R_{\ell,\ell'} - q) A_s   }$, where $A_s$ is some monomial in $u'(S^{\ell}_s)$ or $u'(\tilde{S}^{\ell}_s)$ or their derivatives, as in \eqref{eq:perceptron_vInterpolation_u'(S_v^1)u'(S_v^2)f_FIRST_TERM}. Now when we differentiate again (that is, differentiate $\partial \nu_{1,s} \insquare{B_s f}/\partial s$), we repeat the program for such terms $\nu_{1,s}\insquare{ A_s f  (R_{\ell,\ell} - \rho)   }$, this time with $A_s$ instead $B_s$, and $f(R_{\ell, \ell} - \rho)$ instead of $f$. The extra factor $R_{\ell, \ell} - \rho$ is the ace in the hole, because another factor $R_{\ell, \ell} - \rho$ will be produced in the differentiation and Gaussian integration by parts process (the same holds for terms with $R_{\ell, \ell'} - q$). This eventually gives
\begin{align}
    \abs{ \frac{\partial^2}{\partial s^2} \nu_{1,s}\insquare{B_s f} } &\leq K\inparen{ \nu_{1,s}\insquare{ (R_{1,1} - \rho)^2  } + \nu_{1,s}\insquare{ (R_{1,2} - q)^2  }  } \nonumber\\
    &\leq K\inparen{   \nu\insquare{ \inparen{R_{1,1} - \rho}^2  }  + \nu\insquare{ \inparen{R_{1,2} - q}^2  }  + \frac{1}{N} } \leq \frac{K}{N},
    \label{eq:ST_TS_auxiliary_SecondOrderBound}
\end{align}
where $K = K(D, \kappa_0, h_0)$ is some constant, and where in the second inequality we have used Lemma \ref{lemma:ST_getRidOf_s} to remove $s$ on the RHS, and \eqref{eq:ST_nu_tlim,s_R11-rho,R12-q_BDD} to bound $\sup_s \nu_{1,s} \insquare{\inparen{ R_{1,1} - \rho} } \leq K$ and also $\sup_s \nu_{1,s} \insquare{\inparen{ R_{1,2} - q} } \leq K$, and the last inequality follows from the thin-shell and overlap concentration for the ST main system Theorem \ref{thm:ST_TSOC_mainSystem}. This controls the RHS of \eqref{eq:ST_TS_auxiliary_secondOrderEstimation}. It remains to show 
\begin{align}
    \abs{\frac{\partial}{\partial s} \nu_{1,s}\insquare{B_s f} \bigg\rvert_{s = 0}} \leq \frac{K}{N}
    \label{eq:ST_TS_auxiliary_ShowFirstOrderTermSmall}
\end{align}
Differentiating $\nu_{1,s}\insquare{B_s f}$, now with $s = 0$, reveals that it is the sum of terms of the form $A \Eb \inangle{f(R_{1,1} - \rho)}_{1,\sim\sim} = A \nu_{1,0} \insquare{f(R_{1,1} - \rho)} $ or  $A \Eb \inangle{f(R_{1,2} - q)}_{1,\sim\sim} = A \nu_{1,0} \insquare{f(R_{1,2} - q)} $, where $A$ is the expectation of some monomial in $u'(\eta^{\ell})$ or $u'(\tilde{\eta}^{\ell})$ or their derivatives, which is bounded by $K(D)$. Write
\begin{align*}
    A \nu_{1,0} \insquare{ f(R_{1,1}- \rho)  } \leq K \abs{ \nu\insquare{f(R_{1,1} - \rho)} }  +  K \abs{  \nu\insquare{f(R_{1,1} - \rho)} - \nu_{1,0}\insquare{f(R_{1,1} - \rho)}    }.
\end{align*}
The first term on the RHS is small since with $f = 1$, symmetry among sites gives 
\begin{align*}
    \abs{ \nu\insquare{ R_{1,1} - \rho  } } \leq \abs{ \nu\insquare{x_1^2} - \nu_{0\textnormal{-lim}}\insquare{x_1^2} } \leq \frac{K}{N}
\end{align*}
by \cite{talagrand2010mean} Theorem 3.2.14 (the same was used in the proof of Theorem \ref{thm:ST_TSOC_mainSystem}). The second term is small and this is seen by applying Lemma \ref{lemma:ST_s_interpolation} with $B_s \equiv 1$ and $f = R_{1,1} - \rho$, yielding
\begin{align*}
    \abs{  \nu\insquare{R_{1,1} - \rho} - \nu_{1,0}\insquare{R_{1,1} - \rho}    } \leq K\inparen{ \nu_{1,s}\insquare{ (R_{1,1} - \rho)^2  } + \nu_{1,s}\insquare{ (R_{1,2} - q)^2  }  },
\end{align*}
and we proceed as exactly as in \eqref{eq:ST_TS_auxiliary_SecondOrderBound} so that \eqref{eq:ST_TS_auxiliary_ShowFirstOrderTermSmall} holds. The same story holds for terms with $R_{1,2} - q$. From \eqref{eq:ST_TS_auxiliary_expansion_FIRSTTERM}, we therefore have $\abs{\alpha^2 \nu\insquare{ u^{\prime 2}(S_M) u^{\prime 2}(S_{M-1} ) } - \tau^2  } \leq \frac{K}{N}$.

For the second term in \eqref{eq:ST_TS_auxiliary_expansion}, an almost identical argument will work, just that in \eqref{eq:ST_cavityIn_M,M-1} we have to create instead a cavity-in-$M$, which proceeds exactly as in the $v$-interpolation in Lemma \ref{lemma:perceptron_vInterpolation}. The situation is even simpler in this case and the details are omitted, but the end result is also that $\abs{ \alpha \nu\insquare{ u^{\prime 2}(S_M)  }  - \tau  } \leq K/N$. This finishes the proof of \eqref{eq:ST_TS_auxiliarySystem}.

Expand as before in \eqref{eq:ST_TS_auxiliary_expansion} to get
\begin{align}
    \nu \insquare{ \inparen{  \frac{1}{N} \sum_{m \leq M} u^{\prime} (S^{1}_m) u^{\prime} (S^{2}_m) - r }^{2}  } &\leq  \abs{\alpha^2 \nu\insquare{ u^{\prime }(S^1_M) u^{\prime }(S^1_{M-1} ) u^{\prime }(S^2_M) u^{\prime }(S^2_{M-1} ) } - r^2  } \nonumber\\
    &\quad\quad\quad + K(D)  \abs{ \alpha \nu\insquare{ u^{\prime 2}(S^1_M) u^{\prime 2}(S^2_M)  }  - r  } + \frac{K(D)}{N}.
    \label{eq:ST_OC_auxiliary_expansion}
\end{align}
For the first term in \eqref{eq:ST_OC_auxiliary_expansion}, note that
\begin{align}
    r^2 &= \alpha^2 \Eb \frac{ \Eb_\xi \insquare{  u^{\prime}(\eta^1) u^{\prime}(\eta^2) u^{\prime}(\tilde{\eta}^1) u^{\prime}(\tilde{\eta}^2) \exp\inparen{ \sum_{\ell \leq 2} u(\eta^\ell) + u(\tilde{\eta}^{\ell})   }  } }{ \Eb_\xi \exp\inparen{ \sum_{\ell \leq 2} u(\eta^\ell) + u(\tilde{\eta}^{\ell}) } },
    \label{eq:ST_r^2_identity}
\end{align}
so that we should set $B_s = u'(S_s^{1}) u'(S_s^{2}) u'(\tilde{S}_s^{1}) u'(\tilde{S}_s^{2})$ and $f \equiv 1$ (here $n=2$). We repeat the process described above starting from \eqref{eq:ST_TS_auxiliary_expansion_FIRSTTERM} with this new $B_s$ and $f$; the argument is exactly the same. This finishes the proof of \eqref{eq:ST_OC_auxiliarySystem}.
\end{proofof}

\subsection{Proof of Proposition \ref{proposition:ST_nuf_to_nu0limf}}
\label{sec:ST_nuf_to_nu0limf}

This interpolation path is very similar to that of Talagrand, who showed $\nu f \simeq \nu_{\textnormal{0-lim}} f$, (i) in the case $k = 1$, and (ii) with $\rho$ and $q$ in the quantities \eqref{eq:ST_RS_r_tau_sigma} replaced by $\nu\insquare{R_{1,1}}$ and $\nu\insquare{R_{1,2}}$. However, we still need to verify that the result will extend for general $k \geq 1$ (for instance, there are cross terms involved in the quadratic term in the Taylor expansion, if we repeat for the Shcherbina-Tirozzi model what we did for the Perceptron in \eqref{eq:perceptron_taylorExpansion}). The structure of the argument is very similar to the $t$ and $v$-interpolations in the Perceptron model.

Define the $t\textnormal{-lim}$ interpolating Hamiltonian, for fixed $k$, and for $0 \leq t \leq 1$:
\begin{align}
    -H_{N,M,t\textnormal{-lim}}(\vec{x}) &:= \sum_{m \leq M} u\inparen{S_{m,t}} - \kappa \norm{\vec{x}}^{2} + h\sum_{i \leq N} g_i x_i + \sqrt{1 - t} \sum_{j \leq k} x_j \sqrt{r}z_j - (1-t)(r - \bar{r}) \sum_{j \leq k} \frac{x_j^2}{2},
    \label{eq:ST_Hamiltonian_H_N,M,tlim}
\end{align}
with associated Gibbs measure $\inangle{\cdot}_{t\textnormal{-lim}}$, so that $\nu_{t\textnormal{-lim}}\insquare{\cdot} = \Eb \inangle{\cdot}_{t\textnormal{-lim}}$. The following proposition is almost equivalent to \cite{talagrand2010mean} Proposition 3.2.1, except that we have arbitrary $k \geq 1$.

\begin{proposition}
Let $f$ be a function on $\inparen{ \RR^{N} }^n$, then
\begin{align}
    \frac{\ud}{\ud t} \nu_{t\textnormal{-lim}}\insquare{f} &= \textnormal{I} + \textnormal{II} + \textnormal{III} + \textnormal{IV},
\end{align}
where
\begin{align*}
    \textnormal{I} &:= \frac{\alpha k}{2} \inparen{ \sum_{\ell \leq n} \nu_{t\textnormal{-lim}} \insquare{x_{1}^{\ell, 2} \inparen{ u^{\prime 2}(S_{M,t}^{\ell}) + u''\inparen{ S_{M,t}^{\ell}} } f } -n \nu_{t\textnormal{-lim}}\insquare{ x_{1}^{n+1, 2} \inparen{ u^{\prime 2}(S_{M,t}^{n+1}) + u''\inparen{ S_{M,t}^{n+1}} } f   }  } \\
    \textnormal{II} &:= \alpha k \left( \sum_{\ell < \ell' \leq n }  \nu_{t\textnormal{-lim}}\insquare{ x_1^\ell x_1^{\ell'} u'\inparen{ S_{M,t}^\ell  } u'\inparen{ S_{M,t}^{\ell'} } f  }  -n\sum_{\ell \leq n}   \nu_{t\textnormal{-lim}}\insquare{ x_1^\ell x_1^{n+1} u'\inparen{ S_{M,t}^\ell  } u'\inparen{ S_{M,t}^{n+1} } f  }  \right.\\
    &\quad\quad\quad\quad \left. + \frac{n(n+1)}{2}  \nu_{t\textnormal{-lim}}\insquare{ x_1^{n+1} x_1^{n+2} u'\inparen{ S_{M,t}^{n+1}  } u'\inparen{ S_{M,t}^{n+2} } f  }  \!\phantom{\Bigg\rvert}\! \right) \\
    \textnormal{III} &:=  -rk \left( \sum_{\ell < \ell' \leq n }  \nu_{t\textnormal{-lim}}\insquare{ x_1^\ell x_1^{\ell'} f  }  -n\sum_{\ell \leq n}   \nu_{t\textnormal{-lim}}\insquare{ x_1^\ell x_1^{n+1}  f  }  + \frac{n(n+1)}{2}  \nu_{t\textnormal{-lim}}\insquare{ x_1^{n+1} x_1^{n+2} f  }  \!\phantom{\Bigg\rvert}\! \right) \\
    \textnormal{IV} &:= -\frac{\sigma + \tau}{2} \inparen{ \sum_{\ell \leq n} \nu_{t\textnormal{-lim}}\insquare{x_1^{\ell,2} f}  - n \nu_{t\textnormal{-lim}}\insquare{ x_{1}^{n+1,2} f  } }.
\end{align*}
\label{proposition:ST_tlim_interpolation}
\end{proposition}

\begin{proof}
The proof is identical to Proposition \ref{proposition:perceptron_t_interpolation}.
\end{proof}

As in \eqref{eq:perceptron_S_v^ell_definition}, define for $0 \leq v \leq 1$:
\begin{align*}
    S_{v}^{0,\ell} &:= \sqrt{v} S_{M}^{0,\ell} + \sqrt{1 - v} \theta^\ell;\quad\quad
    S_{v}^{\ell} := \sqrt{v} S_{M,t}^{\ell} + \sqrt{1 - v} \theta^\ell.
\end{align*}
Let $f$ be a function on $\inparen{\RR^{N}}^{n}$; and let $B_v$ be a function that may depend on the randomness in $M$. Define
\begin{align}
    \nu_{t\textnormal{-lim},v}\insquare{B_v f} &:= \Eb \frac{ \Eb_{\xi} \inangle{ B_v f \exp\inparen{ \sum_{\ell \leq n} u\inparen{S_{v}^{\ell}} } }_{t,\sim}  }{ \inparen{ \Eb_{\xi} \inangle{ \exp u\inparen{S_{v}^{1}}   }_{t,\sim}}^{n}  },
\end{align}
where $\Eb_{\xi}$ denotes an expectation over all r.v.'s $\xi^{\ell}$, and where $\Eb$ here is an expectation over all sources of disorder, including $(g_{i,m})$ and the new disorder $z$. Mirroring \eqref{eq:ST_Hamiltonian_H_N,M,1,s}, the Hamiltonian associated with $\inangle{\cdot}_{t\textnormal{-lim},v}$ is 
\begin{align*}
    -H_{N,M,t-\textnormal{lim},v}(\vec{x}) &= \sum_{m \leq M-1} u\inparen{ S_{m,t}(\vec{x})  } + u_{v}\inparen{  \sqrt{v} S_{M,t}(\vec{x}) + \sqrt{1-s}\sqrt{q}z  } \\
    &\quad\quad\quad\quad -\kappa \norm{\vec{x}}^2 + h\sum_{i \leq N} g_i x_i + \sqrt{1 - t} \sum_{j \leq k} x_j \sqrt{r}z_j - (1-t)(r - \bar{r}) \sum_{j \leq k} \frac{x_j^2}{2},
\end{align*}
where $u_v(X) := \log \Eb_\xi \exp u\inparen{ X + \sqrt{1-v}\sqrt{\rho - q}\xi  }$. Similarly to Lemma \ref{lemma:ST_H_N,M,tlim,s_stronglyConcave}, the Hamiltonian $-H_{N,M,t-\textnormal{lim},v}$ is $(2\kappa)$-strongly-concave, and we have the following similarly nice properties (as in Lemma \ref{lemma:ST_nicePropertiesOf_nu_1s}):

\begin{lemma}[Essentially \cite{talagrand2010mean} Equations (3.100), (3.102), (3.97), (3.98), and (3.104)] In the following, there are constants $K$ depending on $D, \kappa_0, h_0$ (that possibly changes from line to line) such that
\begin{align}
    \nu_{t\textnormal{-lim},v}\insquare{\exp \frac{x_1^2}{K}} \leq K; \, \quad \textnormal{and }\quad \nu_{t\textnormal{-lim},v}\insquare{\exp \frac{\norm{\vec{x}^2}}{K} } \leq \exp LN,
    \label{eq:ST_nu_tlim,v_exponentialInteg_livingOnSphere}
\end{align}
where $L$ is a universal constant, and
\begin{align}
    \nu_{t\textnormal{-lim},v}\insquare{ \rchi_{\inbraces{ \abs{R_{1,1} - \rho} \geq K }} } \leq \exp(-4N); \quad \nu_{t\textnormal{-lim},v}\insquare{ \rchi_{\inbraces{ \abs{R_{1,2} - q} \geq K  }} } \leq \exp(-4N),
    \label{eq:ST_nu_tlim,v_TSOC_mainSystem}
\end{align}
and
\begin{align}
    \nu_{t\textnormal{-lim},v}\insquare{ \inparen{R_{1,1} - \rho}^8 } \leq K; \quad \nu_{t\textnormal{-lim},v}\insquare{ \inparen{R_{1,2} - q}^8 } \leq K.
    \label{eq:ST_nu_tlim,v_R11-rho,R12-q_BDD}
\end{align}
\label{lemma:ST_nicePropertiesOf_nu_tlim,v}
\end{lemma}

\begin{lemma}
Let $f$ be a function on $\inparen{\RR^N}^n$, independent of $g_{i,M}$'s, $g_{i,M-1}$'s, $z_j$'s, $\xi_j^{\ell}$'s.
For $B_v$ equal to one of $\, 1$, or $\, u^{\prime} \inparen{ S^{1}_{v}  } u^{\prime }( S^2_{v}  )$ or $\, u''\inparen{ S_v^1} + u^{\prime 2}\inparen{S_v^1}$, then
\begin{align}
    \abs{ \frac{\partial}{\partial v} \nu_{t\textnormal{-lim}, v}\insquare{B_v f} } &\leq K\left( \sum_{\ell \leq n+1} \nu_{t\textnormal{-lim}, v}\insquare{\abs{f}\abs{R_{\ell, \ell} - \rho}} \right. \nonumber\\
    &\quad\quad\quad\quad\quad\quad \left. + \sum_{\ell < \ell' \leq n+2} \nu_{t\textnormal{-lim}, v}\insquare{\abs{f}\abs{R_{\ell, \ell'} - q}}  + \frac{1}{N} \inparen{\nu_{t\textnormal{-lim}, v}\insquare{f^2}}^{1/2}    \right),
    \label{eq:ST_v_interpolation_statement}
\end{align}
where $K = K(n,D, \kappa_0, h_0) > 0$ is some constant.
\label{lemma:ST_v_interpolation}
\end{lemma}

\begin{proof}
The steps are similar to Lemma \ref{lemma:ST_s_interpolation}, except that we arrive at a sum of terms of two types: $\nu_{t\textnormal{-lim}, v}\insquare{f(R^{t}_{\ell,\ell} - \rho)A}$ and $\nu_{t\textnormal{-lim}, v}\insquare{f(R^{t}_{\ell,\ell'} - q)A}$, where $A$ is some monomial in $u'(S^{\ell}_v)$ or its derivatives, and $R^t_{\ell,\ell'}$ is defined as in \eqref{eq:perceptron_R^t_ell,ell'_definition}. The difference between $R_{\ell,\ell'}^t$ and $R_{\ell,\ell'}$ is responsible for the last lower order term. We write 
\begin{align*}
    \nu_{t\textnormal{-lim}, v}\insquare{f(R^{t}_{\ell,\ell'} - q)A} &\leq K(D) \nu_{t\textnormal{-lim}, v}\insquare{\abs{f}\abs{R^{t}_{\ell,\ell'} - q} } \\
    &\leq  K(D) \inparen{ \nu_{t\textnormal{-lim}, v}\insquare{\abs{f}\abs{R_{\ell,\ell'} - q} } + \frac{k}{N} \nu\insquare{ x_1^{\ell,2}  } \inparen{\nu\insquare{ f^2  }}^2  },
\end{align*}
where the last inequality follows from Cauchy-Schwarz. The result follows from the exponential integrability of $x_1^2$ in \eqref{eq:ST_nu_tlim,v_exponentialInteg_livingOnSphere}.
\end{proof}

\begin{lemma}
Let $f^*$ be a function on $\inparen{ \RR^{N} }^n$ with $f^* \geq 0$, then
\begin{align}
     \nu_{t\textnormal{-lim}, v} \insquare{f^*} \leq K \nu\insquare{ f^* } +  \frac{K}{N} \sup_{0 \leq v \leq 1} \inparen{\nu_{t\textnormal{-lim},v} \insquare{f^2}}^{1/2}.
\end{align}
\label{lemma:ST_getRidOf_v}
\end{lemma}

\begin{proof}
The proof is identical to Lemma \ref{lemma:ST_getRidOf_s}.
\end{proof}

\begin{proposition}
Let $f$ be a function on $\inparen{\RR^N}^n$ with $\norm{f}_\infty \leq 1$, then
\begin{align}
    \abs{\frac{\ud}{\ud t} \nu_{t\textnormal{-lim}} \insquare{f} } \leq K \inparen{ \inparen{ \nu_{t\textnormal{-lim}}\insquare{ \inparen{R_{1,1} - \rho}^2  } }^{1/2} + \inparen{ \nu_{t\textnormal{-lim}}\insquare{ \inparen{R_{1,2} - q}^2  } }^{1/2}  + \frac{1}{N} }
\end{align}
\label{proposition:ST_nutlim_intermsOf_nutlim}
\end{proposition}

\begin{proof}
The statement follows from Proposition \ref{proposition:ST_tlim_interpolation}, where terms $\textnormal{I}$ and $\textnormal{IV}$ cancel under the $v$-interpolation \eqref{eq:ST_v_interpolation_statement}, and similarly for $\textnormal{II}$ and $\textnormal{III}$. We then use Lemma \ref{lemma:ST_getRidOf_v} to remove the dependence on $v$ in RHS. This approach is almost identical to that in Lemma \ref{lemma:perceptron_ddt_nu_t_f_inTermsOf_nu_t_f}, except that now the `spin factors' $\abs{x_1^\ell x_1^{\ell'}}$ in the terms in Proposition \ref{proposition:ST_tlim_interpolation} are no longer bounded by 1. Nevertheless, we can apply Cauchy-Schwarz and use the exponential integrability $\nu_{t\textnormal{-lim},v}\insquare{ x_1^{\ell,2 }} \leq K$ as in \eqref{eq:ST_nu_tlim,v_exponentialInteg_livingOnSphere} to get the result.
\end{proof}

\begin{lemma}
We have
\begin{align}
    \nu_{t\textnormal{-lim}}\insquare{ \inparen{R_{1,1} - \rho}^2 } \leq \frac{K}{N}; \quad \nu_{t\textnormal{-lim}}\insquare{ \inparen{R_{1,2} - q}^2 } \leq \frac{K}{N}.
\end{align}
\label{lemma:ST_nutlim_TSOC}
\end{lemma}

\begin{proof}
Let $\rho_{N,M,t\textnormal{-lim}} \equiv \nu_{t\textnormal{-lim}}\insquare{R_{1,1}}$ and $\rho_{N,M} \equiv \nu\insquare{ R_{1,1}  }$. Write
\begin{align*}
    \nu_{t\textnormal{-lim}}\insquare{ \inparen{R_{1,1} - \rho}^2 } \leq K\inparen{ \nu_{t\textnormal{-lim}}\insquare{ \inparen{R_{1,1} - \rho_{N,M,t}}^2 }  + \abs{ \rho_{N,M,t} - \rho_{N,M}  }^2 + \abs{  \rho_{N,M} - \rho }^2   }.
\end{align*}
The first term on RHS is bounded by $K/N$, because the Hamiltonian $-H_{N,M,t\textnormal{-lim}}$ is also $(2\kappa)$-strongly concave, and we can extend the thin-shell (and overlap concentration) of $R_{1,1}$ (and $R_{1,2}$) to the setting of $\nu_{t\textnormal{-lim}}$, i.e., the analog of \cite{talagrand2010mean} Theorem 3.1.18 holds for the $t\textnormal{-lim}$ system. The second term is bounded by $K/N^2$ due to \cite{talagrand2010mean} Lemma 3.2.11. The third term is bounded by $K/N^2$ due to \cite{talagrand2010mean} Theorem 3.2.14. The argument can be repeated for the $R_{1,2} - q$ statement.
\end{proof}

\begin{proofof}{Proposition \ref{proposition:ST_nuf_to_nu0limf}}
The result follows from combining Propositions \ref{proposition:ST_nutlim_intermsOf_nutlim} and Lemma \ref{lemma:ST_nutlim_TSOC}
\end{proofof}


\begin{remark}
Note that it is possible as in the proof of Theorem \ref{thm:ST_TSOC_auxiliary} to use a second order approximation to boost the rate in Corollary \ref{proposition:ST_nuf_to_nu0limf} to $K/N$, but this does not lead to any change in the local independence rate, so we do not proceed with this additional step.
\end{remark}

\subsection{Proof of Proposition \ref{proposition:ST_nu0_to_tildenu0lim}}
\label{sec:ST_nu0_to_tildenu0lim}


Define a slight variant of the $-H_{N,M,0\textnormal{-lim}}$ Hamiltonian:
\begin{align}
    -\tilde{H}_{N,M,0\textnormal{-lim}}(\vec{x}) &= \sum_{m \leq M} u\inparen{ \frac{1}{\sqrt{N}}\sum_{k+1 \leq i \leq N} g_{i,m} x_i   } - \kappa\norm{\vec{x}}^2 + h\sum_{i \leq N} g_i x_i \nonumber\\
    &\quad\quad\quad\quad\quad + \sqrt{r^-}\sum_{j \leq k} x_j z_j  + \sum_{j \leq k} \log \Eb_\xi \insquare{ \exp x_j \sqrt{\tau^- - r^-}\xi_j  } +  \frac{\sigma^-}{2} \sum_{j \leq k} x_j^2,
    \label{eq:ST_tilde_H_N,M,0lim}
\end{align}
with associated annealed expectations $\tilde{\nu}_{0\textnormal{-lim}}$. The only difference with $-H_{N,M,0\textnormal{-lim}}$ is that the quantities $r$,  $\tau$, $\sigma$ are replaced by $r^-$, $\tau^-$, $\sigma^-$. This discrepancy arises from the projection result Theorem \ref{thm:ProjResult_DisorderedCase_supOverBL_LM_2p_replicated_n} that gives convergence under the $\inangle{\cdot}^-$ measure to the quantities that carry a minus superscript.

Define, for $\ell \leq 2p$, $j \leq k$,
\begin{align}
    X_{N,j}^{\ell} := \sum_{m \leq M} \frac{g_{j,m}}{\sqrt{N-k}} \tilde{u}'(S_m^{-,\ell}); \quad \, X_{0,j}^{\ell} := \sqrt{r^-}z_j + \sqrt{\tau^- - r^-} \xi_{j}^{\ell},
    \label{ST:def_X_Nj^ell,X_0j^ell}
\end{align}
and let $\gamma$ denote Gaussian measure on $\RR$ with mean-zero and variance $1/(2\kappa - \sigma^-)$.

\begin{lemma}
Let $B \subseteq \RR^k$ be measurable. Then for any $f$ defined as in \eqref{eq:ST_f_cylinderIndicator},  
\begin{align}
    \nu_{0}\insquare{f} = \Eb \frac{D'}{E'} ;\quad\quad \tilde{\nu}_{0\textnormal{-lim}}\insquare{f} = \Eb \frac{D}{E},
    \label{eq:ST_nu0lim_D/E}
\end{align}
where 
\begin{align*}
    D' &= \int_B\cdots \int_B \inangle{ \exp\inparen{ \sum_{\ell \leq n} \sum_{j \leq k} x_{j}^{\ell} \inparen{ X_{N,j}^{\ell} + h g_j}  }  }^- \prod_{\ell \leq n} \gamma^{\otimes k}\inparen{\ud x_1^{\ell},\dots, \ud x_k^{\ell}} \\
    E' &= \int_{\RR^k} \cdots \int_{\RR^k} \inangle{ \exp\inparen{ \sum_{\ell \leq n} \sum_{j \leq k} x_{j}^{\ell} \inparen{ X_{N,j}^{\ell} + h g_j}  }  }^- \prod_{\ell \leq n}  \gamma^{\otimes k}\inparen{\ud x_1^{\ell},\dots, \ud x_k^{\ell}} \\
    D &:= \int_B \cdots \int_B  \Eb_\xi \insquare{  \exp\inparen{ \sum_{\ell \leq n} \sum_{j \leq k} x_{j}^{\ell} \inparen{ X_{0,j}^{\ell} + h g_j}  } }   \prod_{\ell \leq n}  \gamma^{\otimes k}\inparen{\ud x_1^{\ell},\dots, \ud x_k^{\ell}} \\
    E &:= \int_{\RR^k} \cdots \int_{\RR^k}  \Eb_\xi \insquare{  \exp\inparen{ \sum_{\ell \leq n} \sum_{j \leq k} x_{j}^{\ell} \inparen{ X_{0,j}^{\ell} + h g_j}  } }   \prod_{\ell \leq n}  \gamma^{\otimes k}\inparen{\ud x_1^{\ell},\dots, \ud x_k^{\ell}}.
\end{align*}
\label{lemma:ST_nu0f_to_tildenu0f_D'E'DE}
\end{lemma}
\begin{proof}
The result follows from $-H_{N,M,0}(\vec{x}^{\ell}) = -H^-_{N-k,M,0}(x^{\ell}_{k+1},\dots,x^{\ell}_N) + \sum_{j \leq k} x^{\ell}_j(X_{N,j}^{\ell} + h g_j)  - \frac{2\kappa - \sigma^-}{2} \sum_{j \leq k} x_{j}^{\ell \, 2}$ and similarly $-\tilde{H}_{N,M,0\textnormal{-lim}}(\vec{x}^{\ell}) = -\tilde{H}^-_{N-k,M,0\textnormal{-lim}}(x^{\ell}_{k+1},\dots,x^{\ell}_N) + \sum_{j \leq k} x^{\ell}_j(X_{0,j}^{\ell} + h g_j)  - \frac{2\kappa - \sigma^-}{2} \sum_{j \leq k} x_{j}^{\ell \, 2}$, followed by muliplying the numerator and denominator by $1/\tilde{Z}_{N-k,M,0\textnormal{-lim}}^{-}$ to produce the $\inangle{\cdot}^{-}$ measure, and then multiplying the numerator and denominator by $\inparen{ (2\kappa - \sigma^-)/(2\pi)  }^{kn/2}$ to produce the Gaussian measure $\gamma^{\otimes kn}$. 
\end{proof}


\begin{lemma}
Let $E$ be given as in \eqref{eq:ST_nu0lim_D/E}. Then for any realization of the disorder $z_j$'s, $g_j$'s, $E \geq 1$.
\label{lemma:ST_tildenu0lim_E>=1}
\end{lemma}
\begin{proof}
Fix $z_j's$ and $g_j's$. By independence, we have $E = \prod_{\ell \leq n} \prod_{j \leq k} \exp \inparen{ \frac{(X_{0,j}^{\ell} + h g_j)^2 }{2(2\kappa - \sigma^-)} } \geq 1$.
\end{proof}

Let us now consider the condition 
\begin{align}
    \kappa > 16k^2p^3 (1 + 8\bar{\epsilon}^2)(\alpha D^2 + h^2),
    \label{eq:ST_asmpt_on_kappa}
\end{align}
for $k, p \geq 1$, $0 < \epsilon < 1/4$, and $\bar{\epsilon} = -1 + 1/4\epsilon$ (this is included in \eqref{eq:ST_kappa0_assumption}).


The purpose of Lemmas \ref{lemma:ST_nu0f_to_tildenu0f_D'E'DE} and \ref{lemma:ST_tildenu0lim_E>=1} is to establish that Lemma \ref{lemma:TalVol1_Lemma1.7.14} can be used to write
\begin{align*}
    \abs{\nu_0\insquare{f} - \tilde{\nu}_{0]\textnormal{-lim}}\insquare{f}} \leq \Eb \abs{D' - D} + \Eb \abs{E' - E},
\end{align*}
from which the random projections result will state that $X_{N,j}^\ell$ is close to $X_{0,j}^\ell$ (defined in \eqref{ST:def_X_Nj^ell,X_0j^ell}) in distribution for large $N$. 

We will need a straightforward extension of the projection result Theorem \ref{thm:ProjResult_DisorderedCase_supOverBL_LM_2p_replicated_n}. The proof is similar to that of Theorem \ref{thm:ProjResult_DisorderedCase_supOverBL_LM_2p}. Here, to accommodate the $n$ replicas of $X$, we have to replicate the original measure $\inangle{\cdot}$ $2pn$ times, where previously we only replicated $2p$ times. This brings no additional difficulty and the details are omitted.

\begin{theorem}
Under the conditions of Theorem \ref{thm:ProjResult_DisorderedCase_supOverBL_LM_2p},  for every integer $p \geq 1$ and $n \geq 1$,
\begin{align}
    &\sup_{\substack{\norm{g}_{\textnormal{Lip}} \leq L \\ \norm{g}_{\infty} \leq M} } \Eb_{\Theta}\Eb_{\textnormal{d}}\Eb_{z}  \insquare{ \inparen{ \inangle{g(\Theta^{\top} X^1,\dots, \Theta^{\top} X^n)  } - \Eb_{\xi} \insquare{  g\inparen{\sqrt{q}z + \sqrt{\rho - q}\xi^1, \dots, \sqrt{q}z + \sqrt{\rho - q}\xi^n  }  }  }^{2p}  } \nonumber\\
    &\quad\quad\quad\quad\quad\quad\quad\quad\quad \leq \frac{ K k M^{2p - 1} L k n^2 /\sqrt{\rho - q}}{N - 1} \inparen{ d_1 + d_2  },
\end{align}
where $z$, $\inparen{ \xi^\ell }_{\ell \leq n}$ are independent standard Gaussian random vectors in $\RR^k$, and where $K = K(p,\rho,q) > 0$, and
\begin{align*}
    d_1 := \sqrt{3N^2 c_1 + 4N\rho \sqrt{c_1} + 2N\rho^2 }, \quad 
    d_2 := \sqrt{3N^2 c_2 + 4Nq\sqrt{c_2} + 2Nq^2 }.
\end{align*}
\label{thm:ProjResult_DisorderedCase_supOverBL_LM_2p_replicated_n}
\end{theorem}

\begin{proofof}{Proposition \ref{proposition:ST_nu0_to_tildenu0lim}}
The proof is largely similar to that of Theorem \ref{thm:LI_abstract_theorem}. In what follows, the indices $j$ and $\ell$ range over $1 \leq j \leq k$, and $1 \leq \ell \leq 2p$ respectively. This means for instance $\sum_{j,\ell} = \sum_{j \leq k} \sum_{\ell \leq 2p}$. Let $D'$, $E'$, $D$, $E$ be defined as in Lemma \ref{lemma:ST_nu0f_to_tildenu0f_D'E'DE}. We have from Lemma \ref{lemma:TalVol1_Lemma1.7.14}, since also $E \geq 1$ from Lemma \ref{lemma:ST_tildenu0lim_E>=1}, that
\begin{align}
    \abs{\nu_0\insquare{f} - \tilde{\nu}_{0\textnormal{-lim}}\insquare{f} } \leq \Eb \abs{D' - D} + \Eb \abs{E' - E}.
    \label{eq:ST_nu0limf_nu0f_truncation_firstExpansion}
\end{align}
By Cauchy-Schwarz, Jensen's inequality, and Fubini-Tonelli's theorem, we have
\begin{align}
    \Eb \abs{ D' - D  } \leq \sqrt{ \int_B \cdots \int_B \Eb\insquare{ \inparen{ D'_0 - D_0  }^2  } \prod_{\ell \leq 2p} \gamma^{\otimes k}\inparen{ \ud x_1^\ell,\dots, \ud x_k^\ell  } },
    \label{eq:ST_nu0limf_nu0f_truncation_E|D'-D|}
\end{align}
where
\begin{align*}
    D'_0 &= D'_0\inparen{ \inparen{x_j^\ell}  } := \inangle{ \exp\inparen{ \sum_{j,\ell} x_{j}^{\ell} \inparen{ X_{N,j}^{\ell} + h g_j}  }  }^- \\
    D_0 &= D_0\inparen{ \inparen{x_j^\ell}  } := \Eb_\xi \insquare{  \exp\inparen{ \sum_{j,\ell} x_{j}^{\ell} \inparen{ X_{0,j}^{\ell} + h g_j}  } },
\end{align*}
with the quantities $X_{N,j}^\ell$, $X_{0,j}^{\ell}$ defined as in \eqref{ST:def_X_Nj^ell,X_0j^ell}.

Fix $A > 0$, and define $\tilde{\varphi}(x) \in C_b(\RR^1)$ by
\begin{align*}
\tilde{\varphi}(x) = 
    \begin{cases}
    1, & x \in [-A, A], \\
    0, & x \in [-2A, 2A]^c,\\
    \frac{1}{A}x + 2 & x \in [-2A, -A] \\
    -\frac{1}{A}x + 2 & x \in [A, 2A]. 
    \end{cases}
\end{align*}
Note that $\rchi_{[-A,A]} \leq \tilde{\varphi} \leq \rchi_{[-2A, 2A]}$. Define $\varphi \in C_b(\RR^{kn})$ by $\varphi\inparen{ \inparen{ Y_{j}^\ell } } := \prod_{j,\ell} \tilde{\varphi}(Y_j^\ell)$. Define the function $f : \RR^{2kp} \rightarrow \RR$ by
\begin{align*}
    f\inparen{ \inparen{Y_j^\ell}  } := M_0\, \varphi\!\inparen{ \inparen{ Y_{j}^\ell } } \prod_{j,\ell} \exp\inparen{ x_j^\ell Y_j^\ell  },
\end{align*}
where $M_0 = M_0\inparen{ \inparen{ x_j^\ell  } } := \exp\inparen{ \sum_{j,\ell} x_j^\ell h g_j  }$. It is not hard to compute that 
\begin{align}
    \norm{f}_{\infty} &\leq M := M_0 \exp \inparen{ 2A \sum_{j,\ell} \abs{x_j^\ell}  } \nonumber\\
    \norm{f}_{\textnormal{Lip}} &\leq L := 2M_0 \exp\inparen{ 2A \sum_{j,\ell} \abs{x_j^\ell}  } \sqrt{ \max\inparen{  \sum_{j,\ell} \abs{x_j^\ell}^2, \frac{2kp}{A^2}  }  },
    \label{eq:ST_nu0limf_nu0f_truncation_M,L}
\end{align}
where the last follows from $\abs{ \partial f/\partial Y_{j}^\ell  } \leq M_0 \inparen{ \abs{x_j^\ell} + 1/A  } \exp \inparen{ 2A  \sum_{j,\ell} \abs{x_j^\ell} }$, which holds everywhere except for a finite number of points where $f$ is not differentiable. This yields a bound on $\norm{\nabla f}_2$ which holds almost everywhere, from which \eqref{eq:ST_nu0limf_nu0f_truncation_M,L} follows.

By Lemma \ref{lemma:triangleIneq_E[()^2]}, we have
\begin{align}
    \Eb \insquare{ \inparen{ D'_0 - D_0 }^2  } &\leq 4 \inparen{ \Eb \insquare{ F^2 } + \Eb\insquare{G^2} },
    \label{eq:ST_nu0limf_nu0f_truncation_E(D'0-D_0)^2}
\end{align}
where
\begin{align*}
    F &= \inangle{ f\inparen{ \inparen{ X_{N,j}^{\ell}  }  } }^- - \Eb_\xi \insquare{ f\inparen{ \inparen{ X_{0,j}^{\ell}  }  }  } \\
    G &= M_0 \inangle{  \inparen{1 - \varphi\!\inparen{ \inparen{ X_{N,j}^\ell } }} \prod_{j,\ell} \exp\inparen{ x_j^\ell X_{N,j}^\ell  }  }^- - M_0 \Eb_\xi \insquare{  \inparen{1 - \varphi\!\inparen{ \inparen{ X_{0,j}^\ell }  }} \prod_{j,\ell} \exp\inparen{ x_j^\ell X_{0,j}^\ell  } }.
\end{align*}
A bound for $\Eb\insquare{F^2}$ follows from the limiting projection result Theorem \ref{thm:ProjResult_DisorderedCase_supOverBL_LM_2p_replicated_n} (applied with $n = 2p$ and $p = 1$), together with thin-shell and overlap concentration for the truncated auxiliary system \eqref{eq:ST_TS_truncated_auxiliarySystem}, \eqref{eq:ST_OC_truncated_auxiliarySystem}. This yields
\begin{align}
    \Eb\insquare{F^2} &\leq \frac{K(k,p) \Eb\insquare{M_0^2} \exp\inparen{ 4A \sum_{j,\ell} \abs{x_j^\ell}  } \sqrt{ \max\inparen{  \sum_{j,\ell} \abs{x_j^\ell}^2, \frac{2kp}{A^2}  }  }}{\sqrt{N-k}},
    \label{eq:ST_nu0limf_nu0f_truncation_EF^2_bound}
\end{align}
where we have used that $\sqrt{M} \leq K(k) M/\sqrt{N-k}$. On the other hand
\begin{align*}
    \Eb\insquare{G^2} &= G_1 + G_2 + G_3,
\end{align*}
where
\begin{align*}
    G_1 &= \Eb\insquare{ M_0^2 \inangle{  \inparen{1 - \varphi\!\inparen{ \inparen{ X_{N,j}^\ell } }} \prod_{j,\ell} \exp\inparen{ x_j^\ell X_{N,j}^\ell  }  }^{-,2}  } \\
    G_2 &= -2 \Eb\insquare{ M_0^2 \inangle{  \inparen{1 - \varphi\!\inparen{ \inparen{ X_{N,j}^\ell } }} \prod_{j,\ell} \exp\inparen{ x_j^\ell X_{N,j}^\ell  }  }^-  \Eb_\xi \insquare{  \inparen{1 - \varphi\!\inparen{ \inparen{ X_{0,j}^\ell } }} \prod_{j,\ell} \exp\inparen{ x_j^\ell X_{0,j}^\ell  } }  } \\
    G_3 &= \Eb \insquare{ \inparen{\Eb_\xi \insquare{  \inparen{1 - \varphi\!\inparen{ \inparen{ X_{0,j}^\ell }  }} \prod_{j,\ell} \exp\inparen{ x_j^\ell X_{0,j}^\ell  } }}^2  }.
\end{align*}
Since $1 - \varphi\!\inparen{ \inparen{Y_j^\ell }  } \leq \rchi_{\inbraces{  \bigcup_{j,\ell} \abs{ Y_j^\ell } \geq A   } }$, we have, for $\eta > 0$,
\begin{align}
    G_1 &\leq \Eb\insquare{M_0^2 \inangle{ \rchi_{\inbraces{  \bigcup_{j,\ell} \abs{ X_{N,j}^\ell } \geq A   } } \prod_{j,\ell} \exp\inparen{  2x_j^\ell X_{N,j}^\ell  }   }^- } \nonumber\\
    &= \Eb\insquare{M_0^2 \inangle{ \exp\inparen{-2\eta \sum_{j,\ell} \abs{X_{N,j}^\ell} } \rchi_{\inbraces{  \bigcup_{j,\ell} \abs{ X_{N,j}^\ell } \geq A   } } \prod_{j,\ell} \exp\inparen{  2\inparen{ x_j^\ell X_{N,j}^\ell + \eta \abs{X_{N,j}^\ell} }  }   }^- }
    \nonumber \\
    &\leq \exp\inparen{-2\eta A} \Eb\insquare{M_0^2 \inangle{ \prod_{j,\ell} \insquare{ \exp\inparen{2\inparen{ x_j^\ell X_{N,j}^\ell + \eta X_{N,j}^\ell }} + \exp\inparen{2\inparen{ x_j^\ell X_{N,j}^\ell - \eta X_{N,j}^\ell }} }   }^-  },
    \label{eq:ST_nu0limf_nu0f_truncation_G1_expansion}
\end{align}
where the first inequality follows from Jensen's inequality, and in the last inequality we used $e^{t\abs{z}} \leq e^{tz} + e^{-tz}$. Define for $s \in \inbraces{\pm 1}$,
\begin{align*}
    A^{\ell}_{N,j}(s) := \inparen{ x_j^\ell + s \eta }X_{N,j}^{\ell}.
\end{align*}
Let us note that, for every $j, \ell, s$, and $\lambda \in \RR$,
\begin{align}
    \Eb \exp \lambda A^{\ell}_{N,j}(s) &= \Eb \insquare{ \exp \lambda \inparen{  x_j^\ell + s\eta } \sum_{m \leq M} \frac{g_{j,m}}{\sqrt{N-k}} \tilde{u}'\inparen{S_m^{-,\ell}} } \nonumber\\
    &= \exp \inparen{ \frac{\lambda^2 \inparen{  x_j^\ell + s\eta }^2 }{2} \sum_{m \leq M} \frac{\tilde{u}^{\prime 2}\inparen{ S_m^{-,\ell}  }}{N-k}  } \nonumber\\
    &= \exp\inparen{ 2 \lambda^2\inparen{ \frac{1}{2}x_j^\ell + \frac{1}{2} s\eta  }^2  \alpha D^2 } \nonumber\\
    &\leq \exp\inparen{ \lambda^2 \alpha D^2 \inparen{ x_* + \eta^2  }  },
    \label{eq:ST_nu0limf_nu0f_truncation_bound_Eexp_lambda_A}
\end{align}
where we set $x_* := \max x_j^{\ell,2}$. From \eqref{eq:ST_nu0limf_nu0f_truncation_G1_expansion} we have, letting $\Eb_{\Theta}$ denote expectation in $(g_{j,m})_{j \leq k, m \leq M}$, by independence
\begin{align}
    G_1 &\leq \exp\inparen{-2\eta A} \Eb\insquare{M_0^2 \inangle{ \prod_{j,\ell} \insquare{ \exp{2 A^\ell_{N,j}(+1) } + \exp{2 A^\ell_{N,j}(-1) }   } }^- } \nonumber\\
    &\leq \exp\inparen{-2\eta A} \Eb\insquare{M_0^2 \prod_{j \leq k} \Eb_{\Theta} \inparen{ \prod_{\ell \leq n} \insquare{ \exp 2 A_{N,j}^\ell(+1) + \exp 2 A_{N,j}^\ell(-1) }   } } \nonumber\\
    &= \exp\inparen{-2\eta A} \Eb\insquare{M_0^2 \prod_{j \leq k} \sum_{\substack{ 1 \leq \ell_1 < \cdots < \ell_{2p} \leq 2p \\ s_1,\dots,s_{2p} \in \inbraces{\pm 1} }} \Eb_{\Theta} \insquare{\exp  2A_{N,j}^{\ell_1}(s_1) \exp 2A_{N,j}^{\ell_2}(s_2)\cdots \exp 2A_{N,j}^{\ell_{2p}}(s_{2p}) } } \nonumber\\
    &\leq \exp\inparen{-2\eta A} \Eb\insquare{M_0^2 \prod_{j \leq k} \sum_{\substack{ 1 \leq \ell_1 < \cdots < \ell_{2p} \leq 2p \\ s_1,\dots,s_{2p} \in \inbraces{\pm 1} }} \inparen{ \Eb_{\Theta} \insquare{\exp  4p A_{N,j}^{\ell_1}(s_1)  } }^{1/2p} \cdots \inparen{ \Eb_{\Theta} \insquare{\exp  4p A_{N,j}^{\ell_{2p}}(s_{2p})  } }^{1/2p}  } \nonumber\\
    &\leq \exp\inparen{-2\eta A} \Eb\insquare{M_0^2} K(p,k) \exp\inparen{ 16 k p^2 \alpha D^2 \inparen{x_* + \eta^2}  },
    \label{eq:ST_nu0limf_nu0f_truncation_G_1}
\end{align}
where the last inequality follows from \eqref{eq:ST_nu0limf_nu0f_truncation_bound_Eexp_lambda_A}. For $G_3$, we define
\begin{align*}
    B_{N,j}^{\ell}(s) &:= \inparen{x_j^\ell + s\eta} X_{0,j}^{\ell}, 
\end{align*}
and we have for any $j, \ell, s$, and $\lambda \in \RR$,
\begin{align*}
    \Eb \exp \lambda B_{N,j}^{\ell}(s) &= \exp \inparen{2 \lambda^2\inparen{ \frac{1}{2}x_j^\ell + \frac{1}{2} s\eta  }^2  \tau^-   } \leq \exp \inparen{\lambda^2 \alpha D^2 (x_* + \eta^2)},
\end{align*}
which yields by a similar argument that
\begin{align}
    G_3 &\leq \exp\inparen{-2\eta A} \Eb\insquare{M_0^2} K(p,k) \exp\inparen{ 16 k p^2 \alpha D^2 \inparen{x_* + \eta^2}  }.
    \label{eq:ST_nu0limf_nu0f_truncation_G_3}
\end{align}
For $G_2$, the same argument will work, with Cauchy-Schwarz:
\begin{align}
    \abs{G_2} &\leq 2 \sqrt{G_1}\sqrt{G_3} \leq \exp\inparen{-2\eta A} \Eb\insquare{M_0^2} K(p,k) \exp\inparen{ 16 k p^2 \alpha D^2 \inparen{x_* + \eta^2}  }.
    \label{eq:ST_nu0limf_nu0f_truncation_G_2}
\end{align}
We also compute
\begin{align}
    \Eb\insquare{M_0^2} &= \Eb \exp \sum_{j, \ell} 2x_j^\ell h g_j = \prod_{j \leq k} \exp \inparen{16 k p^2 h^2 \inparen{\sum_{ \ell \leq 2p  } \frac{x_j^\ell}{2p} }^2} \leq \exp\inparen{ 16p^2h^2 x_*  }.
     \label{eq:ST_nu0limf_nu0f_truncation_EM_0^2}
\end{align}

The bounds for $G_1$, $G_2$, $G_3$, and $\Eb M_0^2$ in \eqref{eq:ST_nu0limf_nu0f_truncation_G_1}, \eqref{eq:ST_nu0limf_nu0f_truncation_G_2}, \eqref{eq:ST_nu0limf_nu0f_truncation_G_3}, \eqref{eq:ST_nu0limf_nu0f_truncation_EM_0^2} yield that
\begin{align}
    \Eb \insquare{G^2} &\leq K(p,k) \exp\inparen{-2\eta A} \exp\inparen{ 16 k p^2 \inparen{ \alpha D^2 + h^2 } \inparen{x_* + \eta^2}  }.
    \label{eq:ST_nu0limf_nu0f_truncation_EG^2_bound}
\end{align}
Let $T = \sum_{j,\ell } \abs{x_j^\ell}$, and set for $0 < \epsilon < 1/2$,
\begin{align*}
    A := \frac{1}{4T} \log (N-k)^\epsilon; \quad \eta = 2T \inparen{ \frac{1}{2\epsilon} - 1  } \equiv 2T \bar{\epsilon},
\end{align*}
(note that $T = 0$ on a negligible set, for $x_j^\ell$ having Gaussian law). Note that 
\begin{align*}
    x_* + \eta^2 \leq \sum_{j,\ell} x_{j}^{\ell,2} + 4\bar{\epsilon}^2 \inparen{ \sum_{j,\ell}  \abs{x_j^\ell}  }^2 \leq \inparen{1 + 8\bar{\epsilon}^2} kp \sum_{j,\ell} x_j^{\ell,2},
\end{align*}
and that whenever $N \geq k + \exp \sqrt{32kp}/\epsilon$, then 
\begin{align*}
    \sqrt{\max \inparen{  \sum_{j,\ell} \abs{x_j^\ell}^2, \, \frac{2kp}{A^2}  }  } \leq \sqrt{ \max\inparen{  \inparen{\sum_{j,\ell} \abs{x_j^\ell}}^2, \frac{32kp \inparen{\sum_{j,\ell} \abs{x_j^\ell}}^2 }{ \inparen{ \log (N - k)^\epsilon  }^2 }  } } \leq \sum_{j,\ell} \abs{x_j^\ell}.
\end{align*}
From \eqref{eq:ST_nu0limf_nu0f_truncation_EF^2_bound} and \eqref{eq:ST_nu0limf_nu0f_truncation_EG^2_bound} substituted into \eqref{eq:ST_nu0limf_nu0f_truncation_E(D'0-D_0)^2}, we then have for $N$ sufficiently large as above,
\begin{align}
    \Eb\insquare{\inparen{D'_0 - D_0}^2} \leq \frac{K\inparen{ \sum_{j,\ell} \abs{x_j^\ell}  } \exp \inparen{ 16k^2p^3 (1 + 8\bar{\epsilon}^2)(\alpha D^2 + h^2) \sum_{j,\ell} x_j^{\ell,2}   }}{(N-k)^{\frac{1}{2} - \epsilon}}.
\end{align}
From \eqref{eq:ST_nu0limf_nu0f_truncation_E|D'-D|} we then have
\begin{align}
    \Eb \abs{ D' - D  } &\leq  \frac{K \sqrt{C}}{(N-k)^{1/4 - \epsilon/2}}
    \label{eq:ST_nu0limf_nu0f_truncation_E|D'-D|_finalBound}
\end{align}
where $C$ is given by 
\begin{align*}
    C &= \int_{\RR^{2kp}} \inparen{\sum_{j,\ell} \abs{x_j^\ell} } \exp \inparen{ 16k^2p^3 (1 + 8\bar{\epsilon}^2)(\alpha D^2 + h^2) \sum_{j,\ell} x_j^{\ell,2}   } \prod_{\ell \leq 2p} \gamma^{\otimes k}\inparen{ \ud x_1^\ell,\dots, \ud x_k^\ell  } \\
    &= \inparen{ \sqrt{ \frac{2\kappa - \sigma^-}{2\pi}  } \int \abs{x} \exp \inparen{ -\frac{1}{2}\inparen{ 2\kappa - \sigma^- - 32k^2p^3 (1 + 8\bar{\epsilon}^2)(\alpha D^2 + h^2) } x^2  }  }^{2kp},
\end{align*}
which exists under assumption \ref{eq:ST_asmpt_on_kappa}. A similar bound as in \eqref{eq:ST_nu0limf_nu0f_truncation_E|D'-D|_finalBound} holds for $\Eb \abs{E' - E}$. The result then follows from \eqref{eq:ST_nu0limf_nu0f_truncation_firstExpansion}.
\end{proofof}

\subsection{Proof of Proposition \ref{proposition:ST_nu0lim_to_tildenu0lim}}
\label{sec:ST_nu0lim_to_tildenu0lim}


Define for $0 \leq \eta \leq 1$,
\begin{align}
    -H_{N,M,0\textnormal{-lim},\eta}(\vec{x}) &= \sum_{m \leq M} u\inparen{S_m^0} - \kappa \norm{\vec{x}}^2 + h\sum_{i \leq N} g_i x_i   + \sqrt{\eta}\sqrt{r} \sum_{j \leq k} x_j z_j + \sqrt{1 - \eta}\sqrt{r^-} \sum_{j \leq k} x_j z_j^- \nonumber\\
    &\quad\quad -\inparen{ \eta(r - \tau - \sigma) + (1-\eta)\inparen{r^- - \tau^- - \sigma^-}  } \sum_{j \leq k} \frac{x_j^2}{2},
    \label{eq:ST_H_NM0-lim_eta_definition}
\end{align}
with associated Gibbs expectations denoted $\nu_{0\textnormal{-lim},\eta}$. Note that $\nu_{0\textnormal{-lim},1} = \nu_{0\textnormal{-lim}}$ and $\nu_{0\textnormal{-lim},0} = \tilde{\nu}_{0\textnormal{-lim}}$.

\begin{proposition}
Let $f$ be an integrable function on $\inparen{\RR^N}^n$, then
\begin{align}
    \frac{\ud}{\ud \eta} \nu_{0\textnormal{-lim},\eta} \insquare{f} &= \frac{k}{2} \left( \sum_{\ell, \ell' \leq n} (r - r^-) \nu_{0\textnormal{-lim},\eta} \insquare{ x_1^\ell x_1^{\ell'} f}  - n\sum_{\ell \leq n } (r-r^-) \nu_{0\textnormal{-lim},\eta} \insquare{ x_1^\ell x_1^{n+1} f}  \right. \nonumber\\
    &\quad\quad\quad\quad \left. - \sum_{\ell \leq n} \inparen{r - \tau - \sigma - (r^- - \tau^- - \sigma^-)} \nu_{0\textnormal{-lim},\eta} \insquare{ x_1^{\ell,2} f}  \right).
    \label{eq:ST_etaInterpolation_generalCase}
\end{align}
Moreover, if $\norm{f}_{\infty} \leq 1$, we have
\begin{align}
    \abs{ \frac{\ud}{\ud \eta} \nu_{0\textnormal{-lim},\eta} \insquare{f} } &\leq K \inparen{ \abs{r - r^-} + \abs{\tau - \tau^-} + \abs{\sigma - \sigma^-}  }. 
    \label{eq:ST_etaInterpolation_|f|<=1}
\end{align}
\end{proposition}

\begin{proof}
Identity \eqref{eq:ST_etaInterpolation_generalCase} is a standard GIPF argument, similar to Proposition \ref{proposition:perceptron_t_interpolation}.

To deduce \eqref{eq:ST_etaInterpolation_|f|<=1}, it suffices to have $\nu_{0\textnormal{-lim},\eta} \insquare{x_1^2} \leq K$. But this is clear, since $-H_{N,M,0\textnormal{-lim},\eta}$ is $(2\kappa)$-strongly concave, so that we may extend Lemma \ref{lemma:ST_nicePropertiesOf_nu_1s}, and in particular \eqref{eq:ST_nu_tlim,s_exponentialInteg_livingOnSphere} to this setting.
\end{proof}

From \eqref{eq:ST_etaInterpolation_|f|<=1}, we see that to fulfill our program that $\tilde{\nu}_{0\textnormal{-lim}} f \simeq \nu_{0\textnormal{-lim}} f$ for the class of $f$'s \eqref{eq:ST_f_cylinderIndicator}, it suffices to show $r^- \simeq r$, $\tau^- \simeq \tau$, $\sigma^- \simeq \sigma$. Define the function $F : [0,1] \times \RR^5 \rightarrow \RR^5$ by
\begin{align}
    F(\gamma, r, \tau, \sigma, \rho, q) = F(\gamma, \vec{y}) &:= \inparen{ F_1(\gamma,\vec{y}), F_2(\gamma,\vec{y}), F_3(\gamma,\vec{y}), F_4(\gamma,\vec{y}), F_5(\gamma,\vec{y})  },
\end{align}
where a generic point $\vec{y} = (r,\tau,\sigma,\rho,q)$, and, with $\theta_\gamma = \sqrt{1 - \gamma} \inparen{\sqrt{q}z + \sqrt{\rho - q} \xi} $, where
\begin{align}
    &F_1(\gamma, \vec{y}) := \alpha \Eb \insquare{ \inparen{   \frac{\Eb_\xi u'\inparen{ \theta_\gamma  } \exp u\inparen{\theta_\gamma} }{\Eb_\xi \exp u\inparen{\theta_\gamma}}   }^2  }; \quad\quad 
    F_2(\gamma, \vec{y}) := \alpha \Eb \insquare{    \frac{\Eb_\xi u^{\prime 2}\inparen{ \theta_\gamma  } \exp u\inparen{\theta_\gamma} }{\Eb_\xi \exp u\inparen{\theta_\gamma}}     }; \nonumber\\ 
    &F_3(\gamma, \vec{y}) := \alpha \Eb \insquare{    \frac{\Eb_\xi u''\inparen{ \theta_\gamma  } \exp u\inparen{\theta_\gamma} }{\Eb_\xi \exp u\inparen{\theta_\gamma}}     }; \nonumber\\
    & F_4(\gamma, \vec{y}) := \frac{1}{2\kappa + r - \sigma - \tau} + \frac{r + h^2}{\inparen{2\kappa + r - \sigma - \tau}^2}; \quad\quad F_5(\gamma, \vec{y}) := \frac{r + h^2}{\inparen{2\kappa + r - \sigma - \tau}^2}.
\end{align}
Set 
\begin{align}
    \inparen{ r(\gamma), \tau(\gamma), \sigma(\gamma), \rho(\gamma), q(\gamma) } := F(\gamma, r(\gamma), \tau(\gamma), \sigma(\gamma), \rho(\gamma), q(\gamma) ),
    \label{ST:definition_of_ygamma}
\end{align}
or more succinctly, $\vec{y}(\gamma) = F(\gamma, \vec{y}(\gamma))$. From \eqref{eq:ST_RS_r_tau_sigma} and \eqref{eq:ST_truncated_RS_r_tau_sigma_convenientForm}, we see that $\inparen{r(0), \tau(0), \sigma(0), \rho(0), q(0)}$ are the solutions to the RS equations for the original system, while $\inparen{r(k/N), \tau(k/N), \sigma(k/N), \rho(k/N), q(k/N)}$ are the solutions for the $(N-k)$-truncated system. Define
\begin{align*}
    \tilde{F}(\gamma, r, \tau, \sigma, \rho, q) &:= (r,\tau,\sigma,\rho, q) - F(\gamma, r,\tau,\sigma,\rho, q),
\end{align*}
so that $\tilde{F}(\gamma, r(\gamma), \tau(\gamma), \sigma(\gamma), \rho(\gamma), q(\gamma) ) = \tilde{F}(\gamma, \vec{y}(\gamma)) =  0$. From now on, we fix $\gamma$, and we fix that all quantities $r, \tau, \sigma, \rho, q$ are in fact $r(\gamma)$, $\tau(\gamma)$, $\sigma(\gamma)$, $\rho(\gamma)$, $q(\gamma)$ (that is, they satisfy \eqref{ST:definition_of_ygamma}. Write 
\begin{align*}
    F_{v, w} &:= \frac{\partial F_v}{\partial w}(\gamma, \vec{y}(\gamma)), \quad v \in \inbraces{1,2,\dots,5}, w \in \inbraces{r,\tau,\sigma,\rho,q}.
\end{align*}
Let $J_{\tilde{F}, \vec{y}} (\gamma, \vec{y}(\gamma))$ be the $5\times 5$ Jacobian of $\tilde{F}$ wrt.~coordinates $\vec{y}$. We have 
\begin{align}
    J_{\tilde{F}, \vec{y}} (\gamma, \vec{y}(\gamma))  = \begin{bmatrix}
    1 & 0 & 0 & F_{1,\rho} & F_{1,q} \\
    0 & 1 & 0 & F_{2,\rho} & F_{2,q} \\
    0 & 0 & 1 & F_{3,\rho} & F_{3,q} \\
    F_{4,r} & F_{4,\tau} & F_{4,\sigma} & 1 & 0 \\
    F_{5,r} & F_{5,\tau} & F_{5,\sigma} & 0 & 1
    \end{bmatrix} =: \begin{bmatrix}
    I_3 & B \\
    C & I_2
    \end{bmatrix}.
    \label{eq:ST_J_Fy_firstForm}
\end{align}

Consider the following assumption on $\kappa$ (this is included in \eqref{eq:ST_kappa0_assumption})
\begin{align}
    \kappa > \sqrt{150(1+h^2)} \alpha D^4.
    \label{eq:ST_asmpt_on_kappa_FOR_JACOBIAN}
\end{align}

\begin{lemma}
The matrix $J_{\tilde{F}, \vec{y}} (\gamma, \vec{y}(\gamma))$ is invertible under \eqref{eq:ST_asmpt_on_kappa_FOR_JACOBIAN}.
\label{lemma:ST_J_Fy_invertible}
\end{lemma}

The proof of Lemma \ref{lemma:ST_J_Fy_invertible} is in Appendix \ref{sec:suppProofs_for_ST}. With Lemma \ref{lemma:ST_J_Fy_invertible}, and under \eqref{eq:ST_asmpt_on_kappa_FOR_JACOBIAN}, we have by the implicit function theorem that 
\begin{align}
    \begin{bmatrix}
    r'(\gamma) & \tau'(\gamma) & \sigma'(\gamma) & \rho'(\gamma) & q'(\gamma)
    \end{bmatrix}^{\top} = - J_{\tilde{F}, \vec{y}} (\gamma, \vec{y}(\gamma))^{-1} J_{\tilde{F}, \gamma} (\gamma, \vec{y}(\gamma)),
\end{align}
where $J_{\tilde{F}, \gamma} (\gamma, \vec{y}(\gamma))$ is the $5\times 1$ Jacobian of $\tilde{F}$ wrt.~$\gamma$.

\begin{lemma}
The derivatives $r'(\gamma)$, $\tau'(\gamma)$, $\sigma'(\gamma)$ are all bounded by a constant $K$ that does not depend on $N$, for all $\gamma$.
\label{lemma:ST_r',tau',sigma'_bounded}
\end{lemma}

\begin{proof}
The expressions in \eqref{eq:ST_F4r_to_F5sigma} and \eqref{eq:ST_F1rho_to_F3q} show that all entries of $J_{\tilde{F}, \gamma} (\gamma, \vec{y}(\gamma))$ are not growing with $N$, and since the matrix is invertible, the spectral norm of $J_{\tilde{F}, \vec{y}} (\gamma, \vec{y}(\gamma))^{-1}$ is a constant that does not depend on $N$. On the other hand, the last two entries in $J_{\tilde{F}, \gamma} (\gamma, \vec{y}(\gamma))$ are zero, and a very similar computation involving differentiation and GIPF as in \eqref{eq:ST_F1rho_to_F3q} will show that the first three entries of $J_{\tilde{F}, \gamma} (\gamma, \vec{y}(\gamma))$ are of the form $\frac{\alpha}{2} \mu \insquare{A}$, where $A$ is some polynomial in $U'$, $U''$, $U'''$. Hence, the entries in $J_{\tilde{F}, \gamma} (\gamma, \vec{y}(\gamma))$ are not growing with $N$ as well and the conclusion follows. 
\end{proof}

\begin{proposition}
We have
\begin{align}
    \abs{r - r^-} \leq \frac{K}{N}; \quad \abs{\tau - \tau^-} \leq \frac{K}{N}; \quad \abs{\sigma - \sigma^-} \leq \frac{K}{N}.
\end{align}
\label{proposition:ST_r-r^-,tau-tau^-,sigma-sigma^-}
\end{proposition}
\begin{proof}
Follows from $\abs{r - r^-} = \abs{r(0) - r(k/N)} \leq \inparen{\sup_\gamma \abs{r'(\gamma)}}k/N$ and Lemma \ref{lemma:ST_r',tau',sigma'_bounded}; similarly for the rest.
\end{proof}

\begin{proofof}{Proposition \ref{proposition:ST_nu0lim_to_tildenu0lim}}
The result follows from \eqref{eq:ST_etaInterpolation_|f|<=1} and Proposition \ref{proposition:ST_r-r^-,tau-tau^-,sigma-sigma^-}.
\end{proofof}


\subsection{Proof of Lemma \ref{lemma:ST_J_Fy_invertible}}

\begin{proofof}{Lemma \ref{lemma:ST_J_Fy_invertible}}
Reminder that here $r,\tau,\sigma,\rho,q$ are really $r(\gamma), \tau(\gamma), \sigma(\gamma), \rho(\gamma), q(\gamma)$. Write $R = 2\kappa + r - \sigma - \tau$. Direct computation gives
\begin{align}
    &F_{4,r} = \frac{-2q}{R}; \quad F_{4,\tau} = \frac{\rho + q}{R}; \quad F_{4,\sigma} = \frac{\rho + q}{R}; \nonumber\\
    &F_{5,r} = \frac{\rho - 3q}{R}; \quad F_{5,\tau} = \frac{2q}{R}; \quad F_{5,\sigma} = \frac{2q}{R}.
    \label{eq:ST_F4r_to_F5sigma}
\end{align}
Let us denote $\theta^\ell_\gamma := \sqrt{1-\gamma} \inparen{ \sqrt{q}z + \sqrt{\rho - q}\xi^\ell }$, and $\theta_{\gamma,\rho}^\ell := (\partial \theta^\ell_{\gamma})/(\partial \rho)$ and $\theta_{\gamma,q}^\ell := (\partial \theta^\ell_{\gamma})/(\partial q)$. Let $\mu$ be the functional given by, for some integer $n$,
\begin{align*}
    \mu\insquare{ \phi\inparen{ \inparen{ \theta^{\ell}_{\gamma} }_{\ell \leq n} } }  &:= \Eb \insquare{ \frac{ \Eb_\xi \insquare{ \phi( (\theta^\ell_\gamma)_{\ell \leq n} ) \exp \inparen{ \sum_{\ell \leq n} u\inparen{ \theta^\ell_\gamma  }  }  }   }{ \inparen{ \Eb_\xi \exp u(\theta^1_\gamma)}^{n} } }.
\end{align*}
Observe that $F_1 = \alpha\mu\insquare{  u'(\theta^1_\gamma)u'(\theta^2_\gamma)  }$, $F_2 = \alpha\mu\insquare{  u^{\prime 2}(\theta^1_\gamma)  }$, $F_3 = \alpha\mu\insquare{  u^{\prime\prime }(\theta^1_\gamma)  }$. The expressions for the $F_{v,\rho}, F_{v,q}$, $v = 1,2$ are more tedious, but the computation is straightforward---differentiation followed by GIPF to simplify; using the relations, if  $\Eb \theta_{\gamma, \rho}^{\ell} \theta_\gamma^\ell = (1-\gamma)/2$, $\Eb \theta_{\gamma, \rho}^{\ell} \theta_\gamma^{\ell'} = 0$ for $\ell \neq \ell'$, and $\Eb \theta_{\gamma, q}^{\ell} \theta_\gamma^{\ell'} = (1-\gamma)/2 -  \Eb \theta_{\gamma, \rho}^{\ell} \theta_\gamma^{\ell'} $. Using the shorthand $U^{(d)}_{\ell} \equiv u^{(d)}(\theta_\gamma^\ell)$, we have
\begin{align}
    F_{1,\rho} &= \frac{\alpha(1-\gamma)}{2} \mu\left[ U'''_{1} U'_2  + U''_1 U'_1 U'_2 + U'''_2 U'_1 + U''_2 U'_2 U'_1 + 2 U''_1 U'_2  \right. \nonumber\\
    &\quad\quad\quad\quad\quad\quad\quad + \left. U^{\prime 3}_1 U'_2 + 2 U''_2 U'_1 + U^{\prime 3}_2 U'_1 - 2 U'_1 U'_2 U''_3 - 2 U'_1 U'_2 U^{\prime 2}_3 \right]   \nonumber \\
    F_{1,q} &= \frac{\alpha(1-\gamma)}{2}\mu\left[  U''_1 U''_2 + U''_1 U^{\prime 2}_2 - 2U''_1 U'_2 U'_3 + U''_2 U''_1 + U''_2 U^{\prime 2}_1  \right. \nonumber\\
    &\quad\quad\quad\quad\quad\quad\quad -2 U''_2 U'_2 U'_3 +  U^{\prime 2}_1 U''_2 + 2U^{\prime 2}_1 U^{\prime 2}_2 - 2 U^{\prime 2}_1 U'_2 U'_3 + U^{\prime 2}_2 U'_1 \nonumber\\
    &\quad\quad\quad\quad\quad\quad\quad  \left. - 2 U^{\prime 2}_2 U'_1 U'_3 -2 U''_1 U'_2 U'_3 - 2U'_1 U''_2 U'_3 - 2U^{\prime 2}_1 U'_2 U'_3 - 2 U'_1 U^{\prime 2}_2 U'_3 + 6 U'_1 U'_2 U'_3 U'_4 \right] \nonumber \\
    F_{2,\rho} &= \frac{\alpha(1-\gamma)}{2} \mu \insquare{  2 U'''_1 + 2U''_1U'_1 + 3U''_1 + U^{\prime 4}_1 - U^{\prime 2}_1 U''_2 - U^{\prime 2}_1 U^{\prime 2}_2  } \nonumber\\
    F_{2,q} &= \frac{\alpha(1-\gamma)}{2} \mu \insquare{ -2U''_1 U'_2 - 2 U^{\prime 3}_1 U'_2 - 2 U''_1 U'_2 + 2 U^{\prime 2}_1 U'_2 U'_3 } \nonumber\\
    F_{3,q} &= \frac{\alpha(1-\gamma)}{2} \mu \insquare{ U'''_1 + 2U'''_1 U'_1 + U^{\prime \prime 2}_1 - U''_1 U''_2 - U''_1 U^{\prime 2}_2  } \nonumber\\
    F_{3,\rho} &= \frac{\alpha(1-\gamma)}{2} \mu \insquare{  -U'''_1 U'_2 + U^{\prime\prime 2}_1 - U''_1 U'_1 U'_2 - U'''_1 U'_2 - U''_1 U'_2 U'_1 + 2U''_1 U'_2 U'_3 }
    \label{eq:ST_F1rho_to_F3q}
\end{align}
All the $F_{v,\rho}, F_{v,q}$, $v = 1,2$ are of the form $\frac{\alpha(1-\gamma)}{2} \mu\insquare{A} $, where $A$ is some polynomial in $u'$, $u''$, $u'''$ of bounded degree. From \eqref{eq:ST_J_Fy_firstForm}, we have
\begin{align}
    \det\inparen{ J_{\tilde{F}, \vec{y}}  } = \det\inparen{I_3}\det\inparen{I_2 - CB},
    \label{eq:ST_det_J_Fy_blockMatrixForm}
\end{align}
and it suffices to give a condition so that the $2\times 2$ matrix $I_2 - CB$ has nonzero determinant. We have, after some simplifications using \eqref{eq:ST_F4r_to_F5sigma}, that
\begin{align*}
    \det\inparen{I_2 - CB} &= 1 - E,
\end{align*}
where
\begin{align*}
    E &:= \frac{1}{R}\inparen{ -2q F_{1,\rho} + (\rho + q) F_{2,\rho} + (\rho + q) F_{3,\rho} + (\rho - 3q) F_{1,q} + 2q F_{2,q} + 2q F_{3,q} } \\
    &\quad\quad\quad - \frac{1}{R^3} \inparen{ F_{1,\rho} F_{2,q} + F_{1,\rho} F_{3,q} + F_{2,\rho} F_{1,q} + F_{3,\rho}F_{1,q}  }.
\end{align*}
Using the bounds $\rho, q \leq (\alpha D^2 + h^2 + 1)/R$, and using \eqref{eq:ST_F1rho_to_F3q}, we see $E$ can be bounded as 
\begin{align*}
    \abs{E} \leq \frac{(\rho + 3q)F_{1,q}^2}{R} \leq \frac{6\abs{\rho + 3q} \abs{F_{1,q}}}{R} + \frac{1}{R^3} F_{1,q}^2 \leq \frac{1}{R^2} 585 \alpha^2 D^8 (1+h^2) \leq \frac{150 \alpha^2 D^8 (1+h^2) }{\kappa^2}.
\end{align*}
Then whenever $\kappa$ satisfies assumption \eqref{eq:ST_asmpt_on_kappa_FOR_JACOBIAN}, we have that $\abs{E} < 1$. This guarantees that $\det(I_2 - CB) \neq 0$ so that from \eqref{eq:ST_det_J_Fy_blockMatrixForm}, the matrix $J_{\tilde{F},\vec{y}}$ is invertible.
\end{proofof}

\begin{proofof}{Proposition \ref{proposition:ST_+_to_originalSystem}}
The approach is similar to that of Theorem \ref{thm:perceptron_+_to_originalSystem}. With $\theta^{+,\ell} := \sqrt{q^+}z + \sqrt{\rho^+ - q^+}\xi^\ell$, set for $0 \leq v \leq 1$,
\begin{align*}
    S_v^{\ell} &:= \sqrt{v}S_{M+1}^{\ell} + \sqrt{1-v} \theta^{+,\ell}.
\end{align*}
Define
\begin{align*}
    \nu^{+}_{t\textnormal{-lim},v}\insquare{f} &:= \Eb   \frac{ \Eb_\xi \inangle{ f \exp \inparen{\sum_{\ell \leq n} u(S_{v}^{\ell}) } }_{t\textnormal{-lim}} }{ \inparen{ \Eb_\xi \inangle{ \exp u(S_{v}^1) }_{t\textnormal{-lim}} }^n }.   
\end{align*}
One notes that for every $t$, $\nu^{+}_{t\textnormal{-lim},1}\insquare{f} = \nu^{+}_{t\textnormal{-lim}}\insquare{f}$, and that $\nu^{+}_{t\textnormal{-lim},0}\insquare{f} = \nu_{t\textnormal{-lim}}\insquare{f}$. The latter follows because $f$ is independent of any $\theta^{+,\ell}$'s. Therefore, it suffices to show that for any fixed $t$, $\ud \nu^{+}_{t\textnormal{-lim},v}\insquare{f}/\ud v$ is bounded by $K/\sqrt{N}$ (in particular we only need it for $t = 0$ or $1$).

However, this is exactly a $v$-interpolation, and we can extend Lemma \ref{lemma:ST_v_interpolation} to this $+$-system, so that $\abs{ \ud \nu^{+}_{t\textnormal{-lim},v}\insquare{f}/\ud v  }$ is bounded by a constant $K$ times a sum of terms of type
\begin{align*}
    \nu^+_{t\textnormal{-lim},v}\abs{ R_{\ell,\ell }- \rho^+  }, \quad \textnormal{or } \quad  \nu^+_{t\textnormal{-lim},v}\abs{ R_{\ell,\ell' }- q^+ },
\end{align*}
or a lower order term $1/N$. Extend \ref{lemma:ST_getRidOf_v} to this setting, which gets rid of the dependence on $v$. Extend Lemma \ref{lemma:ST_nutlim_TSOC} as well and we have that $\inparen{ \nu^+_{t\textnormal{-lim}}\insquare{R_{1,1} - \rho^+} }^{1/2}$ and $\inparen{ \nu^+_{t\textnormal{-lim}}\insquare{R_{1,2} - q^+} }^{1/2}$ are both bounded by $K/\sqrt{N}$. It follows that $\abs{ \ud \nu^{+}_{t\textnormal{-lim},v}\insquare{f}/\ud v  }$ is bounded by $K/\sqrt{N}$. This holds uniformly over $t$ and finishes the proof.
\end{proofof}

\addcontentsline{toc}{section}{References}
\bibliographystyle{alpha}
\bibliography{mybib.bib}

\end{document}